%% file: main.tex
\documentclass[a4paper]{article}
\usepackage[T1]{fontenc}
\usepackage{lmodern}
\usepackage{amsmath,amssymb}
\usepackage{amsthm}
\usepackage{mathrsfs}
\usepackage[all]{xy}
\usepackage{rotate}
\usepackage{enumerate}
\usepackage{mathtools}
\usepackage[dvipdfmx]{graphics}
\usepackage{color}

\usepackage{here}

\usepackage{atbegshi}

\input{newcommands.tex}

\begin{document}

\title{Family of $\mathscr{D}$-modules and representations with a boundedness property}

\author{Masatoshi Kitagawa}

\date{}

\maketitle

\input{introduction.tex}

\input{preliminary.tex}

\input{d_module.tex}

\input{uniformly_bounded_family.tex}

\input{examples.tex}

\input{zuckerman.tex}

\input{applications.tex}


\input{ref.bbl}

\end{document}

%% file: newcommands.tex
\newcommand{\Ad}{\textup{Ad}}
\newcommand{\ad}{\textup{ad}}

\renewcommand{\Im}{\textup{Im}}

\newcommand{\id}{\textup{id}}
\newcommand{\tr}{\textup{tr}}

\newcommand{\Aut}{\textup{Aut}}

\renewcommand{\sl}{\mathfrak{sl}}

\newcommand{\Sp}{\textup{Sp}}

\newcommand{\so}{\mathfrak{so}}

\newcommand{\End}{\textup{End}}
\newcommand{\RR}{\mathbb{R}}

\newcommand{\CC}{\mathbb{C}}
\newcommand{\ZZ}{\mathbb{Z}}
\newcommand{\NN}{\mathbb{N}}
\newcommand{\QQ}{\mathbb{Q}}

\newcommand{\DD}{\mathbb{D}}

\newcommand{\Ind}{\textup{Ind}}

\newcommand{\calF}{\mathcal{F}}
\newcommand{\calI}{\mathcal{I}}
\newcommand{\calO}{\mathcal{O}}
\newcommand{\calS}{\mathcal{S}}

\newcommand{\calB}{\mathcal{B}}
\newcommand{\calC}{\mathcal{C}}

\newcommand{\calP}{\mathcal{P}}
\newcommand{\calU}{\mathcal{U}}
\newcommand{\calL}{\mathcal{L}}

\newcommand{\calD}{\mathcal{D}}
\newcommand{\calT}{\mathcal{T}}

\newcommand{\calV}{\mathcal{V}}
\newcommand{\calZ}{\mathcal{Z}}
\newcommand{\calJ}{\mathcal{J}}
\newcommand{\calR}{\mathcal{R}}
\newcommand{\calM}{\mathcal{M}}
\newcommand{\calN}{\mathcal{N}}

\newcommand{\frakm}{\mathfrak{m}}

\newcommand{\pr}{\textup{pr}}

\newcommand{\Ann}{\textup{Ann}}

\newcommand{\gr}{\textup{gr}}
\newcommand{\Hom}{\textup{Hom}}

\newcommand{\Tor}{\mathrm{Tor}}

\theoremstyle{plain}
\newtheorem{theorem}{Theorem}[section]
\newtheorem*{theorem*}{Theorem}

\newtheorem{proposition}[theorem]{Proposition}
\newtheorem{lemma}[theorem]{Lemma}
\newtheorem{corollary}[theorem]{Corollary}

\newtheorem{fact}[theorem]{Fact}

\theoremstyle{definition}
\newtheorem{definition}[theorem]{Definition}

\newtheorem{remark}[theorem]{Remark}

\theoremstyle{remark}
\renewenvironment{pf}{\begin{proof}[\sc Proof]}{\end{proof}}

\numberwithin{equation}{subsection}

\newcommand{\rring}[1]{\calO(#1)}
\newcommand{\rsheaf}[1]{\calO_{#1}}
\newcommand{\define}[1]{\textit{#1}}

\newcommand{\calA}{\mathcal{A}}
\newcommand{\spn}[1]{\mathrm{span}_{\CC}\{#1\}}

\newcommand{\set}[1]{\left\{#1 \right\}}

{\begin{equation} 
\addtolength{\abovedisplayskip}{-1ex}
\addtolength{\abovedisplayshortskip}{-1ex}
\addtolength{\belowdisplayskip}{-1ex}
\addtolength{\belowdisplayshortskip}{-1ex}
\begin{minipage}[t]{0.87\linewidth}}
{\end{minipage} \end{equation} \ignorespacesafterend}

\newcommand{\univ}[1]{\calU(\mathfrak{#1})}
\newcommand{\univcent}[1]{\calZ(\mathfrak{#1})}

\newcommand{\Ext}{\mathrm{Ext}}

\newcommand{\Len}{\mathrm{Len}}

\newcommand{\Ker}{\mathrm{Ker}}

\newcommand{\proots}{{\Delta^+}}

\newcommand{\lie}[1]{{\mathfrak{#1}}}

\newcommand{\lieR}[1]{\mathfrak{#1_{\RR}}}

\newcommand{\Mod}{\mathrm{Mod}}

\newcommand{\sect}{\Gamma}

\newcommand{\invshf}{\vee}
\newcommand{\locZuck}[2]{\mathrm{L}^{#1}_{#2}}
\newcommand{\DlocZuck}[3]{\DD^{#3}\locZuck{#1}{#2}}
\newcommand{\zuck}[2]{\Gamma^{#1}_{#2}}
\newcommand{\Dzuck}[3]{\DD^{#3}\zuck{#1}{#2}}

\newcommand{\dual}[1]{{#1}^{\vee}}

\newcommand{\Dsheaf}{\mathscr{A}}
\newcommand{\DsheafB}{\mathscr{B}}
\newcommand{\DsheafC}{\mathscr{C}}
\newcommand{\ntDsheaf}{\mathscr{D}}
\newcommand{\Dalg}[1]{\mathcal{A}_{#1}}
\newcommand{\ntDalg}[1]{\mathcal{D}_{#1}}
\newcommand{\tensorDshf}{\mathbin{\#}}
\newcommand{\tensorBorn}{\mathbin{\#}}

\newcommand{\EP}{\mathrm{EP}}

\newcommand{\opalg}{\mathrm{op}}

\newcommand{\CE}[3]{\mathrm{CE}_{#3}({#1}, {#2})}

\newcommand{\cxdot}{\bullet}
\newcommand{\cpx}[1]{#1^{\cxdot}}

%% file: introduction.tex
\begin{abstract}
	In the representation theory of real reductive Lie groups, many objects have finiteness properties.
	For example, the lengths of Verma modules and principal series representations are finite, and more precisely, they are bounded.
	In this paper, we introduce a notion of uniformly bounded families of holonomic $\ntDsheaf$-modules to explain and find such boundedness properties.

	A uniform bounded family has good properties. For instance, the lengths of modules in the family are bounded and the uniform boundedness is preserved by direct images and inverse images.
	By the Beilinson--Bernstein correspondence, we can deduce several boundedness results about the representation theory of complex reductive Lie algebras from corresponding results of uniformly bounded families of $\ntDsheaf$-modules.
	In this paper, we concentrate on proving fundamental properties of uniformly bounded families, and preparing abstract results for applications to the branching problem and harmonic analysis.
\end{abstract}

\section{Introduction}

In this paper, we introduce a notion of uniformly bounded families of $\ntDsheaf$-modules, which are good families of holonomic $\ntDsheaf$-modules with bounded lengths.
We show that the uniform boundedness is preserved by fundamental operations of $\ntDsheaf$-modules such as direct images, inverse images and taking subquotients.
By the Beilinson--Bernstein correspondence \cite{BeBe81}, we can deduce several boundedness results about the representation theory of complex reductive Lie algebras from corresponding results of uniformly bounded families of $\ntDsheaf$-modules.

In the representation theory of real reductive Lie groups, finiteness results about lengths of modules and multiplicities in branching laws are fundamental and enable us to study Harish-Chandra modules and unitary representations.
We list typical examples of the results: finiteness of the lengths of Verma modules and principal series representations, Harish-Chandra's admissibility theorem \cite{Ha53_admissible}, irreducibility of $\univ{g}^K$-actions on $K$-isotypic components, and finiteness of multiplicities in the Plancherel formula of symmetric spaces \cite{Ba87,KoOs13}.

Our main concern is that the finiteness is uniform.
The length of a Verma module is bounded by some constant independent of its highest weight, and a similar result holds for principal series representations.
The former is an easy consequence of Soergel's theorem \cite{So90} (see also Remark \ref{rmk:Soergel}), and the latter is proved by Kobayashi--Oshima in \cite{KoOs13}.

In \cite{KoOs13}, T.\ Kobayashi and T.\ Oshima give criteria for the finiteness and the uniform boundedness of multiplicities in the branching problem and harmonic analysis of real reductive Lie groups.
The criteria are given by conditions on the existence of open orbits in flag varieties,
and proved by using hyperfunction boundary value maps.
A.\ Aizenbud, D.\ Gourevitch and A.\ Minchenko give an alternative proof to some of the results using families of holonomic $\ntDsheaf$-modules in \cite{AiGoDm16}.
T.\ Tauchi proves similar results based on the finiteness of hyperfunction solutions in \cite{Ta18}.
Their results are one of our motivations.

In this paper, we do not deal with concrete applications to the branching problem and harmonic analysis.
We concentrate on providing fundamental properties of uniformly bounded families, and preparing abstract results for such applications.
See Proposition \ref{prop:TorUniformlyBoundedGmod} and Remark \ref{rmk:TorUniformlyBoundedGmod} for an easy application to the estimate of multiplicities.

Let us state the definition of uniformly bounded families and their properties.
Our definition is based on Bernstein's work \cite{Be72}.
In the paper, he has introduced the multiplicity $m(M)$ of a module $M$ of the Weyl algebra $\ntDalg{\CC^n}$, and proved that the multiplicity is well-behaved for direct images, inverse images and taking subquotients.
We denote by $\Mod_h(\ntDsheaf_{X})$ the category of holonomic $\ntDsheaf$-modules on an smooth variety $X$.
Let $f\colon \CC^n \rightarrow \CC^m$ be a morphism of algebraic varieties of degree $d'$
and set $d = \max(1, d')$.
Let $Df_+$ (resp. $L f^*$) denote the direct (resp. inverse) image functor.
Then we have
\begin{align*}
	\sum_i m(D^if_+(\calM)) \leq d^{n+m} m(\calM)\text{ and } \sum_i m(L_i f^*(\calN)) \leq d^{n+m} m(\calN).
\end{align*}
for any $\calM \in \Mod_h(\ntDsheaf_{\CC^n})$ and $\calN \in \Mod_h(\ntDsheaf_{\CC^m})$ (see Fact \ref{fact:WeylAlgebraDirectInverseImage}).
Here we put $m(\calM):= m(\sect(\calM))$.
The key point is that the coefficient $d^{n+m}$ is independent of $\calM$ (or $\calN$).
In other words, the estimates of the multiplicities are uniform with respect to $\calM$ (or $\calN$).
This is the starting point of our definition.

Let $\Dsheaf_{X, \Lambda}:=(\Dsheaf_{X,\lambda})_{\lambda \in \Lambda}$
be a family of algebras of twisted differential operators on a smooth variety $X$ over $\CC$.
We say that $(U,\varphi, \Phi)$ is a trivialization of $\Dsheaf_{X,\Lambda}$
if $\varphi \colon U\rightarrow X$ is a surjective \'etale morphism and $\Phi_\lambda \colon \varphi^\#\Dsheaf_{X,\lambda} \rightarrow \ntDsheaf_{U}$ is an isomorphism.
Here $\varphi^\#$ is the pull back of algebras of twisted differential operators by $\varphi$.
Take a trivialization $(U, \varphi, \Phi)$ with affine $U$ and a closed embedding $\iota \colon U\rightarrow \CC^n$.
Then for a family $\calM \in \prod_{\lambda \in \Lambda}\Mod_h(\Dsheaf_{X,\lambda})$, we can consider a function
\begin{align}
	\Lambda \ni \lambda \mapsto m(\iota_+(\varphi^*(\calM_\lambda))) \in \NN.
	\label{eqn:MultiplicityFunction}
\end{align}
The boundedness of the function does not depend on $\iota$ (see Proposition \ref{prop:LocalMultiplicityEtale}), and does depend on the isomorphisms $\Phi$.

We introduce a relation $\sim$ of trivializations.
For two trivialization $(U,\varphi, \Phi)$ and $(V,\psi, \Psi)$, we write $(U,\varphi, \Phi) \sim (V,\psi, \Psi)$ if the set
\begin{align*}
	\set{\widetilde{\varphi}^\#\Psi_\lambda \circ (\widetilde{\psi}^\#\Phi_\lambda)^{-1}: \lambda \in \Lambda} \subset \Aut(\ntDsheaf_{U\times_X V}) \simeq \calZ(U\times_X V)
\end{align*}
spans a finite-dimensional subspace of the space $\calZ(U\times_X V)$ of closed $1$-forms.
Here $\widetilde{\varphi}\colon U\times_X V\rightarrow V$ and $\widetilde{\psi}\colon U\times_X V \rightarrow U$ are the projections of the fiber product.
See Definition \ref{def:BoundedTrivialization}.

We say that a trivialization $T$ is bounded if $T\sim T$.
Although the relation is not an equivalence relation of trivializations,
it is an equivalence relation of bounded trivializations.
Moreover, if two bounded trivialization $T=(U,\varphi, \Phi)$ and $S=(V, \psi, \Psi)$ with affine $V$ and $U$ are equivalent, the boundedness of the function \eqref{eqn:MultiplicityFunction} defined for $T$ is equivalent to that for $S$.
An equivalence class of bounded trivializations is called a bornology of $\Dsheaf_{X,\Lambda}$ (see Definition \ref{def:EquivalenceBornology}).

For a bornology $\calB$ of $\Dsheaf_{X,\lambda}$, we say that $\calM \in \prod_{\lambda\in \Lambda}\Mod_{h}(\Dsheaf_{X,\lambda})$ is a uniformly bounded family with respect to $\calB$
if the function \eqref{eqn:MultiplicityFunction} defined for any/some $T \in \calB$ is bounded.
We denote by $\Mod_{ub}(\Dsheaf_{X,\Lambda}, \calB)$ the full subcategory of $\prod_{\lambda\in \Lambda}\Mod_{h}(\Dsheaf_{X,\lambda})$ whose objects are uniformly bounded.
Similarly, we define a derived category version $D^b_{ub}(\Dsheaf_{X,\Lambda}, \calB)$.
See Definition \ref{def:UniformlyBoundedFamilyOri} for the details.

Corresponding to operations of $\Dsheaf_{X,\Lambda}$, we define operations of bornologies by natural ways: pull-back $f^\#\calB$, external tensor product $\calB\boxtimes \calB'$, twisting by an invertible sheaf $\calB^\calL$ and opposite $\calB^\opalg$.
The following theorem is a fundamental result about uniformly bounded families.
See Subsections \ref{sect:UniformlyBoundedFamily} and \ref{sect:twisting}.

\begin{theorem}\label{thm:IntroUniformlyBounded}
	The uniform boundedness is preserved by direct images, inverse images, external tensor products, twisting by an invertible sheaf and taking subquotients.	
\end{theorem}

For example, for a morphism $f\colon Y\rightarrow X$ of smooth varieties, we can define
a direct image functor
\begin{align*}
	Df_+ \colon D^b_{ub}(\Dsheaf_{X,\Lambda}, \calB) \rightarrow D^b_{ub}(f^\#\Dsheaf_{X,\Lambda}, f^\# \calB)
\end{align*}
via $Df_+(\calM) = (Df_+(\calM_\lambda))_{\lambda \in \Lambda}$,
which is the restriction of the direct product of the direct image functors on $D^b_h(\Dsheaf_{X,\lambda})$ ($\lambda \in \Lambda$).
Here $f^\#\Dsheaf_{X,\Lambda}$ is the family $(f^\# \Dsheaf_{X,\lambda})_{\lambda \in \Lambda}$.

The proofs for the last three operations in Theorem \ref{thm:IntroUniformlyBounded} are easy by the definition of uniformly bounded families.
The proofs for the others are essentially the same as a proof of preserving holonomicity (see e.g.\ \cite[VII. \S 12]{Bo87} and \cite[3.2]{HTT08}).

When each $\Dsheaf_{X,\lambda}$ is $G$-equivariant, we can define a notion of $G$-equivariant bornologies by a natural way.
The $G$-equivariance is preserved by the pull-back by a $G$-equivariant morphism.
It is important for the representation theory that if $X$ is a homogeneous variety $G/H$, there exists a unique $G$-equivariant bornology of a family of $G$-equivariant algebras of twisted differential operators (see Proposition \ref{prop:UniqueBornologyHomogeneous}).

In Beilinson--Bernstein's paper \cite{BeBe81}, they give a way to classify equivariant $\ntDsheaf$-modules.
Combining the classification and the notion of $G$-equivariant bornologies,
we obtain

\begin{theorem}[Theorem \ref{thm:UniformlyBoundedIrreducibles}]\label{thm:IntroIrreducibles}
	Let $\calB$ be a bornology of $\Dsheaf_{X,\Lambda}$.
	Suppose that $X$ is a $G$-variety of an affine algebraic group $G$,
	and $\calB$ and $\Dsheaf_{X,\Lambda}$ are $G$-equivariant.
	If $G$ has finitely many orbits in $X$, then any family of $(\Dsheaf_{X,\lambda}, G)$-modules with bounded lengths
	is uniformly bounded with respect to $\calB$.
\end{theorem}

In Section \ref{sect:ExampleUBF}, we give several methods to construct bornologies and uniformly bounded families from algebraic group actions.
In particular, we will see that there are many uniformly bounded families.

Let us state applications of uniformly bounded families to the representation theory.
Let $G$ be a connected reductive algebraic group over $\CC$ and $B$ a Borel subgroup of $G$.
By the Beilinson--Bernstein correspondence, any $\lie{g}$-module with an infinitesimal character is isomorphic to $\sect(\calM)$ for some twisted $\ntDsheaf$-module on $G/B$.
We always choose the twist of $\ntDsheaf$ such that $\sect$ is exact on the category of quasi-coherent twisted $\ntDsheaf$-modules.
We say that a family $(M_i)_{i \in I}$ of $\lie{g}$-modules is uniformly bounded if the lengths of $M_i$ are bounded and the localization of the family of all composition factors of all $M_i$ is a uniformly bounded family on $G/B$.

The uniform boundedness is preserved by several operations of $\lie{g}$-modules such as taking subquotients, tensoring with finite-dimensional modules and cohomological parabolic inductions.
This follows from corresponding results for $\ntDsheaf$-modules in Theorem \ref{thm:IntroUniformlyBounded}.
By Theorem \ref{thm:IntroIrreducibles}, any family of Harish-Chandra modules (or objects in the BGG category $\calO$) with bounded lengths is uniformly bounded.
This implies that many families in the representation theory of real reductive Lie groups are uniformly bounded.

We shall state the preservation result for cohomological parabolic inductions.
Let $P$ be a parabolic subgroup of $G$ and $L$ a Levi subgroup of $P$.
Let $(\lie{g}, K)$ be a pair (see Definition \ref{def:pair}).
Take a reductive subgroup $K_L \subset K$ such that $\lie{k}_L \subset \lie{k}\cap \lie{l}$ and $K_L$ normalizes $\lie{l}$ and $\lie{p}$.

\begin{theorem}[Theorem \ref{thm:uniformlyBoundedLengthCohInd}]
	Let $(M_i)_{i \in I}$ be a uniformly bounded family of $(\lie{l}, K_L)$-modules.
	Then the family $(\Dzuck{K}{K_L}{j}(\univ{g}\otimes_{\univ{p}} M_i))_{i \in I, j \in \ZZ}$ is uniformly bounded,
	where $\Dzuck{K}{K_L}{j}$ is the $j$-th Zuckerman derived functor.
\end{theorem}

As mentioned before, the lengths of Verma modules (or principal series representations) are bounded, which is a special case of the theorem.
It is well-known that the length of a cohomologically induced module is finite (see e.g.\ \cite[Theorem 0.46]{KnVo95_cohomological_induction}).

For the proof of the theorem, we need the localization of the Zuckerman derived functor.
In this paper, we construct the localization following F.\ Bien \cite{Bi90}.
A conceptual treatment of the localization using the equivariant derived category
is given by S.\ N.\ Kitchen \cite{Ki12}.
See also \cite{MiPa98}.
We do not treat the equivariant derived category in this paper.

An algebra of invariant differential operators plays an important role in the representation theory of real reductive Lie groups such as the Schur--Weyl duality and the compact Howe duality \cite{Ho89}, a characterization of compact Gelfand pair \cite{Th82}, and
Harish-Chandra's study of $(\lie{g}, K)$-modules \cite{Ha53_admissible,Ha54_subquotient_theorem}.
If $(\lie{g}, K)$ is a pair with connected reductive group $K$, then the $\univ{g}^K$-action on a non-zero $K$-isotypic component of an irreducible $(\lie{g}, K)$-module is irreducible.
This is a classical result that follows from the Jacobson density theorem and completely reducibility of the $K$-action (see e.g.\ \cite[Section 4.2]{GoWa09}).
The following theorem can be considered as a generalization of the result.

Let $G'$ be a reductive subgroup of $G$ and $(\lie{g'}, K')$ a subpair of $(\lie{g}, K)$.
Suppose that $K'$ is a reductive subgroup of $K$ and $\Ad_{\lie{g}}(K')$ is contained in $\Ad_{\lie{g}}(G')$.

\begin{theorem}[Theorem \ref{thm:uniformlyBoundedLengthGeneralGK}]\label{thm:intro:Zuckerman}
	Let $(V_i)_{i \in I}$ and $(V'_i)_{i \in I}$ be uniformly bounded families of $(\lie{g}, K)$-modules and $(\lie{g'}, K')$-modules, respectively (e.g.\ families of irreducible Harish-Chandra modules).
	Then there exists a constant $C$ such that for any $i \in I$ and $j \in \ZZ$, we have
	\begin{align*}
		\Len_{\univ{g}^{G'}}(H_j(\lie{g'}, K'; V_i \otimes V'_i)) \leq C,
	\end{align*}
	where $\Len(\cdot)$ means the length of a module.
\end{theorem}

One of our motivations of Theorem \ref{thm:intro:Zuckerman} is to study multiplicities in the branching problem and harmonic analysis of real reductive Lie groups.
The theorem asserts that the multiplicities are roughly controlled by $\univ{g}^{G'}$.
We can give criteria for the uniformly boundedness of the multiplicities by a ring invariant of $\univ{g}^{G'}$.
We postpone the results to the sequel \cite{Ki20}.

Another motivation of Theorem \ref{thm:intro:Zuckerman} is the Howe duality \cite{Ho89_transcending}.
If $V$ is the Segal--Weil--Shale representation and $(G', H')$ is a reductive dual pair of $G$ (i.e.\ $Z_G(G') = H', Z_{G}(H') = G'$), then Theorem \ref{thm:intro:Zuckerman} asserts that (higher) theta lifts $\Theta_i(V')=H_i(\lie{g'}, K'; V\otimes \dual{V'})$ are of finite length as $\lie{h'}$-modules and the lengths are bounded.
By our method, we can not prove one of important parts of Howe's theorem that the theta lift $\Theta_0(V')$ has a unique irreducible quotient.
However our theorem enables us to define the Euler--Poincar\'e characteristic of the higher theta lifts.
See Theorem \ref{thm:ThetaLift}.
We remark that the Euler–Poincar\'e characteristic of the theta lifting
for p-adic groups is studied in \cite{APS17}.

Let $G$ be a connected reductive algebraic group over $\CC$ and $K$ its connected reductive subgroup.
I. Penkov and G. Zuckerman call a $\lie{g}$-module $M$ a generalized Harish-Chandra module if $M$ is locally finite, completely reducible and admissible as a $\lie{k}$-module (see e.g.\ \cite{PeZu14}).
A relation between generalized Harish-Chandra modules and supports of $\ntDsheaf$-modules on $G/B$ is studied by A. V. Petukhov \cite{Pe12}.

A motivation of our study of the category of generalized Harish-Chandra modules is to study the category of $\univ{g}^K$-modules.
By Lepowsky--McCollum's result \cite{LeMc73}, the category of $\univ{g}^K$-modules can be embedded in the category of $(\lie{g}\oplus \lie{k}, \Delta(K))$-modules.
Hence we can relate the branching problem and harmonic analysis to the study of $(\lie{g}\oplus \lie{k}, \Delta(K))$-modules.
As an application of uniformly bounded families, we prove fundamental results of the category of generalized Harish-Chandra modules: finiteness of equivalence classes of irreducible objects (Corollary \ref{cor:NumberOfIrreducibles}), boundedness of the Loewy lengths of modules (Theorem \ref{thm:LoewyLength}), and 
existence of projective objects (Proposition \ref{prop:EnoughProj}).

This paper is organized as follows.
In Section 2, we review the notions of generalized pairs, pairs $(\lie{g}, K)$, relative Lie algebra cohomologies and truncation functors.
In Section 3, we recall the definition of algebras of twisted differential operators and their operations.
At the end of the section, we study the direct image functors with respect to the projections of principal $G$-bundles.
The definition of uniformly bounded families of $\ntDsheaf$-modules is in Section 4.
Theorem \ref{thm:IntroUniformlyBounded} is proved here.
Section 5 is devoted to constructions of bornologies and uniformly bounded families.
The proof of Theorem \ref{thm:IntroIrreducibles} is given here.
In Section 6, we review the localization of the Zuckerman derived functor following \cite{Bi90}.
Applications to the representation theory are given in Section 7.

\subsection*{Notation and convention}

In this paper, any algebraic variety is assumed to be quasi-projective and defined over $\CC$.
Let $\rsheaf{X}$ and $\rring{X}$ denote the structure sheaf of a variety $X$ and the algebra of its global sections, respectively.
Suppose that $X$ is smooth.
We write $\ntDsheaf_X$ for the algebra of non-twisted (local) differential operators.
We express algebras of twisted differential operators and the spaces of their global sections by script letters and calligraphic letters, respectively.
For example, the space of global sections of algebras $\mathscr{A}_X, \mathscr{B}_{X,\lambda}$ and $\mathscr{D}_X$ are denoted as $\calA_X, \calB_{X,\lambda}$ and $\calD_X$, respectively.

Any representation and module in this paper are assumed to be defined over $\CC$.
We express affine algebraic groups and their Lie algebras by Roman alphabets and corresponding German letters.
For example, the Lie algebras of affine algebraic groups $G, K$ and $H$ are denoted as $\lie{g}, \lie{k}$ and $\lie{h}$.
For a complex Lie algebra $\lie{g}$, we write $\univ{g}$ and $\univcent{g}$ for the universal enveloping algebra and its center, respectively.
For an affine algebraic group $G$, let $G_0$ denote the identity component of $G$.
For a $G$-module (resp.\ a $\lie{g}$-module) $V$, we write $V^G$ (resp.\ $V^{\lie{g}}$) for the space of all invariant vectors in $V$.

We denote by $\Mod(\Dsheaf_X)$, $\Mod(\calA, G)$, $\Mod(\lie{g}, K)$ and $\Mod(\lie{g})$ the categories of (left) modules of a sheaf $\Dsheaf_X$ of algebras, a generalized pair $(\calA, G)$, a pair $(\lie{g}, K)$ and a Lie algebra $\lie{g}$, respectively.
We write $\Len_\calR(V)$ for the length of a $\calR$-module $V$ in each category, e.g.\ $\Len_{\Dsheaf_X}(V)$, $\Len_{\calA, G}(V)$, $\Len_{\lie{g}, K}(V)$ and $\Len_{\lie{g}}(V)$.
We denote by $\Mod_{qc}(\Dsheaf_X)$ (resp.\ $\Mod_{h}(\Dsheaf_X)$) the category of quasi-coherent modules (resp.\ holonomic modules) of an algebra $\Dsheaf_X$ of twisted differential operators.
We use the same notation for categories of equivariant modules such as $\Mod_{qc}(\Dsheaf_X, G)$ and $\Mod_h(\Dsheaf_X, G)$.

For an algebra $\Dsheaf_X$ of twisted differential operators on a smooth variety $X$, let $D^b_{qc}(\Dsheaf_X)$ (resp.\ $D^b_{h}(\Dsheaf_X)$) denote the full subcategory of the derived category $D(\Mod(\Dsheaf_X))$ consisting of objects $\cpx{\calM}$ whose cohomologies $H^i(\cpx{\calM})$ are quasi-coherent (resp.\ holonomic)
and vanish for any $|i| \gg 0$.
We list operations of sheaves:
\begin{itemize}
	\item $\calL^\invshf$: the dual of an invertible sheaf $\calL$
	\item $\sect$, $R\sect$: the global section functor and its right derived functor of sheaves
	\item $f^{-1}$: the inverse image functor of sheaves
	\item $f_*, Rf_*$: the direct image functor and its right derived functor of sheaves
	\item $f^*, Lf^*$: the inverse image functor and its left derived functor of $\rsheaf{Y}$-modules (or twisted $\ntDsheaf$-modules)
	\item $Df_+$: the direct image functor of twisted $\ntDsheaf$-modules.
\end{itemize}
Here $f\colon X\rightarrow Y$ is a morphism of smooth varieties.
We denote by $R^i f_*$, $L_{-i}f^*$ and $D^if_+$ the compositions $H^i\circ Rf_*$, $H^i\circ Lf^*$ and $H^i\circ D f_+$, respectively.

Let $(\cdot)\otimes (\cdot)$ (without subscript) denote the tensor product over $\CC$.
For a $\calR$-module $M$ and an $\calS$-module $N$, we write $M\boxtimes N$ for the external tensor product of $M$ and $N$.

\subsection*{Acknowledgments}

I would like to thank T. Tauchi, Y. Ito and T. Kubo for reading the manuscript very carefully and for many corrections.
The author was partially supported by Waseda University Grants for Special Research Projects (No. 2019C-528).

%% file: preliminary.tex
\section{Preliminary}

In this section, we prepare several known results and definitions.
We deal with generalized pairs, Lie algebra cohomology groups and truncation functors.

\subsection{Generalized pair}

In this subsection, we recall the definitions of generalized pairs and $(\calA, G)$-modules,
and show easy propositions related to generalized pairs.
We refer the reader to \cite[p.96]{KnVo95_cohomological_induction}.

In this paper, any algebraic group is affine and defined over $\CC$,
and any $\CC$-algebra is associative and unital without Lie algebras.
For a representation $V$ of an affine algebraic group $G$ as an abstract group,
we say that $V$ is a $G$-module or $G$ acts rationally on $V$ if
the $G$-action is locally finite and any finite-dimensional $G$-subrepresentation of $V$ is a representation of an algebraic group.

\begin{definition} \label{def:GeneralizedPair}
Let $\calA$ be a $\CC$-algebra and $G$ an affine algebraic group over $\CC$ acting rationally on $\calA$ by algebra automorphisms.
We say that the pair $(\calA, G)$ equipped with a $G$-equivariant algebra homomorphism $\iota\colon\univ{g}\rightarrow \calA$ is a generalized pair if
the adjoint action of $\lie{g}$ on $\calA$ determined by $\iota$ coincides with the differential of the action of $G$ on $\calA$.
\end{definition}

For a generalized pair $(\calA, G)$, we denote by $\Ad_{\calA}$ (or $\Ad$) the action of $G$ on $\calA$.
For example, if $G'$ is a closed subgroup of $G$, then $(\univ{g}, G')$ is a generalized pair.
For a $G$-equivariant algebra $\Dsheaf_X$ of twisted differential operators on a smooth variety $X$,
$(\sect(\Dsheaf_{X}), G)$ forms a generalized pair.

\begin{definition}\label{def:AG-mod}
Let $(\calA, G)$ be a generalized pair.
We say that an $\calA$-module $V$ equipped with a rational $G$-action is an $(\calA, G)$-module if the following two conditions hold.
\begin{enumerate}[(i)]
	\item The differential of the $G$-action on $V$ coincides with the $\lie{g}$-action
	via the composition $\lie{g}\xrightarrow{\iota} \calA \rightarrow \End_{\CC}(V)$ and
	\item $gXv = \Ad(g)(X)gv$ holds for any $g \in G, X\in \calA$ and $v \in V$.
\end{enumerate}
We denote by $\Mod(\calA, G)$ the category of $(\calA, G)$-modules.
\end{definition}

If $G$ is reductive, any $(\calA, G)$-module is completely reducible as a $G$-module.
Hence the functor $\Mod(\calA, G) \ni V\mapsto V^G \in \Mod(\calA^G)$
is exact, where $V^G$ is the space of all $G$-invariant vectors in $V$.
Moreover, we can see that the functor sends an irreducible object to zero or irreducible one.
See e.g.\ \cite[Theorem 4.2.1]{GoWa09}.
Hence we have the following proposition.

\begin{proposition}\label{prop:AGandLength}
Let $(\calA, G)$ be a generalized pair with reductive $G$.
Then for any $(\calA, G)$-module $V$, we have
\begin{align*}
\Len_{\calA^G}(V^G) \leq \Len_{\calA, G}(V),
\end{align*}
where $\Len$ means the length of a module.
\end{proposition}

We will reduce some propositions about $(\calA, G)$-modules to those for $(\calA, G_0)$-module.
To do this, we need the following easy lemma.

\begin{lemma}\label{lem:GeneralizedPairConnected}
Let $(\calA, G)$ be a generalized pair and $V$ an $(\calA, G)$-module.
Then we have
\begin{align*}
\Len_{\calA, G}(V)\leq \Len_{\calA, G_0}(V) \leq |G/G_0|\Len_{\calA, G}(V).
\end{align*}
\end{lemma}

\begin{proof}
The first inequality is trivial.
It is enough to show the second inequality when $V$ is an irreducible $(\calA, G)$-module.

By Zorn's lemma, we can take a proper $(\calA, G_0)$-submodule $W$ such that any non-zero $(\calA, G_0)$-submodule of $V/W$ contains a unique irreducible $(\calA, G_0)$-submodule.
Take a maximal subset $S\subset G/G_0$ such that $W_S:=\bigcap_{g \in S} gW$ is non-zero.
Since $V$ is irreducible as an $(\calA, G)$-module, $S$ is a proper subset of $G/G_0$.
Fix $g \in G/G_0-S$.
Since $W_S\cap gW=0$, the composition $W_S\hookrightarrow V\twoheadrightarrow V/gW$ is injective.
Hence $W_S$ contains an irreducible $(\calA, G_0)$-submodule $V_0$.

Since $V$ is an irreducible $(\calA, G)$-module, we have $V = \sum_{g \in G/G_0} g V_0$.
This implies that $V$ is completely reducible as an $(\calA, G_0)$-module, and hence the length as an $(\calA, G_0)$-module is less than or equal to $|G/G_0|$.
\end{proof}

\subsection{\texorpdfstring{$(\lie{g}, K)$}{(g, K)}-module}\label{sect:CEcomplex}

We review the notion of pairs $(\lie{g}, K)$ and the relative Lie algebra cohomology groups.
We refer the reader to \cite[Chapters I, IV]{KnVo95_cohomological_induction}
and \cite[Chapter I]{BoWa00_continuous_cohomology}.

\begin{definition}\label{def:pair}
Let $\lie{g}$ be a complex Lie algebra and $K$ an affine algebraic group with Lie algebra $\lie{k} \subset \lie{g}$.
We say that $(\lie{g}, K)$ is a \define{pair} if the following two conditions hold.
\begin{enumerate}[(i)]
	\item A rational $K$-action on $\lie{g}$ by Lie algebra automorphisms is given, whose restriction to $\lie{k}$ is equal to the adjoint action of $K$ on $\lie{k}$.
	\item The differential of the $K$-action on $\lie{g}$ coincides with the adjoint action of $\lie{k}$ on $\lie{g}$.
\end{enumerate}
We denote by $\Ad_{\lie{g}}$ (or simply $\Ad$) the action of $K$ on $\lie{g}$.
\end{definition}

If $(\lie{g}, K)$ is a pair, then $(\univ{g}, K)$ forms a generalized pair.
A $(\univ{g}, K)$-module is called a \define{$(\lie{g}, K)$-module}.

For a complex Lie algebra $\lie{g}$, the functor $\Tor^{\univ{g}}_i(\cdot, \cdot)$ can be computed by an explicit complex called the Chevalley--Eilenberg chain complex.
We recall the complex in the relative setting.
Let $(\lie{g}, K)$ be a pair.
Remark that the following results also hold if $K$ is replaced with its Lie algebra $\lie{k}$.
The relative Chevalley--Eilenberg chain complex is a sequence
\begin{align*}
	\cdots \xrightarrow{\partial_{k+1}} \CE{\lie{g}}{K}{k} \xrightarrow{\partial_k} \CE{\lie{g}}{K}{k-1} \xrightarrow{\partial_{k-1}} \cdots \xrightarrow{\partial_1} \CE{\lie{g}}{K}{0} \rightarrow 0,
\end{align*}
where $\CE{\lie{g}}{K}{k}:=\univ{g}\otimes_{\univ{k}}\wedge^k (\lie{g}/\lie{k})$.
The differential $\partial_k$ is given by
\begin{align*}
	&\partial_k(v\otimes X_1 \wedge \cdots \wedge X_k)
	\\= &\sum_i (-1)^{i+1} v\widetilde{X}_i \otimes X_1\wedge \cdots \wedge \hat{X_i} \wedge \cdots \wedge X_k \\
	+ &\sum_{i<j} (-1)^{i+j} v\otimes ([\widetilde{X}_i, \widetilde{X}_j]+\lie{k})\wedge X_1\wedge \cdots \wedge \hat{X_i} \wedge \cdots \wedge \hat{X_j} \wedge \cdots \wedge X_k,
\end{align*}
where each $X_i$ is in $\lie{g}/\lie{k}$ and each $\widetilde{X}_i$ is a representative of $X_i$ in $\lie{g}$.
For a $(\lie{g}, K)$-module $V$, we set
\begin{align*}
	H_i(\lie{g}, K; V)&:= H_i((V\otimes_{\univ{g}}\CE{\lie{g}}{K}{\cxdot})^K) \simeq H_i((V\otimes_{\univ{k}}\wedge^\cxdot (\lie{g}/\lie{k}))^K),\\
	H^i(\lie{g}, K; V)&:= H^i(\Hom_{\lie{g}, K}(\CE{\lie{g}}{K}{\cxdot}, V)) \simeq H^i(\Hom_K(\wedge^\cxdot(\lie{g}/\lie{k}), V)).
\end{align*}
We call them the relative Lie algebra homology and cohomology of $V$, respectively.

If $K$ is reductive, the complex $(\CE{\lie{g}}{K}{\cxdot}, \partial_\cxdot)$ is a projective resolution of $\CC$ in $\Mod(\lie{g}, K)$.
Hence we can compute $\Tor$ and $\Ext$ by the complex.
See \cite[Prposotion 2.117]{KnVo95_cohomological_induction} and \cite[Lemma 3.1.9]{Ku02}.

\begin{fact}\label{fact:LiealgebraCohomology}
	Let $V$ and $W$ be $(\lie{g}, K)$-modules.
	If $K$ is reductive, then we have natural isomorphisms
	\begin{align*}
		H_i(\lie{g}, K; V\otimes W) \simeq \Tor^{\lie{g},K}_i(V, W), \\
		H^i(\lie{g}, K; \Hom_{\CC}(V, W)) \simeq \Ext_{\lie{g}, K}^i(V,W).
	\end{align*}
\end{fact}

The following result is called the Poincar\'e duality.
See \cite[Corollary 3.6]{KnVo95_cohomological_induction}.

\begin{fact}\label{fact:PoincareDuality}
	Let $V$ be a $(\lie{g}, K)$-module.
	Put $n = \dim_{\CC}(\lie{g}/\lie{k})$.
	If $K$ is reductive, then we have a natural isomorphism
	\begin{align*}
		H^i(\lie{g}, K; V\otimes \wedge^n(\lie{g}/\lie{k})) \simeq H_{n-i}(\lie{g}, K; V),
	\end{align*}
	where $\wedge^n(\lie{g}/\lie{k})$ is a $(\lie{g}, K)$-module with the natural $K$-action and the $\lie{g}$-action given by the character $X\mapsto \tr(\Ad_\lie{g}(X))$.
\end{fact}

\subsection{Truncation functor}\label{sect:truncation}

In many places, we reduce assertions about a complex to those about a single object.
To do so, we need the truncation functors.
We refer the reader to \cite[Definitions 11.3.11 and 12.3.1]{KaSc06}.

Let $\calA$ be an abelian category and $\calC(\calA)$ the category of complexes in $\calA$.
For a complex $(C^\cxdot, d^\cxdot) \in \calC(\calA)$, we set
\begin{align*}
	\tau^{\leq k} C^\cxdot &:= \cdots \rightarrow C^{k-2} \rightarrow C^{k-1} \rightarrow \Ker(d^k) \rightarrow 0 \rightarrow 0 \rightarrow \cdots \\
	\tau^{> k} C^\cxdot &:= \cdots \rightarrow 0 \rightarrow 0 \rightarrow \Im(d^k) \rightarrow C^{k+1} \rightarrow C^{k+2} \rightarrow \cdots \\
	\tau^{\leq k}_s C^\cxdot &:= \cdots \rightarrow C^{k-2} \rightarrow C^{k-1} \rightarrow C^k \rightarrow 0 \rightarrow 0 \rightarrow \cdots \\
	\tau^{> k}_s C^\cxdot &:= \cdots \rightarrow 0 \rightarrow 0 \rightarrow 0 \rightarrow C^{k+1} \rightarrow C^{k+2} \rightarrow \cdots.
\end{align*}
$\tau^{\leq k}$ and $\tau^{> k}$ are called \define{truncation functors},
and $\tau^{\leq k}_s$ and $\tau^{> k}_s$ are called \define{stupid truncation functors}.
Then we have distinguished triangles
\begin{align*}
	&\tau^{\leq k} C^\cxdot \rightarrow C^\cxdot \rightarrow \tau^{> k} C^\cxdot \xrightarrow{+1}, \\
	&\tau^{\leq k}_s C^\cxdot \rightarrow C^\cxdot \rightarrow \tau^{> k}_s C^\cxdot \xrightarrow{+1}.
\end{align*}

\begin{lemma}\label{lem:AdditiveFunctionDerived}
	Let $\calA$ and $\calB$ be abelian categories and $m$ a $\CC$-valued additive function on the Grothendieck group of $\calA$.
	Assume $m(M)\geq 0$ for any $M \in \calA$.
	\begin{enumerate}[(i)]
		\item For any distinguished triangle $(N^\cxdot \rightarrow M^\cxdot \rightarrow L^\cxdot \xrightarrow{+1})$ in $D^b(\calA)$ and $i \in \ZZ$, we have
		\begin{align*}
			m(H^i(M^\cxdot)) \leq m(H^i(N^\cxdot)) + m(H^i(L^\cxdot)).
		\end{align*}
		\item For any functor $F\colon D^b(\calB)\rightarrow D^b(\calA)$ of triangulated categories, complex $M^\cxdot \in D^b(\calB)$ and $i \in \ZZ$, we have
		\begin{align*}
			m(H^i(F(M^\cxdot))) \leq \sum_j m(H^{i-j}(F(H^j(M^\cxdot)))).
		\end{align*}
		\item For any functor $F\colon D^b(\calB)\rightarrow D^b(\calA)$ of triangulated categories, bounded complex $M^\cxdot \in C(\calB)$ and $i \in \ZZ$, we have
		\begin{align*}
			m(H^i(F(M^\cxdot))) \leq \sum_j m(H^{i-j}(F(M^j))).
		\end{align*}
	\end{enumerate}
\end{lemma}

\begin{proof}
	(i) is clear by the long exact sequence associated to the distinguished triangle $(N^\cxdot \rightarrow M^\cxdot \rightarrow L^\cxdot \xrightarrow{+1})$.
	Set $l(M^\cxdot):=|\set{n \in \ZZ: H^n(M^\cxdot)\neq 0}|$.
	By induction on $l(M^\cxdot)$ and the truncation functors, we can reduce the assertion (ii) to the case $M^\cxdot \simeq N[n]$ for some $N \in \calB$ and $n \in \ZZ$.
	In fact, we can take $k \in \ZZ$ such that $l(\tau^{\leq k} C^\cxdot), l(\tau^{> k} C^\cxdot) < l(C^\cxdot)$.
	Applying (i) to the following distinguished triangle iteratively, we obtain (ii):
	\begin{align*}
		F(\tau^{\leq k} M^\cxdot) \rightarrow F(M^\cxdot) \rightarrow F(\tau^{> k} M^\cxdot) \xrightarrow{+1}.
	\end{align*}
	Similarly, using the stupid truncation functors, we obtain (iii).
\end{proof}

%% file: d_module.tex
\section{\texorpdfstring{$\ntDsheaf$}{D}-module and its operations}

The purpose of this section is to summarize fundamental results about $\ntDsheaf$-modules.

\subsection{\texorpdfstring{Twisted $\ntDsheaf$}{D}-module}\label{subsect:DirectImage}

We review algebras of twisted differential operators and their operations.
We refer to \cite{BeBe93}, \cite[Section 1]{KaTa96}, \cite[Sections 3 and 4]{Ka08_dmodule} and \cite{Ka89} for ($G$-equivariant) algebras of twisted differential operators.
In this paper, we denote by $\ntDsheaf_X$ the algebra of non-twisted (local) differential operators on a smooth variety $X$.
Any variety in this paper is assumed to be quasi-projective.

Let $X$ be a (quasi-projective) smooth variety over $\CC$.
Let $i_X$ be the standard homomorphism $\rsheaf{X}\rightarrow \ntDsheaf_X$.
There are several definitions of algebras of twisted differential operators on $X$.
We adopt a definition in which the algebras are locally trivial in the \'etale topology.

\begin{definition}
	We say that a sheaf $\Dsheaf$ of algebras on $X$ equipped with a $\CC$-algebra homomorphism $i\colon \rsheaf{X}\rightarrow \Dsheaf$ is an \define{algebra of twisted differential operators}
	if $\Dsheaf$ is quasi-coherent as a left $\rsheaf{X}$-module and the pair $(\Dsheaf, i)$ is locally isomorphic to $(\ntDsheaf_X, i_X)$ in the \'etale topology (see \cite[A.1]{HMSW87} and \cite[1.1]{KaTa96}).
	We identify $\rsheaf{X}$ with $i(\rsheaf{X})$.
\end{definition}
Then $\Dsheaf$ admits a canonical filtration called the ordered filtration
such that its associated graded algebra is isomorphic to $p_* \rsheaf{T^*X}$, where $p$ is the natural projection from the cotangent bundle $T^*X$ to $X$.

Let $f\colon X\rightarrow Y$ be a morphism of smooth varieties
and $\Dsheaf_{Y}$ an algebra of twisted differential operators on $Y$.
We set
\begin{align*}
\Omega_{f}&:= f^{-1}\Omega_{Y}^{\invshf}\otimes_{f^{-1}\rsheaf{Y}}\Omega_{X},\\
\Dsheaf_{Y\leftarrow X} &:= f^{-1}\Dsheaf_{Y}\otimes_{f^{-1}\rsheaf{Y}}\Omega_{f},\\
\Dsheaf_{X\rightarrow Y} &:= \rsheaf{X} \otimes_{f^{-1}\rsheaf{Y}}f^{-1}\Dsheaf_Y,
\end{align*}
where $\Omega_{X}$ (resp.\ $\Omega_{Y}$) denotes the canonical sheaf of $X$ (resp.\ $Y$).
$f^{\#}\Dsheaf_{Y}$ denotes the sheaf of all differential endomorphisms
of the $\rsheaf{X}$-module $\Dsheaf_{X\rightarrow Y}$ that commute with the right $f^{-1}\Dsheaf_{Y}$-action.
Then $f^{\#}\Dsheaf_Y$ is an algebra of twisted differential operators on $X$
and $\Dsheaf_{Y\leftarrow X}$ is an $(f^{-1}\Dsheaf_{Y}, f^{\#}\Dsheaf_Y)$-bimodule.

The direct image of $\calM^{\cxdot} \in D^b_{qc}(f^{\#}\Dsheaf_Y)$ is defined by
\begin{align*}
D f_{+}(\calM^\cxdot) = Rf_*(\Dsheaf_{Y\leftarrow X}\otimes^L_{f^{\#}\Dsheaf_Y}\calM^{\cxdot}) \in D^b_{qc}(\Dsheaf_Y),
\end{align*}
and the inverse image of $\calN^{\cxdot} \in D^b_{qc}(\Dsheaf_Y)$ is defined by
\begin{align*}
	Lf^*(\calN^\cxdot) = \Dsheaf_{X\rightarrow Y}\otimes^L_{f^{-1}\Dsheaf_Y}f^{-1}(\calN^\cxdot)
	\in D^b_{qc}(f^\#\Dsheaf_Y).
\end{align*}
It is well-known that the functors $D f_{+}$ and $Lf^*$ are local for $Y$,
and preserves holonomicity.

\begin{proposition}\label{prop:DirectImageAffineMorphism}
	Suppose that $f\colon X\rightarrow Y$ is an affine morphism.
	Then $Df_+$ is isomorphic to the left derived functor of $D^0 f_+$.
\end{proposition}
	
\begin{proof}
	Let $\calM^{\cxdot} \in D^b_{qc}(f^{\#}\Dsheaf_Y)$.
	Then $\calM^\cxdot$ has a locally free resolution $\calF^\cxdot$ by \cite[Corollary 1.4.20]{HTT08}.
	$\Dsheaf_{Y\leftarrow X}\otimes_{f^{\#}\Dsheaf_Y} \calF^i$ locally admits a quasi-coherent $\rsheaf{X}$-module structure, and hence it is $f_*$-acyclic.
	Therefore we have
	\begin{align*}
		D f_{+}(\calM^\cxdot) &= f_*(\Dsheaf_{Y\leftarrow X}\otimes_{f^{\#}\Dsheaf_Y} \calF^\cxdot) \\
		&\simeq f_*(\Dsheaf_{Y\leftarrow X})\otimes_{f_*f^{\#}\Dsheaf_Y} f_*(\calF^\cxdot)\\
		&\simeq  f_*(\Dsheaf_{Y\leftarrow X})\otimes_{f_*f^{\#}\Dsheaf_Y}^L f_*(\calM^{\cxdot}).
	\end{align*}
	This shows the proposition.
\end{proof}

The following two facts are fundamental.
See \cite[Lemma 1.1.7, Propositions 1.2.3 and 1.2.6]{KaTa96},
and see \cite[Propositions 1.5.11 and 1.5.21, and Theorem 1.7.3]{HTT08} for the non-twisted case.

\begin{fact}\label{fact:FundamentalDmodule}
	Let $f\colon X\rightarrow Y$ and $g\colon Y\rightarrow Z$ be morphisms of smooth varieties and $\Dsheaf_Z$ an algebra of twisted differential operators on $Z$.
	Then we have $(g\circ f)^\# \Dsheaf_Z = f^\# g^\# \Dsheaf_Z$ and
	\begin{enumerate}[(i)]
		\item \label{eqn:DirectImage} $Dg_+\circ Df_+ = D(g\circ f)_+$,
		\item \label{eqn:InverseImage} $Lf^*\circ Lg^* = L(g\circ f)^*$.
	\end{enumerate}
\end{fact}

\begin{fact}[Base change theorem]\label{fact:BaseChange}
	Suppose that we have the following cartesian square of smooth varieties:
	\begin{align*}
		\xymatrix {
			Y\times_X Z \ar[r]^-{\widetilde{g}} \ar[d]^-{\widetilde{f}}& Y \ar[d]^-f\\
			Z \ar[r]^-g & X.
		}
	\end{align*}
	Let $\Dsheaf_X$ be an algebra of twisted differential operators.
	Then there exists an isomorphism $Lg^*\circ Df_+\simeq D\widetilde{f}_+\circ  L\widetilde{g}^*$ of functors.
\end{fact}

\subsection{Picard algebroid}\label{sect:PicardAlgebroid}

We review the notion of Picard algebroids and describe the action of $f^\# \Dsheaf_Y$ on $\Dsheaf_{Y\leftarrow X}$ using Picard algebroids.
We refer the reader to \cite[\S 2]{BeBe93}.

Let $Z$ be a smooth variety and $\calT_Z$ the tangent sheaf of $Z$.
\begin{definition}[{\cite[1.2 and 2.1.3]{BeBe93}}] \label{def:Picard}
	Let $\widetilde{\calT}$ be a quasi-coherent $\rsheaf{Z}$-module on $Z$.
	$\widetilde{\calT}$ is called a \define{Lie algebroid} on $Z$ if
	$\widetilde{\calT}$ is a sheaf of complex Lie algebras equipped with a $\rsheaf{Z}$-module homomorphism $\sigma\colon \widetilde{\calT} \rightarrow \calT_Z$ such that
	\begin{align*}
		[T, fT'] = (\sigma(T)f) T' + f[T, T']
	\end{align*}
	for any local sections $T, T' \in \widetilde{\calT}$ and $f \in \rsheaf{Z}$.

	A Lie algebroid $\widetilde{\calT}$ on $Z$ is called a \define{Picard algebroid} if
	$\sigma\colon\widetilde{\calT}\rightarrow \calT_Z$ is epimorphic and
	there is an isomorphism $i\colon\rsheaf{Z}\rightarrow \ker(\sigma)$ of $\rsheaf{Z}$-modules such that
	$[T, i(f)] = \sigma(T)f$ for any local sections $T \in \widetilde{\calT}$ and $f \in \rsheaf{Z}$.
\end{definition}

The isomorphism $i$ in the definition is unique.
We identify $\rsheaf{X}$ with $i(\rsheaf{X})$.

For an algebra $(\Dsheaf_Z, i)$ of twisted differential operators on a smooth variety $Z$,
we denote by $\calP(\Dsheaf_Z)$ the sheaf of sections in $\Dsheaf_Z$ with the order less than or equal to $1$.
Then $\calP(\Dsheaf_Z)$ is a Picard algebroid on $Z$ equipped with
the homomorphism $i\colon\rsheaf{Z}\rightarrow \calP(\Dsheaf_Z)$.
Since $\Dsheaf_Z$ is generated by $\calP(\Dsheaf_Z)$, to define an action of $\Dsheaf_Z$, it is enough to define an action of $\calP(\Dsheaf_Z)$ such that $i(1)$ acts by the identity morphism (\cite[Lemma 2.1.4]{BeBe93}).

We describe the action $f^\# \Dsheaf_Y$ on $\Dsheaf_{Y\leftarrow X}$.
Let $X, Y, f$ and $\Dsheaf_Y$ be as in the previous subsection.
We denote by $f^\# \calP(\Dsheaf_Y)$ the fiber product $f^* \calP(\Dsheaf_Y) \times_{f^*\calT_Y} \calT_X$ of
\begin{align*}
	f^* \calP(\Dsheaf_Y) \xrightarrow{f^*\sigma} f^* \calT_Y \leftarrow \calT_X.
\end{align*}
Then $f^\# \calP(\Dsheaf_Y)$ is a Picard algebroid on $X$ equipped with
$i\colon\rsheaf{X}\rightarrow f^\# \calP(\Dsheaf_Y)$ ($h \mapsto (h\otimes 1, 0)$).
We can define an action of $f^\# \calP(\Dsheaf_Y)$ on $\Dsheaf_{X\rightarrow Y}$ via
\begin{align*}
	(\sum_i f_i \otimes T_i, T') \cdot g \otimes S = T'g \otimes S + \sum_i f_i g \otimes T_i S
\end{align*}
for $(\sum_i f_i \otimes T_i, T') \in f^\# \calP(\Dsheaf_Y)$ and $g \otimes S \in \Dsheaf_{X\rightarrow Y}$.
This induces a canonical isomorphism $f^\# \calP(\Dsheaf_Y) \simeq \calP(f^\# \Dsheaf_Y)$ of Picard algebroids.
See \cite[Lemma 2.2]{BeBe93}.

\begin{proposition}\label{prop:ActionOnDXY}
	For local sections $(\sum_i f_i \otimes T_i, T') \in f^\# \calP(\Dsheaf_Y)$ and $S\otimes \tau \otimes \omega \in \Dsheaf_{Y\leftarrow X} = f^{-1}(\Dsheaf_Y \otimes_{\rsheaf{Y}} \Omega_Y^\invshf)\otimes_{\rsheaf{X}} \Omega_X$,
	we define
	\begin{align*}
		S\otimes \tau \otimes \omega \cdot (\sum_i f_i \otimes T_i, T')
		&= \sum_i ST_i \otimes \tau \otimes f_i\omega \\
		&-\sum_i S\otimes \sigma(T_i) \tau \otimes f_i\omega - S \otimes \tau \otimes \sigma(T') \omega,
	\end{align*}
	where $\sigma(T') \omega$ and $\sigma(T_i) \tau$ are defined by the Lie derivative.
	Then this induces a right action of $f^\# \Dsheaf_Y$ on $\Dsheaf_{Y\leftarrow X}$.
\end{proposition}

\begin{proof}
	A straightforward computation shows the proposition.
	Hence we omit the details.
\end{proof}

\begin{remark}
	Since $\Dsheaf_Y$ is locally trivial, we can reduce the computation to the 
	non-twisted case.
	In the case, the action coincides with that in \cite[Lemma 1.3.4]{HTT08}.
	In \cite[1.1.15]{KaTa96}, the action in the proposition is constructed by a formal computation of algebras of twisted differential operators.
\end{remark}

\subsection{\texorpdfstring{$G$}{G}-equivariant module}

In this subsection, we review the notion of $G$-equivariant $\ntDsheaf$-modules.
We refer the reader to \cite[1.8]{BeBe93}, \cite{Ka89} and \cite[Section 3]{Ka08_dmodule}.

Let $G$ be an affine algebraic group and $X$ a smooth $G$-variety.
We write $\mu\colon G\times X\rightarrow X$ for the multiplication map and 
$p_2\colon G\times X\rightarrow X$ for the projection onto the second factor.
An $\rsheaf{X}$-module $\calM$ is \define{$G$-equivariant} if an $\rsheaf{G\times X}$-module isomorphism
$\mu^* \calM \xrightarrow{\simeq} p_2^* \calM$ is specified and satisfies the associative law \cite[(3.1.2)]{Ka08_dmodule}.
The $G$-equivariant structure is sometimes called an (algebraic) $G$-action on $\calM$.
In fact, the $G$-equivariant structure induces a $G$-action on the set of sections of $\calM$.
We can define the differential of the $G$-action, that is, a Lie algebra homomorphism $\lie{g}\rightarrow \End_{\CC}(\calM)$.

We say that an algebra $\Dsheaf$ of twisted differential operators is \define{$G$-equivariant} if
an algebra homomorphism $i_{\lie{g}}\colon\univ{g}\rightarrow \Dsheaf$ and a $G$-action on $\Dsheaf$ are specified and satisfy
the following conditions:
\begin{enumerate}[(i)]
	\item The $G$-action is given by algebra isomorphisms.
	\item $i_{\lie{g}}$ is $G$-equivariant with respect to the adjoint action on $\univ{g}$.
	\item The differential of the $G$-action on $\Dsheaf$ coincides with the adjoint action of $\lie{g}$ on $\Dsheaf$
	coming from $i_{\lie{g}}$.
\end{enumerate}
The $G$-equivariant structure induces an isomorphism
\begin{align}
\mu^{\#}\Dsheaf \simeq p_2^{\#}\Dsheaf (\simeq \ntDsheaf_{G}\boxtimes \Dsheaf) \label{eqn:Gequivariant}
\end{align}
of algebras satisfying the associative law.
See \cite[Lemma 1.8.7]{BeBe93}.

Let $\Dsheaf_X$ be a $G$-equivariant algebra of twisted differential operators on $X$.
An $\Dsheaf$-module $\calM$ is called \define{$G$-equivariant} or an \define{$(\Dsheaf_X, G)$-module} if $\calM$ is $G$-equivariant as a $\rsheaf{X}$-module
and the morphism $\mu^* \calM \xrightarrow{\simeq} p_2^* \calM$ is a $\ntDsheaf_{G}\boxtimes \Dsheaf_X$-isomorphism.

Let $f\colon Y\rightarrow X$ be a $G$-morphism of smooth $G$-varieties.
Then the natural left action of $\univ{g}$ on $\Dsheaf_{Y\rightarrow X}$
induces an algebra homomorphism $\univ{g}\rightarrow f^{\#}\Dsheaf_X$.
Hence the algebra $f^{\#}\Dsheaf_X$ is $G$-equivariant.
The $G$-equivariant structure coincides with the one obtained from the canonical isomorphism
$\calP(f^{\#}\Dsheaf_X) \simeq f^* \calP(\Dsheaf_X)\times_{f^*\calT_X} \calT_Y$
of Picard algebroids.

In this setting, we can define the direct image functor and the inverse image functor
\begin{align*}
&H^i\circ Df_+\colon\Mod_{qc}(f^{\#}\Dsheaf_X, G) \rightarrow \Mod_{qc}(\Dsheaf_X, G), \\
&H^i \circ Lf^*\colon\Mod_{qc}(\Dsheaf_X, G) \rightarrow \Mod_{qc}(f^{\#}\Dsheaf_X, G).
\end{align*}
Although it is more conceptual to use the equivariant derived category,
we do not deal with it in this paper.

Let $\calA$ be a $G$-equivariant algebra of twisted differential operators on $X$.
We consider $G\times X$ as a $G\times G$-variety via the action $(a, b)\cdot (g, x) = (agb^{-1}, bx)$, and the codomain $X$ of $\mu$ (resp.\ $p_2$) as a $G\times G$-variety by letting the second (resp.\ first) factor of $G\times G$ act trivially.
Then $\mu$ and $p_2$ are $G\times G$-equivariant, and hence $\mu^\# \calA$ and $p_2^\# \calA$ are $G\times G$-equivariant algebras.

\begin{proposition}\label{prop:GGequivariant}
	The isomorphism \eqref{eqn:Gequivariant} is $G\times G$-equivariant.
\end{proposition}

\begin{proof}
	The assertion follows from the associative law of the $G$-equivariant structure on $\calA$ and easy diagram chasing.
	Hence we omit the details.
\end{proof}

\subsection{Principal bundle and direct image}\label{sect:PrincipalBundle}

Let $G$ be an affine algebraic group and $p\colon\widetilde{X}\rightarrow X$ a principal $G$-bundle over a smooth
variety $X$ together with a free right action $\widetilde{X}\times G \rightarrow \widetilde{X}$.
In this paper, a principal bundle over an algebraic variety is assumed to be locally trivial in the \'etale topology.
Then the projection $p$ is affine.
In this subsection, we study the direct image functor with respect to the projection $p$.

Let $\Dsheaf_{\widetilde{X}}$ be a $G$-equivariant algebra of twisted differential operators on $\widetilde{X}$
equipped with a $G$-equivariant algebra homomorphism $R\colon\univ{g} \rightarrow \Dsheaf_{\widetilde{X}}$.

\begin{proposition}\label{prop:DsheafLocal}
Assume that $p\colon\widetilde{X} \rightarrow X$ is a trivial bundle, i.e.\ $\widetilde{X}\simeq X\times G$.
Then there exists some algebra $\Dsheaf_X$ of twisted differential operators on $X$
such that
\begin{align*}
\Dsheaf_{\widetilde{X}} \simeq \Dsheaf_{X}\boxtimes \ntDsheaf_{G}
\end{align*}
under the identification $\widetilde{X}\simeq X\times G$.
\end{proposition}

\begin{proof}
Let $\mu\colon\widetilde{X}\times G \rightarrow \widetilde{X}$ be the multiplication map
and $p_1\colon\widetilde{X}\times G\rightarrow \widetilde{X}$ the projection.
We define $s\colon X\rightarrow \widetilde{X}=X\times G$ via $x\mapsto (x, e)$
and $t\colon\widetilde{X}\rightarrow \widetilde{X}\times G$ via $X\times G \ni (x, g) \mapsto (x, e, g) \in X\times G\times G$.
Then we have $\mu\circ t = \id_{\widetilde{X}}$ and a commutative diagram
\begin{align*}
\xymatrix{
	\widetilde{X} \ar[r]^-t \ar[d]^{p}& \widetilde{X}\times G \ar[r]^-{\mu}\ar[d]^{p_1}& \widetilde{X} \\
	X \ar[r]^s& \widetilde{X}.
}
\end{align*}
We regard $X$ and $\widetilde{X}$ at the bottom as $G$-varieties via the trivial actions
and $\widetilde{X}\times G$ via the right translation on $G$.
Then the morphisms are $G$-equivariant.

Since $\Dsheaf_{\widetilde{X}}$ is $G$-equivariant, we have
\begin{align*}
\mu^{\#}\Dsheaf_{\widetilde{X}} \simeq p_1^{\#}\Dsheaf_{\widetilde{X}}.
\end{align*}
We therefore obtain isomorphisms
\begin{align*}
\Dsheaf_{\widetilde{X}} = (\mu\circ t)^{\#}\Dsheaf_{\widetilde{X}} \simeq t^{\#}p_1^{\#}\Dsheaf_{\widetilde{X}} \simeq p^{\#}s^{\#}\Dsheaf_{\widetilde{X}}
\simeq s^{\#}\Dsheaf_{\widetilde{X}}\boxtimes \ntDsheaf_G
\end{align*}
of $G$-equivariant algebras of twisted differential operators.
\end{proof}

For $\lambda \in (\lie{g}^*)^G$, we set $I_{\lambda}:= \Ker(\lambda) \subset \univ{g}$ and 
\begin{align}
\Dsheaf_{X,\lambda} &:=(\CC_{\lambda-\delta} \otimes_{\univ{g}} p_*\Dsheaf_{\widetilde{X}})^G \nonumber\\
&\simeq (p_*\Dsheaf_{\widetilde{X}}/R(I_{-\lambda+\delta}) p_*\Dsheaf_{\widetilde{X}})^G \nonumber\\
&\simeq p_*\Dsheaf_{\widetilde{X}}^G/(R(I_{-\lambda+\delta}) p_*\Dsheaf_{\widetilde{X}})^G \nonumber\\
&\simeq (p_*\Dsheaf_{\widetilde{X}}/p_*\Dsheaf_{\widetilde{X}}R(I_{-\lambda}))^G, \label{eqn:DefinitionTwisted}
\end{align}
where $\delta \in (\lie{g}^*)^G$ is the character $Z \mapsto \tr(\ad_\lie{g}(Z))$.
Remark that
\begin{align*}
(R(I_{-\lambda+\delta})\rring{G})^G = (\rring{G}R(I_{-\lambda}))^G \subset \ntDalg{G}.
\end{align*}
Then $\Dsheaf_{X,\lambda}$ is a $G$-equivariant algebra of twisted differential operators on $X$
equipped with the homomorphism $\univ{g}\rightarrow \Dsheaf_{X,\lambda}$ ($Z\mapsto \lambda({}^tZ)$).
Here we consider $X$ as a $G$-variety via the trivial action.
Note that if $\lambda$ is a character of $G$ and $\mu\in (\lie{g}^*)^G$, $\Dsheaf_{X,\lambda + \mu}$ is isomorphic to $\calL_{\lambda}\otimes_{\rsheaf{X}}\Dsheaf_{X,\mu}\otimes_{\rsheaf{X}}\calL_{-\lambda}$,
where $\calL_{\lambda}$ is the invertible $\rsheaf{X}$-module corresponding to the line bundle $\widetilde{X}\times_{G}\CC_{\lambda}\rightarrow X$.

For a while, we fix $\lambda \in (\lie{g}^*)^G$.
We write $D p_{+,\lambda}\colon D^b_{qc}(p^{\#}\Dsheaf_{X, \lambda})\rightarrow D^b_{qc}(\Dsheaf_{X, \lambda})$ for the direct image functor in Subsection \ref{subsect:DirectImage}.
We define a right exact functor
\begin{align*}
p_{+,\lambda}(\calM):= \CC_{\lambda-\delta}\otimes_{\univ{g}} p_*(\calM) \in \Mod_{qc}(\Dsheaf_{X,\lambda})
\end{align*}
for $\calM \in \Mod_{qc}(\Dsheaf_{\widetilde{X}})$.
For a character $\lambda$ of $G$, $p_{+,\lambda}(\rsheaf{\widetilde{X}})$ is
isomorphic to $\calL_{\lambda}$ if $G$ is reductive.

%
%

To see $Dp_{+,\lambda} \simeq L p_{+,\lambda}$, we need the following lemma.

\begin{lemma}\label{lem:PrincipalBundleDsheaf}
	$p^{\#}\Dsheaf_{X,\lambda}$ is canonically isomorphic to $\Dsheaf_{\widetilde{X}}$
	as a $G$-equivariant algebra of twisted differential operators.
\end{lemma}

\begin{proof}
	Since $p\colon \widetilde{X}\rightarrow X$ is locally trivial, the multiplication map 
	induces an isomorphism
	\begin{align*}
		\Dsheaf_{\widetilde{X}\rightarrow X} = \rsheaf{\widetilde{X}}\otimes_{p^{-1}\rsheaf{X}} p^{-1}\Dsheaf_{X,\lambda}
		\simeq \Dsheaf_{\widetilde{X}} / \Dsheaf_{\widetilde{X}} R(I_{-\lambda})
	\end{align*}
	of sheaves.
	This is $G$-equivariant and an $(\rsheaf{\widetilde{X}}\otimes \univ{g}, f^{-1}\Dsheaf_{X})$-bimodule isomorphism.
	By the definition of $p^{\#}\Dsheaf_{X,\lambda}$, we obtain 
	an isomorphism $\Dsheaf_{\widetilde{X}} \simeq p^{\#}\Dsheaf_{X,\lambda}$
	of $G$-equivariant algebras of twisted differential operators.
\end{proof}

The isomorphism $\Dsheaf_{\widetilde{X}} \simeq p^{\#}\Dsheaf_{X,\lambda}$ can be obtained from 
the following cartesian square:
\begin{align*}
	\xymatrix{
		p^* \calP(\Dsheaf_{X,\lambda}) \ar[r]^-{p^*\sigma}& p^* \calT_X \\
		\calP(\Dsheaf_{\widetilde{X}}) \simeq p^*(p_*(\calP(\Dsheaf_{\widetilde{X}}))^G) \ar[r] \ar[u] & \calT_{\widetilde{X}} \simeq p^*(p_*(\calT_{\widetilde{X}})^G). \ar[u]
	}
\end{align*}
This implies $\calP(\Dsheaf_{\widetilde{X}}) \simeq p^\# \calP(\Dsheaf_{X, \lambda})$
(see above Proposition \ref{prop:ActionOnDXY}).
We identify $\Dsheaf_{\widetilde{X}}$ with $p^{\#}\Dsheaf_{X,\lambda}$ by the isomorphism.

We shall describe the right $\Dsheaf_{\widetilde{X}}$-action on $\Dsheaf_{X\leftarrow \widetilde{X}}$.
Let $\pi$ be the natural projection $\calT_{\widetilde{X}}^G \twoheadrightarrow p^{-1}(\calT_X)$.
Here $\calT_{\widetilde{X}}^G$ is the sheaf of $G$-invariant local sections of $\calT_{\widetilde{X}}$, that is, $p^{-1}(p_*(\calT_{\widetilde{X}})^G)$.
Fix a basis $X_1, X_2, \ldots, X_{\dim G} \in \lie{g}$.
Since $p\colon\widetilde{X}\rightarrow X$ is a principal $G$-bundle, 
$\Omega_p = p^{-1}(\Omega_X^\invshf)\otimes_{p^{-1}\rsheaf{X}} \Omega_{\widetilde{X}}$ is isomorphic to $\rsheaf{\widetilde{X}}$ as an $\rsheaf{\widetilde{X}}$-module.
The isomorphism is given by
\begin{align*}
	\theta_1\wedge \theta_2 \wedge \cdots \wedge \theta_{\dim X} \otimes \omega
	\mapsto \omega(\widetilde{\theta_1}, \widetilde{\theta_2}, \ldots, \widetilde{\theta_{\dim X}}, 
	R(X_1), R(X_2), \ldots, R(X_{\dim G}))
\end{align*}
for local sections $\theta_1, \theta_2, \ldots, \theta_{\dim X} \in p^{-1}(\calT_X)$ and $\omega \in \Omega_{\widetilde{X}}$,
where each $\widetilde{\theta_i}$ is a local section of $\calT_{\widetilde{X}}^G$ such that $\pi(\widetilde{\theta_i}) = \theta_i$.
Since $[\calT_{\widetilde{X}}^G, R(\lie{g})] = 0$, the isomorphism commutes with the actions of $\calT_{\widetilde{X}}^G$ on $\Omega_p$ and $\rsheaf{\widetilde{X}}$ defined by the Lie derivative.

\begin{lemma}\label{lem:PrincipalBundleDYX}
	$\Dsheaf_{X\leftarrow \widetilde{X}}$ is isomorphic to $\Dsheaf_{\widetilde{X}} / R(I_{-\lambda+\delta}) \Dsheaf_{\widetilde{X}}$
	as a $(p^{-1}\Dsheaf_{X,\lambda}, \Dsheaf_{\widetilde{X}})$-bimodule.
\end{lemma}

\begin{proof}
	Let $i\colon\Omega_p \rightarrow \rsheaf{\widetilde{X}}$ be the above isomorphism.
	Composing $i$ with the multiplication map, we obtain an isomorphism
	\begin{align*}
		\Dsheaf_{X\leftarrow \widetilde{X}} = p^{-1}\Dsheaf_{X,\lambda} \otimes_{p^{-1}\rsheaf{X}} \Omega_p
		\rightarrow p^{-1}\Dsheaf_{X,\lambda} \otimes_{p^{-1}\rsheaf{X}} \rsheaf{\widetilde{X}}
		\rightarrow \Dsheaf_{\widetilde{X}} / R(I_{-\lambda+\delta}) \Dsheaf_{\widetilde{X}}.
	\end{align*}
	of sheaves, and denote it by $\iota$.
	It is trivial that $\iota$ is a $(p^{-1}\Dsheaf_{X, \lambda}, \rsheaf{\widetilde{X}})$-module homomorphism.
	By Proposition \ref{prop:ActionOnDXY}, the action of $\theta \in \calP(\Dsheaf_{\widetilde{X}})^G$ is given by
	\begin{align*}
		(S\otimes \omega)\cdot \theta = S\theta \otimes \omega - S\otimes \sigma(\theta) \omega
	\end{align*}
	for $S\otimes \omega \in \Dsheaf_{X\leftarrow \widetilde{X}}$.
	Here $\sigma\colon\calP(\Dsheaf_{\widetilde{X}})^G\rightarrow \calT_{\widetilde{X}}^G$ is the restriction of the morphism $\sigma\colon\calP(\Dsheaf_{\widetilde{X}}) \rightarrow \calT_{\widetilde{X}}$ attached to the Picard algebroid.
	Hence $\iota$ commutes with the $\calP(\Dsheaf_{\widetilde{X}})^G$-action by the definition of $i$.
	Since $\calP(\Dsheaf_{\widetilde{X}})$ is generated by $\calP(\Dsheaf_{\widetilde{X}})^G$ as an $\rsheaf{\widetilde{X}}$-module,
	$\iota$ is a $(p^{-1}\Dsheaf_{X,\lambda}, \Dsheaf_{\widetilde{X}})$-bimodule isomorphism.
\end{proof}

\begin{remark}
	Suppose that $G$ is not unimodular.
	Although $\iota$ is a right $\lie{g}$-homomorphism, the isomorphism $p^{-1}\Dsheaf_{X,\lambda} \otimes_{p^{-1}\rsheaf{X}} \Omega_p
	\rightarrow p^{-1}\Dsheaf_{X,\lambda} \otimes_{p^{-1}\rsheaf{X}} \rsheaf{\widetilde{X}}$ is not.
	This is because $i\colon \Omega_p \rightarrow \rsheaf{\widetilde{X}}$ is not $G$-equivariant.
\end{remark}

We fix the isomorphism $\Dsheaf_{X\leftarrow \widetilde{X}}\simeq \Dsheaf_{\widetilde{X}} / R(I_{-\lambda+\delta}) \Dsheaf_{\widetilde{X}}$.

\begin{proposition}\label{prop:PrincipalBundleDirectImage}
$D p_{+,\lambda}$ is isomorphic to the left derived functor $L p_{+,\lambda}$ of $p_{+,\lambda}$.
\end{proposition}

\begin{proof}
	By Proposition \ref{prop:DirectImageAffineMorphism}, $Dp_{+,\lambda}$ is isomorphic to $p_*\Dsheaf_{X\leftarrow \widetilde{X}}\otimes^L_{p_*\Dsheaf_{\widetilde{X}}} p_*(\cdot)$.
	Hence it is enough to show that there is a natural isomorphism
	\begin{align*}
		p_*(\Dsheaf_{X\leftarrow \widetilde{X}})\otimes_{p_*(\Dsheaf_{\widetilde{X}})} p_*(\calF) \simeq \CC_{\lambda-\delta}\otimes_{\univ{g}} p_*(\calF)
	\end{align*}
	for any locally free $\Dsheaf_{\widetilde{X}}$-module $\calF$.

	Since $p$ is affine, we have $p_*(\Dsheaf_{X\leftarrow \widetilde{X}}) \simeq p_*(\Dsheaf_{\widetilde{X}})/R(I_{-\lambda+\delta})p_*(\Dsheaf_{\widetilde{X}})$ by Lemma \ref{lem:PrincipalBundleDYX}.
	This implies
	\begin{align*}
		p_*(\Dsheaf_{X\leftarrow \widetilde{X}})\otimes_{p_*(\Dsheaf_{\widetilde{X}})} p_*(\calF) &\simeq p_*(\calF)/R(I_{-\lambda+\delta}) p_*(\calF) \\
		&\simeq \CC_{\lambda-\delta}\otimes_{\univ{g}} p_*(\calF).
	\end{align*}
	We have proved the proposition.
\end{proof}

\begin{lemma}\label{lem:PrincipalBundleLocallyProjective}
	Let $U$ be an open subset of $X$.
	If $p_*(\Dsheaf_{\widetilde{X}})|_U$ is acyclic (e.g.\ $U$ is affine), $\sect(U, p_*\Dsheaf_{\widetilde{X}})$ is a projective left/right $\univ{g}$-module.
\end{lemma}

\begin{proof}
	Since the bundle $p\colon \widetilde{X}\rightarrow X$ is locally trivial in the \'etale topology,
	we can take an affine \'etale covering $\set{U_j \rightarrow U}$ such that
	$U_j \times_X \widetilde{X} \rightarrow U_j$ is a trivial principal $G$-bundle.
	Since $p_*\Dsheaf_{\widetilde{X}}$ is a quasi-coherent $\rsheaf{X}$-module,
	the cohomology group $H^i(U, p_*\Dsheaf_{\widetilde{X}})$ is isomorphic to the \'etale cohomology group $H^i(U_{\text{\'et}}, (p_*\Dsheaf_{\widetilde{X}})_{\text{\'et}})$ for any $i$.
	Here $(p_*\Dsheaf_{\widetilde{X}})_{\text{\'et}}$ is the \'etale sheaf associated to $p_*\Dsheaf_{\widetilde{X}}$.
	Hence the \v{C}ech complex
	\begin{align*}
		0 \rightarrow \sect(U, p_*\Dsheaf_{\widetilde{X}}) \rightarrow C^0 \rightarrow C^1 \rightarrow \cdots
	\end{align*}
	associated to the covering $\set{U_j \rightarrow U}$ is exact and each term $C^j$ is a free left/right
	$\univ{g}$-module.
	$\sect(U, p_*\Dsheaf_{\widetilde{X}})$ is therefore a projective left/right $\univ{g}$-module.
\end{proof}

By Lemma \ref{lem:PrincipalBundleLocallyProjective} and Proposition \ref{prop:PrincipalBundleDirectImage}, we obtain the following theorem.
Note that for a generalized pair $(\calA, G)$ and a left $\calA$-module $M$,
$\Tor^{\univ{g}}_i(\CC_{\lambda-\delta}, M)$ admits a natural $(\calA/I_{-\lambda+\delta}\calA)^G$-module structure if $\calA$ is a flat left $\univ{g}$-module.
In fact, $\Tor^{\univ{g}}_i(\CC_{\lambda-\delta}, M)$ can be computed by using a free resolution of the $\calA$-module $M$.

\begin{theorem}\label{thm:DirectImageTor}
	For any $\calM \in \Mod_{qc}(\Dsheaf_{\widetilde{X}})$, we have a natural isomorphism
	\begin{align*}
		D^{-i}p_{+,\lambda}(\calM) \simeq \Tor^{\univ{g}}_{i}(\CC_{\lambda - \delta}, p_*(\calM))
	\end{align*}
	of $\Dsheaf_{X,\lambda}$-modules, where $\Tor^{\univ{g}}_{i}(\CC_{\lambda - \delta}, p_*(\calM))$ denotes the sheafification of the presheaf $(U\rightarrow \Tor^{\univ{g}}_i(\CC_{\lambda - \delta}, \sect(p^{-1}(U), \calM)))$.
\end{theorem}

By Theorem \ref{thm:DirectImageTor}, there is a natural homomorphism
\begin{align*}
	\Tor^{\univ{g}}_{i}(\CC_{\lambda-\delta}, \sect(\calM)) \rightarrow \sect(D^{-i}p_{+,\lambda}(\calM))
\end{align*}
of $\Dalg{\widetilde{X}}^G$-modules, where $\Dalg{\widetilde{X}} = \sect(\Dsheaf_{\widetilde{X}})$.
In general, it is not isomorphic.
Under some assumption, we can show that the homomorphism is an isomorphism.

\begin{lemma}\label{lem:PrincipalBundleAcylic}
	Assume that $\Dsheaf_{\widetilde{X}}$ is acyclic.
	Then the natural homomorphism $(\Dalg{\widetilde{X}}/R(I_{-\lambda+\delta})\Dalg{\widetilde{X}})^G \rightarrow \Dalg{X,\lambda}$ is bijective.
	Moreover, for a free $\Dsheaf_{\widetilde{X}}$-module $\calF$, the natural homomorphism
	$\CC_{\lambda-\delta}\otimes_{\univ{g}} \sect(\calF) \rightarrow \sect(p_{+,\lambda}(\calF))$ is an isomorphism of $\Dalg{X,\lambda}$-modules.
\end{lemma}

\begin{proof}
	Let $\calF$ be a free $\Dsheaf_{\widetilde{X}}$-module.
	Since $p$ is affine and $\Dsheaf_{\widetilde{X}}$ is acyclic, $p_*(\calF)$ is also acyclic.
	Take a free resolution $\cdots \rightarrow P_1 \xrightarrow{d_1} P_0 \xrightarrow{d_0} \CC_{\lambda-\delta}\rightarrow 0$.
	By Lemma \ref{lem:PrincipalBundleLocallyProjective}, the following sequence is exact:
	\begin{align*}
		\cdots \rightarrow P_1\otimes_{\univ{g}}p_*(\calF) \xrightarrow{d_1} P_0\otimes_{\univ{g}}p_*(\calF) \xrightarrow{d_0} p_{+,\lambda}(\calF) \rightarrow 0.
	\end{align*}
	Since all $P_i\otimes_{\univ{g}}p_*(\calF)$ are acyclic, $\Ker(d_0)$ is also acyclic, and hence
	\begin{align*}
		\sect(P_1\otimes_{\univ{g}}p_*(\calF)) \xrightarrow{d_1} \sect(P_0\otimes_{\univ{g}}p_*(\calF)) \xrightarrow{d_0} \sect(p_{+,\lambda}(\calF)) \rightarrow 0
	\end{align*}
	is exact.
	Since each $P_i$ is free, we have $\sect(P_i\otimes_{\univ{g}} p_*(\calF)) \simeq P_i\otimes_{\univ{g}} \sect(\calF)$.
	This implies that the natural homomorphism $\CC_{\lambda-\delta}\otimes_{\univ{g}} \sect(\calF) \rightarrow \sect(p_{+,\lambda}(\calF))$ is bijective.
	Hence the second assertion follows from the first one.

	Since $(\cdot)^G$ is left exact, we have
	\begin{align*}
		\Dalg{X,\lambda} = \sect(p_{+,\lambda}(\Dsheaf_{\widetilde{X}})^G)=\sect(p_{+,\lambda}(\Dsheaf_{\widetilde{X}}))^G
		\simeq (\Dalg{\widetilde{X}}/R(I_{-\lambda+\delta})\Dalg{\widetilde{X}})^G.
	\end{align*}
	This implies that the natural algebra homomorphism $(\Dalg{\widetilde{X}}/R(I_{-\lambda+\delta})\Dalg{\widetilde{X}})^G\rightarrow \Dalg{X,\lambda}$ is an isomorphism.
\end{proof}

\begin{theorem}\label{thm:PrincipalBundleDirectImageDerived3}
	Let $\calM \in \Mod_{qc}(\Dsheaf_{\widetilde{X}})$ and $i \in \NN$.
	Assume that the global section functors are exact on $\Mod_{qc}(\Dsheaf_{\widetilde{X}})$ and $\Mod_{qc}(\Dsheaf_{X,\lambda})$.
	Then the natural homomorphism
	\begin{align*}
		\Tor_i^{\univ{g}}(\CC_{\lambda-\delta}, \sect(\calM))\rightarrow \sect(D^{-i} p_{+,\lambda}(\calM))
	\end{align*}
	is an isomorphism of $\Dalg{X,\lambda}$-modules.
\end{theorem}

\begin{proof}
	Remark that any object in $\Mod_{qc}(\Dsheaf_{\widetilde{X}})$ is acyclic because $\sect$ is exact on $\Mod_{qc}(\Dsheaf_{\widetilde{X}})$.
	Take a free resolution $\calF^\cxdot$ of $\calM$.
	By Lemma \ref{lem:PrincipalBundleLocallyProjective}, $\sect(\calF^\cxdot)$ is a projective resolution of $\sect(\calM)$ as a $\lie{g}$-module.
	Hence we have
	\begin{align*}
		\sect(D^{-i} p_{+,\lambda}(\calM)) &\simeq \sect(H^{-i}(p_{+,\lambda}(\calF^\cxdot))) \\
		&\simeq H^{-i}\circ \sect(p_{+,\lambda}(\calF^\cxdot)) \\
		&\simeq H^{-i}(\CC_{\lambda-\delta}\otimes_{\univ{g}} \sect(\calF^\cxdot)) \\
		&\simeq \Tor_i^{\univ{g}}(\CC_{\lambda-\delta}, \sect(\calM)).
	\end{align*}
	Here the third isomorphism follows from Lemma \ref{lem:PrincipalBundleAcylic}.
\end{proof}

\begin{remark}
	One can prove a similar result about the commutativity of $R\sect$ and $\CC_{\lambda-\delta}\otimes^L_{\univ{g}}(\cdot)$ without the exactness of $\sect$.
\end{remark}

$\sect$ is exact for any affine variety and $\lambda$, or for any flag variety and good $\lambda$ (see Fact \ref{fact:BeilinsonBernstein}).
To apply Theorem \ref{thm:PrincipalBundleDirectImageDerived3} to direct products of such varieties,
we shall show the following lemma.

\begin{lemma}
Let $X$ and $Y$ be smooth varieties and $\Dsheaf_X$ (resp.\ $\Dsheaf_Y$) an algebra of twisted differential operators on $X$ (resp.\ $Y$).
If the global section functors on $\Mod_{qc}(\Dsheaf_{X})$ and $\Mod_{qc}(\Dsheaf_{Y})$
are exact, then the global section functor on $\Mod_{qc}(\Dsheaf_{X}\boxtimes \Dsheaf_{Y})$ is also exact.
\end{lemma}

\begin{proof}
Let $0\rightarrow \calM_1 \rightarrow \calM_2 \rightarrow \calM_3 \rightarrow 0$
be a short exact sequence of quasi-coherent $\Dsheaf_{X}\boxtimes \Dsheaf_{Y}$-modules.
Let $U$ be an affine open subset of $X$.
We write $p\colon U\times Y \rightarrow Y$ for the projection onto the second factor.
Then $p$ is affine.
Hence we obtain a short exact sequence
\begin{align*}
0\rightarrow p_*(\calM_1|_{U\times Y}) \rightarrow p_*(\calM_2|_{U\times Y}) \rightarrow p_*(\calM_3|_{U\times Y}) \rightarrow 0
\end{align*}
of quasi-coherent $\Dsheaf_Y$-modules.
Since the global section functor $\sect$ on $\Mod_{qc}(\Dsheaf_Y)$ is exact,
we obtain a short exact sequence
\begin{align}
0\rightarrow \sect(U\times Y, \calM_1) \rightarrow \sect(U\times Y, \calM_2) \rightarrow \sect(U\times Y, \calM_3) \rightarrow 0.
\label{eqn:exactexacttoexact}
\end{align}

Let $q\colon X\times Y\rightarrow X$ be the projection onto the first factor.
By (\ref{eqn:exactexacttoexact}), $q_*\colon\Mod_{qc}(\Dsheaf_X\boxtimes \Dsheaf_Y) \rightarrow \Mod_{qc}(\Dsheaf_X)$
is exact.
Since the global section functor $\sect$ on $\Mod_{qc}(\Dsheaf_X)$ is exact,
this implies that the sequence $0\rightarrow \sect(\calM_1) \rightarrow \sect(\calM_2) \rightarrow \sect(\calM_3)\rightarrow 0$
is exact.
We have proved the lemma.
\end{proof}

%% file: uniformly_bounded_family.tex
\section{Uniformly bounded family}

The purpose of this section is to reformulate Bernstein's work \cite{Be72} about the multiplicity of a $\ntDalg{\CC^n}$-module.
We will introduce the notion of uniformly bounded families of twisted $\ntDsheaf$-modules.
A uniformly bounded family is a family with a good boundedness property, which is preserved by direct images and inverse images.
We give several applications of uniformly bounded families in Section \ref{sect:applications}.

\subsection{Multiplicity and functors}

In this subsection, we review Bernstein's work \cite{Be72} about the multiplicity (or the Bernstein degree) of a $\ntDalg{\CC^n}$-module.
We refer the reader to \cite{Be72}, \cite[3.2.2]{HTT08} and \cite[1.\S 3 and \S 4]{Bj79} for the proof of facts.

Let $\ntDsheaf_{\CC^n}$ be the algebra of non-twisted differential operators on $\CC^n$ and $\ntDalg{\CC^n}$ the algebra of global sections of $\ntDsheaf_{\CC^n}$.
Let $(x_1, x_2, \ldots, x_n)$ be the standard coordinate of $\CC^n$ and put $\partial_i = \partial / \partial x_i$.
We denote by $F$ the Bernstein filtration of $\ntDalg{\CC^n}$, and then
$F_0 \ntDalg{\CC^n} = \CC,
F_1 \ntDalg{\CC^n} = \spn{1, x_1, x_2, \ldots, x_n, \partial_1, \partial_2, \ldots, \partial_n},
F_i \ntDalg{\CC^n} = (F_1 \ntDalg{\CC^n})^i$.
The following facts are essential for our study of a family of $\ntDsheaf$-modules.

\begin{fact}\label{fact:WeylAlgebraMultiplicity}
Let $M$ be a finitely generated $\ntDalg{\CC^n}$-module and $M_0$ a generating subspace of $M$ of finite dimension.
Put $F_i M := F_i \ntDalg{\CC^n} \cdot M_0$.
Then
\begin{enumerate}[(i)]
	\item there exists some polynomial $f \in \QQ[t]$ such that $f(i) = \dim_{\CC}(F_i M)$ for any $i \gg 0$
	\item $d(M) := \deg(f)$ does not depend on $M_0$
	\item the coefficient $a_d$ of the degree $d(M)$ of $f$ does not depend on $M_0$
	\item $m(M):=a_d\cdot d(M)!$ is a natural number
	\item $d(M)\geq n$ if $M$ is non-zero.
\end{enumerate}
\end{fact}
The integer $m(M)$ is called the \define{multiplicity} (or the Bernstein degree) of $M$.
A $\ntDalg{\CC^n}$-module $M$ is said to be \define{holonomic} if $M$ is finitely generated and $d(M)=n$ or $d(M) = 0$ holds.
A $\ntDalg{\CC^n}$-module $M$ is holonomic if and only if the corresponding $\ntDsheaf_{\CC^n}$-module $\ntDsheaf_{\CC^n}\otimes_{\ntDalg{\CC^n}}M$
is holonomic (see \cite[Proposition 3.2.11]{HTT08}).
We put $m(\calM):=m(\sect(\calM))$ for $\calM \in \Mod_{h}(\ntDsheaf_{\CC^n})$
and
\begin{align*}
	m(\calM^\cxdot) := \sum_{i} m(H^i(\calM^\cxdot))
\end{align*}
for $\calM^\cxdot \in D^b_h(\ntDsheaf_{\CC^n})$.

\begin{fact}\label{fact:WeylAlgebraExact}
Let $0\rightarrow L \rightarrow M \rightarrow N \rightarrow 0$ be a short exact sequence of finitely generated $\ntDalg{\CC^n}$-modules.
Then we have $d(M) = \max(d(L), d(N))$.
If in addition $d(L)=d(N)$, then $m(M) = m(L) + m(N)$ holds.
In particular, the length of a holonomic $\ntDalg{\CC^n}$-module is less than or equal to its multiplicity.
\end{fact}

The following fact is an easy consequence of the definition of the multiplicity.
See the proof of \cite[Theorem 3.2]{Be72}.

\begin{fact}\label{fact:MultiplicityTensor}
	Let $N$ and $M$ be modules of $\ntDalg{\CC^n}$ and $\ntDalg{\CC^m}$, respectively.
	If $N$ and $M$ are holonomic, then $N\boxtimes M$ is holonomic and we have $m(N\boxtimes M) = m(N)m(M)$.
	Conversely, if $N\boxtimes M$ is holonomic, then $N$ and $M$ are holonomic.
\end{fact}

We need a derived functor version of \cite[Theorem 3.2]{Be72}.
The proof is the same as the original version.

\begin{fact}\label{fact:WeylAlgebraDirectInverseImage}
	Let $f\colon \CC^n \rightarrow \CC^m$ be a morphism of affine varieties.
	Set $d:=\max(\deg(f), 1)$.
	Then for any $\calM^\cxdot \in D^b_h(\ntDsheaf_{\CC^n})$ and $\calN^\cxdot \in D^b_h(\ntDsheaf_{\CC^m})$,
	we have
	\begin{align*}
		m(Df_+(\calM^\cxdot)) &\leq d^{n+m} m(\calM^\cxdot) \\
		m(Lf^*(\calN^\cxdot)) &\leq d^{n+m} m(\calN^\cxdot).
	\end{align*}
\end{fact}

\subsection{\texorpdfstring{$\ntDsheaf$}{D}-modules on affine varieties}

In Fact \ref{fact:WeylAlgebraDirectInverseImage}, we have seen that the multiplicity is well-behaved for operations of $\ntDsheaf$-modules on affine spaces.
In this subsection, we consider similar results about $\ntDsheaf$-modules on affine varieties.

We recall the Kashiwara equivalence \cite[Theorem 1.6.1]{HTT08}.
\begin{fact}\label{fact:Kashiwara}
	Let $f\colon X\rightarrow Y$ be a closed embedding of smooth varieties.
	Then
	\begin{align*}
		f_+ := D^0f_+\colon& \Mod_{qc}(\ntDsheaf_X) \rightarrow \Mod_{qc}^X(\ntDsheaf_Y) \text{ and} \\
		Df_+\colon& D_{qc}^b(\ntDsheaf_X) \rightarrow D_{qc}^{b,X}(\ntDsheaf_Y)
	\end{align*}
	give equivalences of categories.
	Here $\Mod_{qc}^X(\ntDsheaf_Y)$ is the full subcategory of $\Mod_{qc}(\ntDsheaf_Y)$
	whose objects are supported on $X$,
	and $D_{qc}^{b,X}(\ntDsheaf_Y)$ is the full subcategory of $D_{qc}^{b}(\ntDsheaf_Y)$
	consisting of complexes whose cohomologies are supported on $X$.
\end{fact}

For $\calM^\cxdot \in D^b_h(\ntDsheaf_X)$ on a smooth variety $X$
and a closed embedding $\iota\colon X\rightarrow \CC^n$, we set
\begin{align*}
	m_\iota(\calM^\cxdot) := m(D\iota_+(\calM^\cxdot)).
\end{align*}

\begin{proposition}\label{prop:LocalMultiplicity}
	Let $f\colon X\rightarrow Y$ be a morphism of affine smooth varieties.
	Fix closed embeddings $\iota\colon X\rightarrow \CC^n$ and $\iota'\colon Y\rightarrow \CC^m$.
	Then there exists a constant $C>0$ such that
	\begin{align}
		m_{\iota'}(Df_+(\calM^\cxdot)) &\leq C\cdot m_\iota(\calM^\cxdot), \label{eqn:BoundDirectImageAffine}\\
		m_{\iota}(Lf^*(\calN^\cxdot)) &\leq C\cdot m_{\iota'}(\calN^\cxdot) \label{eqn:BoundInverseImageAffine}
	\end{align}
	for any $\calM^\cxdot \in D^b_h(\ntDsheaf_X)$ and $\calN^\cxdot \in D^b_h(\ntDsheaf_Y)$.
\end{proposition}

\begin{proof}
	Fix an extension $\widetilde{f}$ of $f$ to $\CC^n$ such that the diagram
	\begin{align*}
		\xymatrix{
			X \ar[r]^f \ar[d]^-\iota & Y \ar[d]^-{\iota'} \\
			\CC^n \ar[r]^{\widetilde{f}} & \CC^m
		}
	\end{align*}
	is commutative.
	Set $d:=\max(\deg(\widetilde{f}), 1)$.
	By Fact \ref{fact:WeylAlgebraDirectInverseImage}, we obtain
	\begin{align*}
		m_{\iota'}(Df_+(\calM^\cxdot)) = m(D\widetilde{f}_+(D\iota_+(\calM^\cxdot)))
		\leq d^{n+m} m_\iota(\calM^\cxdot).
	\end{align*}
	Here we used $D\iota'_+ \circ Df_+ = D\widetilde{f}_+ \circ D\iota_+$ (Fact \ref{fact:FundamentalDmodule} \eqref{eqn:DirectImage}).
	We have shown the first inequality \eqref{eqn:BoundDirectImageAffine}.

	Consider the following diagram:
	\begin{align*}
		\xymatrix{
			X \ar[r]^-{f_1} \ar[d]^-\iota & X\times Y \ar[r]^-{f_2} \ar[d]^-{\iota\times \iota'} & Y\\
			\CC^n \ar[r]^-{\widetilde{f}_1} & \CC^n\times \CC^m, &
		}
	\end{align*}
	where $f_1(x) = (x, f(x))$, $f_2(x, y) = y$ and $\widetilde{f}_1(x) = (x, \widetilde{f}(x))$.
	Since the left square is cartesian, we have an isomorphism $D\iota_+\circ Lf_1^* \simeq L\widetilde{f}^*_1 \circ D(\iota\times \iota')_+$ of functors by the base change theorem (Fact \ref{fact:BaseChange}).
	Hence we obtain
	\begin{align*}
		m_{\iota_+}(Lf^*(\calN^\cxdot))) &= m(L\widetilde{f}_1^*\circ D(\iota\times \iota')_+ \circ Lf_2^*(\calN^\cxdot))\\
		&\leq d^{n+m}m(D\iota_+(\rsheaf{X})\boxtimes D\iota'_+(\calN^\cxdot)) \\
		&=d^{n+m}m_{\iota}(\rsheaf{X})m_{\iota'}(\calN^\cxdot)
	\end{align*}
	by Fact \ref{fact:MultiplicityTensor} and Fact \ref{fact:WeylAlgebraDirectInverseImage}, which proves the second inequality \eqref{eqn:BoundInverseImageAffine}.
\end{proof}

In the next subsection, we will consider families of twisted $\ntDsheaf$-modules on general smooth varieties.
Although the multiplicity itself is no longer a meaningful value for general smooth varieties, boundedness of multiplicities of twisted $\ntDsheaf$-modules can be defined.

To reduce properties of twisted $\ntDsheaf$-modules on a non-affine variety to that of affine spaces, we have many choices of affine \'etale coverings, closed embeddings to affine spaces, and local trivializations of an algebra of twisted differential operators.
We shall consider the effect on the multiplicity by the choices.

\begin{proposition}\label{prop:LocalMultiplicityEtale}
	Let $f\colon X\rightarrow Y$ be a surjective \'etale morphism of affine smooth varieties.
	Fix closed embeddings $\iota\colon X\rightarrow \CC^n$ and $\iota'\colon Y\rightarrow \CC^m$.
	Then there exists a constant $C>0$ such that
	\begin{align*}
		C^{-1} \cdot m_{\iota'}(\calN^\cxdot) \leq m_{\iota}(Lf^*(\calN^\cxdot)) \leq C\cdot m_{\iota'}(\calN^\cxdot)
	\end{align*}
	for any $\calN^\cxdot \in D^b_h(\ntDsheaf_Y)$.
\end{proposition}

\begin{proof}
	We have proved the second inequality in Proposition \ref{prop:LocalMultiplicity}.
	We shall show the first inequality.

	Since $f$ is smooth, $f^*$ is exact (\cite[Proposition 1.5.13]{HTT08}).
	Hence we can assume $\calN \in \Mod_{h}(\ntDsheaf_Y)$.

	Since $f$ is \'etale, $f_*(\calM)$ admits a natural $\ntDsheaf_Y$-module structure
	for $\calM \in \Mod_h(\ntDsheaf_X)$
	and the direct image functor $Df_+$ is isomorphic to $Rf_* = f_*$ by \cite[Theorem 2.2]{CoLe01}.
	Hence the canonical morphism $\calN \rightarrow f_*(f^*(\calN))$ of $\rsheaf{Y}$-modules
	is a morphism of $\ntDsheaf_Y$-modules.
	The morphism is monomorphic since $f$ is surjective.
	Applying Proposition \ref{prop:LocalMultiplicity} to $f^*(\calN)$,
	we obtain
	\begin{align*}
		m_{\iota'}(\calN) \leq m_{\iota'}(f_*(f^*(\calN))) \leq C\cdot m_{\iota}(f^*(\calN)),
	\end{align*}
	where $C$ is a constant independent of $\calN$.
\end{proof}

Hereafter we consider the effect on the multiplicity by twisting by automorphisms.

Let $X$ be a smooth affine variety and $\iota\colon X\rightarrow \CC^n$ a closed embedding.
The automorphism group $\Aut(\ntDsheaf_X)$ is isomorphic to the additive group $\calZ(X)$
of closed $1$-forms on $X$ (\cite{BeBe81}).
For $\omega \in \calZ(X)$, we denote by $A_\omega$ the corresponding automorphism given by
\begin{align*}
	A_\omega(T) = T - \omega(T) \in \calT_X\oplus \rsheaf{X}
\end{align*}
for $T \in \calT_X$.
A $\ntDsheaf_X$-module $\calM$ can be twisted by $A_\omega$ and 
the twisted module is denoted by $\calM^\omega$.
We use the same notation for a complex of $\ntDsheaf_X$-modules, e.g.\ $(\calM^\cxdot)^\omega$.

\begin{lemma}\label{lem:RegRingDeg}
	Let $W$ be a finite-dimensional subspace of $\calZ(X)$.
	Then there exists a constant $C$ such that
	\begin{align*}
		m_{\iota}(\rsheaf{X}^\omega) \leq C
	\end{align*}
	for any $\omega \in W$.
\end{lemma}

\begin{proof}
	Put $\calM:=\rsheaf{X}\otimes \rring{W}$ equipped with a $\calT_X$-action via
	\begin{align*}
		T\cdot (f\otimes g) = Tf \otimes g - \sum_i \omega_i(T) f\otimes \lambda_i g
		&& (T \in \calT_X),
	\end{align*}
	where $\set{\omega_i}_i$ is a basis of $W$ and $\set{\lambda_i}_i$ is its dual basis.
	Then the action on $\calM$ extends to a $\ntDsheaf_X\otimes \rring{W}$-action.
	
	We denote by $\frakm_\omega$ the maximal ideal of $\rring{W}$ corresponding to $\omega \in W$.
	Then by definition, we have $\calM / \frakm_\omega \calM \simeq \rsheaf{X}^\omega$
	for any $\omega \in W$.
	Since the functors $\iota_+$ and $\sect$ are exact, we have
	\begin{align*}
		\sect(\iota_+(\rsheaf{X}^\omega)) \simeq \sect(\iota_+(\calM))/ \frakm_\omega \sect(\iota_+(\calM)).
	\end{align*}
	Put $M:=\sect(\iota_+(\calM))$.
	
	Since the functors $\iota_+$ and $\sect$ preserve the lattice of submodules, $M$ is noetherian and hence finitely generated as a $\ntDalg{\CC^n}\otimes \rring{W}$-module.
	Take a finite generating subspace $S \subset M$ and put $F_i M := (F_i \ntDalg{\CC^n}\otimes \rring{W}) S$ for $i \geq 0$.
	Then the associated graded module $\gr^F M$ is a finitely generated $\rring{\CC^n\times W}$-module.
	
	By \cite[Theorem 24.1]{Ma86}, we can take an affine open subset $U$ of $W$ such that
	$\rring{U}\otimes_{\rring{W}}\gr^F M$ is a free $\rring{U}$-module.
	Hence $\rring{U}\otimes_{\rring{W}}F_i M$ is a projective $\rring{U}$-module for any $i \geq 0$.
	This implies that the function
	\begin{align*}
		W\ni \omega \mapsto \dim_{\CC}(F_i M/\frakm_\omega F_i M)
	\end{align*}
	is constant on $U$.
	Hence $U\ni \omega \mapsto m(M/\frakm_\omega M)$ is a constant function by the definition of the multiplicity.

	Replacing $W$ by $W\backslash U$ and $M$ by $\rring{W\backslash U}\otimes_{\rring{W}} M$, and repeating this argument, we can see that
	$m(M/\frakm_\omega M)$ is bounded on $W$.
\end{proof}

\begin{remark}
	Lemma \ref{lem:RegRingDeg} can be considered as a special case of \cite[Theorem 3.18]{AiGoDm16}
	and the latter half of our proof is essentially the same as theirs.
\end{remark}

\begin{corollary}\label{cor:MultiplicityTwisted}
	Let $\calM^\cxdot \in D^b_h(\ntDsheaf_X)$ and $W$ be a finite-dimensional subspace of $\calZ(X)$.
	Then there exists a constant $C$ independent of $\calM^\cxdot$ such that
	\begin{align*}
		m_{\iota}((\calM^\cxdot)^\omega) \leq C \cdot m_\iota(\calM^\cxdot)
	\end{align*}
	for any $\omega \in W$.
\end{corollary}

\begin{proof}
	Fix $\omega \in W$.
	Since $(\calM^\cxdot)^\omega \simeq \calM^\cxdot \otimes_{\rsheaf{X}} \rsheaf{X}^\omega$,
	we have
	\begin{align*}
		D\iota_+((\calM^\cxdot)^\omega) \simeq D\iota_+(\calM^\cxdot)\otimes_{\rsheaf{\CC^n}}^L \iota_+(\rsheaf{X}^\omega)
	\end{align*}
	by the base change theorem (Fact \ref{fact:BaseChange}).
	By Facts \ref{fact:MultiplicityTensor} and \ref{fact:WeylAlgebraDirectInverseImage}, we obtain
	\begin{align*}
		m_\iota((\calM^\cxdot)^\omega) = m(D\iota_+(\calM^\cxdot)\otimes^L_{\rsheaf{\CC^n}} \iota_+(\rsheaf{X}^\omega))
		\leq m_\iota(\calM^\cxdot) m_\iota(\rsheaf{X}^\omega).
	\end{align*}
	This inequality and Lemma \ref{lem:RegRingDeg} imply the assertion.
\end{proof}

\subsection{Uniformly bounded family}\label{sect:UniformlyBoundedFamily}

We shall define a good local trivialization of a family of algebras of twisted differential operators.
For an \'etale map $\varphi\colon U\rightarrow V$, we denote by $(\cdot)|_U$ the functors $\varphi^\#(\cdot)$ and $\varphi^*(\cdot)$ by abuse of notation.

Let $\Dsheaf_{X,\Lambda}=(\Dsheaf_{X, \lambda})_{\lambda \in \Lambda}$ be a family of algebras of twisted differential operators on a smooth variety $X$.
Hereafter we deal with $\prod_{\lambda \in \Lambda} \Mod_h(\Dsheaf_{X,\lambda})$ the direct product of categories and its derived category.
Set
\begin{align*}
	\Mod_h(\Dsheaf_{X,\Lambda}) &:= \prod_{\lambda \in \Lambda} \Mod_h(\Dsheaf_{X,\lambda}),\\
	D^b_h(\Dsheaf_{X,\Lambda}) &:= \prod_{\lambda \in \Lambda} D^b_h(\Dsheaf_{X,\lambda}).
\end{align*}
We denote by $H^i, Df_+, Lf^*, (\cdot)|_U$ and $f^\#$ the direct products of the corresponding functors by abuse of notation.

Recall that $\calZ(X)$ is the space of closed $1$-forms on $X$, which is isomorphic to $\Aut(\ntDsheaf_X)$ as an abelian group.

\begin{definition}\label{def:BoundedTrivialization}
	We say that a tuple $(U, \varphi, \Phi)$
	is a \define{trivialization} of $\Dsheaf_{X,\Lambda}$
	if $U$ is a smooth variety, $\varphi\colon U\rightarrow X$ is a surjective \'etale morphism and $\Phi$ is a family of isomorphisms $\Phi_\lambda \colon \Dsheaf_{X,\lambda}|_U \xrightarrow{\simeq} \ntDsheaf_U$.

	Let $T_1=(U, \varphi, \Phi)$ and $T_2=(V, \psi, \Psi)$ be trivializations of $\Dsheaf_{X,\Lambda}$.
	We denote by $\calZ(T_1, T_2)\subset \calZ(U\times_X V)$ the image of
	\begin{align*}
		\set{\widetilde{\varphi}^\#\Psi_{\lambda}\circ (\widetilde{\psi}^\#\Phi_{\lambda})^{-1}: \lambda \in \Lambda}
	\end{align*}
	by the isomorphism $\Aut(\ntDsheaf_{U\times_X V}) \rightarrow \calZ(U\times_X V)$.
	Here $\widetilde{\varphi}\colon U\times_X V\rightarrow V$ and $\widetilde{\psi}\colon U\times_X V \rightarrow U$ are the projections of the fiber product.
	We write $T_1 \sim T_2$ when $\calZ(T_1, T_2)$ spans a finite-dimensional subspace of $\calZ(U\times_X V)$.

	We say that a trivialization $(U, \varphi, \Phi)$ is \define{bounded} if
	$(U,\varphi, \Phi) \sim (U,\varphi, \Phi)$ holds.
\end{definition}

\begin{remark}
	Let $T=(U,\varphi, \Phi)$ be a trivialization of $\Dsheaf_{X,\Lambda}$.
	Then any element of $\calZ(T, T)$ is a $1$-cocycle of the \v{C}ech complex of the sheaf of closed $1$-forms on $X$ with respect to the \'etale covering $\varphi\colon U\rightarrow X$.
	Hence for each $\lambda \in \Lambda$, we have a $1$-cocycle $c(\lambda) \in \calZ(T,T)$, and the cocycle defines an algebra $\ntDsheaf_{X,c(\lambda)}$ of twisted differential operators on $X$.
	Then $\Phi_\lambda$ extends to an isomorphism $\Phi'_\lambda\colon \Dsheaf_{X,\lambda} \rightarrow \ntDsheaf_{X,c(\lambda)}$.
	It is obvious that the correspondence
	\begin{align*}
		(U,\varphi, \Phi) \mapsto (U,\varphi, (c(\lambda))_{\lambda \in \Lambda}, (\Phi'_\lambda)_{\lambda \in \Lambda})
	\end{align*}
	is one-to-one.
	One can use such tuples instead of our trivializations.
\end{remark}

\begin{definition}
	Let $T=(U, \varphi, \Phi)$ be a trivialization of $\Dsheaf_{X,\Lambda}$
	and $f\colon Y\rightarrow X$ a morphism of smooth varieties.
	We set $f^\#T := (U\times_X Y, \widetilde{\varphi}, \widetilde{f}^\# \Phi)$,
	where $\widetilde{\varphi}\colon U\times_X Y\rightarrow Y$ and $\widetilde{f}\colon U\times_X Y \rightarrow U$ are the projections of the fiber product.
\end{definition}

It is clear that $f^\#T$ is a trivialization of $f^\#\Dsheaf_{X,\Lambda}$.

The relation $\sim$ is clearly symmetric and not reflexive in general.
We shall show fundamental properties of bounded trivializations.
The following lemma is well-known and easy.

\begin{lemma}\label{lem:CommutativeClosed1Forms}
	Let $f\colon U\rightarrow V$ be a morphism of smooth varieties.
	Then the following diagram of abelian groups is commutative:
	\begin{align*}
		\xymatrix {
			\Aut(\ntDsheaf_{V}) \ar[r]^-{f^\#} \ar[d]^-{\simeq} & \Aut(\ntDsheaf_U) \ar[d]^-{\simeq} \\
			\calZ(V) \ar[r]^-{f^*} & \calZ(U).
		}
	\end{align*}
	If, in addition, $f$ is dominant, then $f^*$ is injective.
\end{lemma}

\begin{proposition}\label{prop:FundamentalBornology}
	Let $T_i=(U_i, \varphi_i,\Phi_i)$ ($i=1,2,3$) be trivializations of $\Dsheaf_{X,\Lambda}$ and $f\colon Y\rightarrow X$ a morphism of smooth varieties.
	\begin{enumerate}[(i)]
		\item $\sim$ is transitive, i.e.\ $T_1\sim T_2 \text{ and }T_2 \sim T_3\Rightarrow T_1 \sim T_3$.
		\item $T_1 \sim T_2 \Rightarrow f^\# T_1 \sim f^\# T_2$.
		\item If $T_1$ is bounded, then so is $f^\#T_1$.
		\item If $f$ is dominant, the converse of (ii) and (iii) are true.
	\end{enumerate}
	
\end{proposition}

\begin{proof}
	To show (i), let $f_{ij}\colon  U_1\times_X U_2 \times_X U_3 \rightarrow U_i \times_X U_j$
	be the projections of the fiber product for $(i,j)=(1,2),(2,3),(1,3)$.
	Assume $T_1\sim T_2$ and $T_2 \sim T_3$.
	Then we have
	\begin{align*}
		f_{13}^*(\calZ(T_1, T_3)) \subset f_{12}^*(\calZ(T_1, T_2)) + f_{23}^*(\calZ(T_2, T_3))
	\end{align*}
	by Lemma \ref{lem:CommutativeClosed1Forms}.
	Since $f_{13}$ is surjective, $f_{13}^*$ is injective.
	Hence $\calZ(T_1, T_3)$ spans a finite-dimensional subspace of $\calZ(U_1\times_X U_3)$.

	By definition, (iii) follows from (ii).
	We shall show (ii) and (iv).
	Let $\widetilde{f}\colon U_1\times_X U_2\times_X Y \rightarrow U_1\times_X U_2$ be the projection.
	By Lemma \ref{lem:CommutativeClosed1Forms}, we have
	\begin{align*}
		\widetilde{f}^*(\calZ(T_1, T_2)) = \calZ(f^\#T_1, f^\#T_2).
	\end{align*}
	This implies (ii) and (iv).
\end{proof}

By Proposition \ref{prop:FundamentalBornology}, the relation $\sim$ is an equivalence relation of bounded trivializations.

\begin{definition}\label{def:EquivalenceBornology}
	An equivalence class of bounded trivializations is called a \define{bornology} of the family $\Dsheaf_{X,\Lambda}$.
\end{definition}

If $\calB$ is a bornology of $\Dsheaf_{X,\Lambda}$ and $(\lambda(i))_{i \in I}$ is a family of elements of $\Lambda$, then $T:=(U,\varphi,(\Phi_{\lambda(i)})_{i \in I})$ is a bounded trivialization of $(\Dsheaf_{X,\lambda(i)})_{i \in I}$ for any $(U,\varphi,\Phi) \in \calB$.
It is clear that the equivalence class of $T$ does not depend on the choice of $(U,\varphi,\Phi) \in \calB$.
We denote by the same symbol $\calB$ the equivalence class of $T$ by abuse of notation.

\begin{definition}\label{def:pull-backBornology}
	Let $f\colon Y\rightarrow X$ be a morphism of smooth varieties
	and $\calB$ a bornology of $\Dsheaf_{X,\Lambda}$.
	By Proposition \ref{prop:FundamentalBornology} (ii), the equivalence class of $f^\#T$ ($T \in \calB$) does not depend on the choice of $T$.
	We denote by $f^\#\calB$ the equivalence class.
\end{definition}

The following proposition is an easy consequence of the definition.

\begin{proposition}\label{prop:pull-backBornology}
	Let $f\colon Y\rightarrow X$ and $g\colon Z\rightarrow Y$ be morphisms of smooth varieties.
	For any bornology $\calB$ of $\Dsheaf_{X,\Lambda}$, we have $(f\circ g)^\#\calB = g^\# f^\# \calB$ as bornologies of $(f\circ g)^\#\Dsheaf_{X,\Lambda} = g^\# f^\# \Dsheaf_{X,\Lambda}$.
\end{proposition}

It is not obvious that a bornology contains enough many trivializations for applications.
We can make a good bounded trivialization from a bounded trivialization by the following proposition.

\begin{proposition}\label{prop:LiftTrivialization}
	Let $T_U=(U, \varphi, \Phi)$ be a trivialization of $\Dsheaf_{X,\Lambda}$
	and $f\colon V\rightarrow U$ a surjective \'etale morphism.
	Put $T_V:=(V, \varphi\circ f, f^\#\Phi)$.
	Then the following conditions are equivalent:
	\begin{enumerate}[(i)]
		\item $T_U$ is bounded,
		\item $T_V$ is bounded,
		\item $T_U \sim T_V$.
	\end{enumerate}
	In particular, for any bornology $\calB$ of $\Dsheaf_{X,\Lambda}$,
	there exists a trivialization $(W, \psi, \Psi)$ in $\calB$ such that $W$ is affine.
\end{proposition}

\begin{proof}
	Let $f_1\colon V\times_X V\rightarrow U\times_X V$ and $f_2\colon U \times_X V \rightarrow U\times_X U$ be the morphisms determined by the universal property of the fiber bundles.
	Then by Lemma \ref{lem:CommutativeClosed1Forms}, we have
	\begin{align*}
		f_2^*(\calZ(T_U, T_U)) &= \calZ(T_U, T_V) \\
		f_1^*(\calZ(T_U, T_V)) &= \calZ(T_V, T_V).
	\end{align*}
	Since $f_1$ and $f_2$ are surjective, $f_2^*$ and $f_1^*$ are injective.
	Hence (i), (ii) and (iii) are equivalent.
	The second assertion is clear because for any variety $U$, there is a surjective \'etale morphism $W\rightarrow U$ from an affine variety $W$.
\end{proof}

\begin{definition}\label{def:UniformlyBoundedFamilyOri}
	Let $T=(U, \varphi, \Phi)$ be a trivialization of $\Dsheaf_{X,\Lambda}$ with affine $U$.
	We say that an object $(\calM_\lambda)_{\lambda \in \Lambda} \in \Mod_h(\Dsheaf_{X,\Lambda})$ is \define{uniformly bounded} with respect to $T$ if 
	for any closed embedding $\iota\colon U\rightarrow \CC^n$,
	$m_{\iota}(\calM_\lambda|_{U})$ is bounded as a function on $\Lambda$.
	Here we consider an $\Dsheaf_{X,\lambda}|_{U}$-module as a $\ntDsheaf_{U}$-module by the isomorphism $\Phi_\lambda\colon \Dsheaf_{X,\lambda}|_U \rightarrow \ntDsheaf_U$.

	We say that an object $\calM \in D^b_h(\Dsheaf_{X,\Lambda})$
	is \define{uniformly bounded} with respect to $T$ if $H^i(\calM)$
	is uniformly bounded for any $i$ and $H^i(\calM)$ vanishes for any $|i| \gg 0$.
	Here $H^i(\calM)$ is the family $(H^i(\calM_\lambda))_{\lambda \in \Lambda}$.

	We denote by $\Mod_{ub}(\Dsheaf_{X,\Lambda}, T)$ (resp.\ $D^b_{ub}(\Dsheaf_{X,\Lambda}, T)$)
	the full subcategory of $\Mod_{h}(\Dsheaf_{X,\Lambda})$
	(resp.\ $D^b_{h}(\Dsheaf_{X,\Lambda})$)
	consisting of uniformly bounded objects with respect to $T$.
\end{definition}

\begin{remark}
	By Proposition \ref{prop:LocalMultiplicityEtale}, the boundedness of $m_{\iota}(\calM_\lambda|_{U})$
	does not depend on the choice of the embedding $\iota$.
\end{remark}

The following propositions are easy consequences of the definition.

\begin{proposition}\label{prop:FundamentalBoundeFamily}
	Let $T$ be a bounded trivialization of $\Dsheaf_{X,\Lambda}$.
	Then the following hold.
	\begin{enumerate}[(i)]
		\item $\Mod_{ub}(\Dsheaf_{X,\Lambda}, T)$ is abelian.
		\item For a short exact sequence $0\rightarrow L\rightarrow M \rightarrow N\rightarrow 0$
		in $\Mod_{h}(\Dsheaf_{X,\Lambda})$,
		both $L$ and $N$ are uniformly bounded if and only if so is $M$.
		\item $D^b_{ub}(\Dsheaf_{X,\Lambda}, T)$ is a triangulated subcategory of
		$D^b_{h}(\Dsheaf_{X,\Lambda})$.
	\end{enumerate}
\end{proposition}

\begin{proposition}
	Let $T=(U,\varphi, \Phi)$ be a bounded trivialization with affine $U$.
	Then for any $(\calM_\lambda)_{\lambda \in \Lambda} \in \Mod_{ub}(\Dsheaf_{X,\Lambda}, T)$, the function $\Len_{\Dsheaf_{X,\lambda}}(\calM_{\lambda})$ of $\lambda \in \Lambda$
	is bounded.
\end{proposition}

\begin{proof}
	Fix a closed embedding $\iota\colon U\rightarrow \CC^n$.
	Then we have
	\begin{align*}
		\Len_{\Dsheaf_{X,\lambda}|_U}(\calM_\lambda|_U) = \Len_{\ntDsheaf_{\CC^n}}(\iota_+(\calM_\lambda|_{U})) \leq m_\iota(\calM_\lambda|_{U}).
	\end{align*}
	The first equality follows from the Kashiwara equivalence (Fact \ref{fact:Kashiwara}) and the second inequality from Fact \ref{fact:WeylAlgebraExact}.
	By the definition of uniformly bounded family, there is a constant $C$ independent of $\lambda \in \Lambda$ such that
	\begin{align*}
		\Len_{\Dsheaf_{X,\lambda}|_U}(\calM_\lambda|_U) \leq m_\iota(\calM_\lambda|_{U}) \leq C.
	\end{align*}
	Since $\varphi$ is surjective \'etale, the inverse image functor $\varphi^*$ is exact and sends a non-zero module to a non-zero module (see the proof of Proposition \ref{prop:LocalMultiplicityEtale}).
	Hence we obtain
	\begin{align*}
		\Len_{\Dsheaf_{X,\lambda}}(\calM_\lambda) \leq C
	\end{align*}
	for any $\lambda \in \Lambda$.
\end{proof}

We will show that the uniform boundedness is preserved by inverse images
and direct images.
To do so, we need the following basic proposition.

\begin{proposition}\label{prop:WellDefinedBoundedFamilyCat}
	Let $T_i=(U_i, \varphi_i, \Phi_i)$ ($i=1,2$) be bounded trivializations of $\Dsheaf_{X,\Lambda}$ with affine $U_i$.
	If $T_1 \sim T_2$, then we have
	\begin{align*}
		\Mod_{ub}(\Dsheaf_{X,\Lambda}, T_1) &= \Mod_{ub}(\Dsheaf_{X,\Lambda}, T_2),\\
		D^b_{ub}(\Dsheaf_{X,\Lambda}, T_1) &= D^b_{ub}(\Dsheaf_{X,\Lambda}, T_2).
	\end{align*}
\end{proposition}

\begin{proof}
	By definition, the second equation follows from the first one.

	Let $p_i\colon U_1\times_X U_2 \rightarrow U_i$ ($i=1,2$) be the projections and put $T'_i:=(U_1\times_X U_2, \varphi_i \circ p_i, p_i^\#\Phi_i)$ for $i=1,2$.
	Applying Proposition \ref{prop:LocalMultiplicityEtale} to $f = p_i$, we have
	\begin{align}
		\Mod_{ub}(\Dsheaf_{X,\Lambda}, T_i) = \Mod_{ub}(\Dsheaf_{X,\Lambda}, T'_i). \label{eqn:ProofUniformlyBoundedCat}
	\end{align}
	By $T_1 \sim T_2$, $\calZ(T_1, T_2)$ spans a finite-dimensional subspace of $\calZ(U_1\times_X U_2)$.
	Applying Corollary \ref{cor:MultiplicityTwisted} to $W=\mathrm{span}_\CC \calZ(T_1, T_2)$, we have
	\begin{align*}
		\Mod_{ub}(\Dsheaf_{X,\Lambda}, T'_1) = \Mod_{ub}(\Dsheaf_{X,\Lambda}, T'_2).
	\end{align*}
	This and \eqref{eqn:ProofUniformlyBoundedCat} imply the desired equation.
\end{proof}

The following definition is well-defined by Proposition \ref{prop:WellDefinedBoundedFamilyCat}.

\begin{definition}
	Let $\calB$ be a bornology of $\Dsheaf_{X,\Lambda}$.
	We set
	\begin{align*}
		\Mod_{ub}(\Dsheaf_{X,\Lambda}, \calB) &:= \Mod_{ub}(\Dsheaf_{X,\Lambda}, T),\\
		D^b_{ub}(\Dsheaf_{X,\Lambda}, \calB) &:= D^b_{ub}(\Dsheaf_{X,\Lambda}, T),
	\end{align*}
	where $T=(U,\varphi, \Phi)$ is a bounded trivialization in $\calB$
	with affine $U$.
\end{definition}

\begin{theorem}\label{thm:FunctorOnUniformlyBounded}
	Let $f\colon Y\rightarrow X$ be a morphism of smooth varieties and $\calB$ a bornology of $\Dsheaf_{X,\Lambda}$.
	The direct image functor and the inverse image functor preserve the uniform boundedness, that is, we have functors
	\begin{align*}
		Df_+\colon D^b_{ub}(f^\#\Dsheaf_{X,\Lambda}, f^\# \calB) \rightarrow D^b_{ub}(\Dsheaf_{X,\Lambda}, \calB), \\
		Lf^*\colon D^b_{ub}(\Dsheaf_{X,\Lambda}, \calB)\rightarrow D^b_{ub}(f^\#\Dsheaf_{X,\Lambda}, f^\# \calB).
	\end{align*}
\end{theorem}

\begin{proof}
	Take $T=(U, \varphi, \Phi) \in \calB$ with affine $U$, and a surjective \'etale morphism $V\rightarrow U\times_X Y$ from an affine variety $V$.
	Consider the following diagram:
	\begin{align*}
		\xymatrix{
			V \ar[r]^-{g} & U\times_X Y \ar[r]^-{\widetilde{f}} \ar[d]^-{\widetilde{\varphi}} & U \ar[d]^-{\varphi} \\
			& Y \ar[r]^-{f} & X,
		}
	\end{align*}
	where $\widetilde{\varphi}$ and $\widetilde{f}$ are the projections.
	Then $(V, \widetilde{\varphi}\circ g, (\widetilde{f}\circ g)^\# \Phi)$
	is in $f^\#\calB$ by Proposition \ref{prop:LiftTrivialization}.
	Since $L(\widetilde{\varphi}\circ g)^* \circ Lf^* = L(\widetilde{f}\circ g)^* \circ L\varphi^*$ holds (Fact \ref{fact:FundamentalDmodule} \eqref{eqn:InverseImage}),
	the assertion for the inverse image functor is reduced to Proposition \ref{prop:LocalMultiplicity} for $f = \widetilde{f}\circ g$.

	We shall show the assertion for $Df_+$.
	Take a finite affine open covering $\set{V_i}_{i = 0,1,2,\ldots, r}$ of $U\times_X Y$ and replace $V$ with the $(r+1)$-fold fiber product of $\bigsqcup_{i} V_i$.
	Let $\calM \in \Mod_{ub}(f^\#\Dsheaf_{X,\Lambda}, f^\#\calB)$
	and fix $\lambda \in \Lambda$.
	Then $\calM_\lambda|_{U\times_X Y}$ is quasi-isomorphic to the \v{C}ech complex
	\begin{align*}
		0\rightarrow C^0 \rightarrow C^1 \rightarrow \cdots \rightarrow C^r \rightarrow 0
	\end{align*}
	with respect to the covering $\set{V_i}$.
	By the construction of \v{C}ech complex, $\bigoplus_i C^i$ is a direct summand of $g_*(\calM_\lambda|_V)$.

	By the base change theorem (Fact \ref{fact:BaseChange}), there is an isomorphism $L\varphi^* \circ Df_+ \simeq D\widetilde{f}_+ \circ L\widetilde{\varphi}^*$ of functors, and hence we have
	\begin{align*}
		L\varphi^* \circ Df_+(\calM_\lambda) \simeq D\widetilde{f}_+ (\calM_\lambda|_{U\times_X Y}) \simeq D\widetilde{f}_+(C^\cxdot).
	\end{align*}
	For a closed embedding $\iota\colon U\rightarrow \CC^n$, we have
	\begin{align*}
		m_\iota(D\widetilde{f}_+(C^\cxdot)) &\leq \sum_i m_{\iota}(D\widetilde{f}_+(C^i)) \\
		&\leq m_\iota(D\widetilde{f}_+ \circ g_* (\calM_\lambda|_V)) \\
		& =  m_\iota(D(\widetilde{f} \circ g)_+ (\calM_\lambda|_V)).
	\end{align*}
	Here the first inequality follows from Lemma \ref{lem:AdditiveFunctionDerived} (iii) for the complex $C^\cxdot$.
	Note that $Dg_+$ is isomorphic to $g_*$ (see the proof of Proposition \ref{prop:LocalMultiplicityEtale}).

	By Proposition \ref{prop:LocalMultiplicity} and $\calM \in \Mod_{ub}(f^\#\Dsheaf_{X,\Lambda}, f^\#\calB)$, there is a constant $C$ independent of $\lambda$ such that
	\begin{align*}
		m_{\iota}(L\varphi^* \circ Df_+(\calM_\lambda)) \leq m_\iota(D(\widetilde{f} \circ g)_+ (\calM_\lambda|_V)) \leq C.
	\end{align*}
	This shows $Df_+(\calM) \in D^b_{ub}(\Dsheaf_{X,\Lambda}, \calB)$.
\end{proof}

We study the external tensor product of uniformly bounded families.
Let $\Dsheaf_{Y,\Lambda}$ be a family of algebras of twisted differential operators on a smooth variety $Y$ with the same index set $\Lambda$ as $\Dsheaf_{X,\Lambda}$.

\begin{definition}\label{def:TensorBornology}
	Let $\calB$ and $\calB'$ be bornologies of $\Dsheaf_{X,\Lambda}$ and $\Dsheaf_{Y,\Lambda}$, respectively.
	We denote by $\calB\boxtimes \calB'$ the equivalence class of $(U\times V, \varphi \times \psi, \Phi\boxtimes \Psi = (\Phi_\lambda \boxtimes \Psi_\lambda)_{\lambda \in \Lambda})$ for some $(U, \varphi, \Phi) \in \calB$
	and $(V, \psi, \Psi) \in \calB'$.
	If $X = Y$, we denote by $\calB\tensorBorn \calB'$ the pull-back of $\calB\boxtimes \calB'$ by the diagonal embedding $X\hookrightarrow X\times X$.
\end{definition}

It is easy to see that the definition is well-defined.
We set $\Dsheaf_{X,\Lambda}\boxtimes \Dsheaf_{Y,\Lambda}:=(\Dsheaf_{X,\lambda}\boxtimes \Dsheaf_{Y,\lambda})_{\lambda \in \Lambda}$ and
$\calM\boxtimes \calN := (\calM_\lambda \boxtimes \calN_\lambda)_{\lambda \in \Lambda}$ for $\calM \in D^b_{ub}(\Dsheaf_{X,\Lambda}, \calB)$
and $\calN \in D^b_{ub}(\Dsheaf_{Y,\Lambda}, \calB')$.
By Fact \ref{fact:MultiplicityTensor}, we obtain the following theorem.

\begin{theorem}\label{thm:TensorUniformlyBounded}
	Let $\calB$ and $\calB'$ be bornologies of $\Dsheaf_{X,\Lambda}$ and $\Dsheaf_{Y,\Lambda}$, respectively.
	Then $\calB\boxtimes \calB'$ is a bornology of $\Dsheaf_{X,\Lambda}\boxtimes \Dsheaf_{Y,\Lambda}$ and we have a bifunctor
	\begin{align*}
		(\cdot) \boxtimes (\cdot)\colon D^b_{ub}(\Dsheaf_{X,\Lambda}, \calB) \times D^b_{ub}(\Dsheaf_{Y,\Lambda}, \calB') \rightarrow D^b_{ub}(\Dsheaf_{X,\Lambda}\boxtimes \Dsheaf_{Y,\Lambda}, \calB\boxtimes \calB').
	\end{align*}
	Moreover, for any $\calM \in D^b_{h}(\Dsheaf_{X,\Lambda})$
	and $\calN \in D^b_{h}(\Dsheaf_{Y,\Lambda})$,
	both $\calM$ and $\calN$ are uniformly bounded if and only if so is $\calM\boxtimes \calN$.
\end{theorem}

\subsection{Twisting, opposite and tensor product}\label{sect:twisting}

We consider operations of algebras of twisted differential operators:
twisting by an invertible sheaf, taking opposite algebras and tensor products.
Corresponding to the operations, we introduce operations of a bornology.

Let $\Dsheaf_{X, \Lambda} = (\Dsheaf_{X,\lambda})_{\lambda \in \Lambda}$ be a family of algebras of twisted differential operators on a smooth variety $X$.
Let $\calB$ be a bornology of $\Dsheaf_{X,\Lambda}$ and $\calL$ an invertible sheaf on $X$.
Then we have a new family
\begin{align*}
	\Dsheaf_{X,\Lambda}^\calL := (\calL\otimes_{\rsheaf{X}} \Dsheaf_{X,\lambda}\otimes_{\rsheaf{X}} \calL^{\invshf})_{\lambda \in \Lambda}.
\end{align*}
Remark that for a morphism $f\colon Y\rightarrow X$ of smooth varieties, there is a canonical isomorphism
\begin{align}
	f^\#(\calL\otimes_{\rsheaf{X}} \Dsheaf_{X,\lambda}\otimes_{\rsheaf{X}} \calL^{\invshf}) \simeq f^*(\calL) \otimes_{\rsheaf{Y}} f^\#\Dsheaf_{X,\lambda}\otimes_{\rsheaf{Y}} f^*(\calL)^\invshf. \label{eqn:IsomTwist}
\end{align}
See e.g.\ \cite[Lemma 1.1.5]{KaTa96}.

We shall construct a bornology of $\Dsheaf^\calL_{X,\Lambda}$.
Since $\calL$ is an invertible sheaf, there is a bounded trivialization $T=(U,\varphi, \Phi) \in \calB$ such that $\calL|_U$ is isomorphic to $\rsheaf{U}$.
Take a trivialization $\alpha\colon \calL|_U \xrightarrow{\simeq} \rsheaf{U}$.
Then $T$ and $\alpha$ induce an isomorphism $\Phi^{\calL, \alpha}$ given by
\begin{align*}
	(\calL\otimes_{\rsheaf{X}} \Dsheaf_{X,\Lambda}\otimes_{\rsheaf{X}} \calL^{\invshf})|_U \xrightarrow{\id\otimes \Phi\otimes \id} \calL|_U\otimes_{\rsheaf{U}} \ntDsheaf_{U}\otimes_{\rsheaf{U}} (\calL|_U)^{\invshf} \rightarrow \ntDsheaf_{U}.
\end{align*}
We obtain a trivialization $T^{\calL, \alpha}=(U,\varphi, \Phi^{\calL,\alpha})$ of $\Dsheaf_{X,\Lambda}^\calL$.

\begin{lemma}\label{lem:BoundedTwistedBornology}
	Take $S=(V,\psi, \Psi) \in \calB$ and an isomorphism $\beta\colon \calL|_V \rightarrow \rsheaf{V}$.
	Then we have $S^{\calL,\beta}\sim T^{\calL,\alpha}$.
	In particular, $T^{\calL,\alpha}$ is a bounded trivialization.
\end{lemma}

\begin{proof}
	$\alpha$ and $\beta$ induce an isomorphism
	\begin{align*}
		\rsheaf{U\times_X V} \xrightarrow{\alpha^{-1}|_{U\times_X V}} \calL|_{U\times_X V} \xrightarrow{\beta|_{U\times_X V}} \rsheaf{U\times_X V}.
	\end{align*}
	We write $f \in \rring{U\times_X V}^{\times}$ for the image of $1$ by the isomorphism.
	Then we have
	\begin{align*}
		\calZ(T^{\calL,\alpha},S^{\calL,\beta}) = f^{-1}df + \calZ(T,S)
	\end{align*}
	(see Definition \ref{def:BoundedTrivialization}).
	This shows the lemma.
\end{proof}

By the lemma, the following definition is well-defined.

\begin{definition}
	We denote by $\calB^\calL$ the equivalence class of $T^{\calL, \alpha}$.
\end{definition}

It is well-known that the functor
\begin{align*}
	\calL\otimes_{\rsheaf{X}}(\cdot) \colon \Mod_h(\Dsheaf_{X,\lambda})
	\rightarrow \Mod_h(\calL\otimes_{\rsheaf{X}}\Dsheaf_{X,\lambda}\otimes_{\rsheaf{X}} \calL^{\invshf})
\end{align*}
gives an equivalence of categories.
We denote by $\calL\otimes_{\rsheaf{X}}(\cdot)$ the direct product $\prod_{\lambda \in \Lambda}\calL\otimes_{\rsheaf{X}}(\cdot)$ of the functors by abuse of notation.

\begin{proposition}
	The functor $\calL\otimes_{\rsheaf{X}}(\cdot)$ preserves the uniform boundedness, that is, we have functors
	\begin{align*}
		\calL\otimes_{\rsheaf{X}}(\cdot) \colon& \Mod_{ub}(\Dsheaf_{X,\Lambda}, \calB)
		\rightarrow \Mod_{ub}(\Dsheaf_{X,\Lambda}^\calL, \calB^\calL), \\
		\calL\otimes_{\rsheaf{X}}(\cdot) \colon& D^b_{ub}(\Dsheaf_{X,\Lambda}, \calB)
		\rightarrow D^b_{ub}(\Dsheaf_{X,\Lambda}^\calL, \calB^\calL).
	\end{align*}
	Moreover, the two functors give equivalences of categories.
\end{proposition}

\begin{proof}
	Since the assertion is local for $X$, we can assume $\calL \simeq \rsheaf{X}$.
	In the case, the proposition is clear.
\end{proof}

Next we consider the family of the opposite algebras $\Dsheaf_{X,\lambda}^{\opalg}$.
Note that $(\cdot)^\opalg$ is a functor on the category of algebras of twisted differential operators on $X$.
We set $\Dsheaf_{X,\Lambda}^\opalg:= (\Dsheaf_{X,\lambda}^{\opalg})_{\lambda \in \Lambda}$.
We shall construct a bornology of $\Dsheaf_{X,\Lambda}^\opalg$ from the bornology $\calB$.

Recall that there is a canonical isomorphism
\begin{align}
	\ntDsheaf_{X}^\opalg \simeq \Omega_{X}\otimes_{\rsheaf{X}} \ntDsheaf_{X} \otimes_{\rsheaf{X}} \Omega_X^\invshf, \label{eqn:IsomOp}
\end{align}
where $\Omega_X$ is the canonical sheaf of $X$.
See \cite[Lemma 1.2.7]{HTT08}.
The isomorphism induces an automorphism of the space $\calZ(X)$ of closed $1$-forms as
\begin{align*}
	\calZ(X) \xrightarrow{\simeq}& \Aut(\ntDsheaf_X) \xrightarrow{(\cdot)^\opalg} \Aut(\ntDsheaf_X^\opalg)
	\xrightarrow{\simeq} \Aut(\Omega_{X}\otimes_{\rsheaf{X}} \ntDsheaf_{X} \otimes_{\rsheaf{X}} \Omega_X^\invshf) \\
	&\xrightarrow{\simeq} \Aut(\ntDsheaf_X) \xrightarrow{\simeq} \calZ(X).
\end{align*}
Here the fourth isomorphism comes from the isomorphism $\ntDsheaf_X \simeq \Omega_{X}^\invshf\otimes_{\rsheaf{X}}(\Omega_{X}\otimes_{\rsheaf{X}} \ntDsheaf_{X} \otimes_{\rsheaf{X}} \Omega_X^\invshf)\otimes_{\rsheaf{X}} \Omega_X$.

\begin{lemma}\label{lem:OpAutomorphism}
	The automorphism of $\calZ(X)$ is the multiplication map by $-1$.
\end{lemma}

\begin{proof}
	The lemma can be shown by an easy explicit computation.
\end{proof}

Remark that for an \'etale morphism $f\colon U\rightarrow X$, there are canonical isomorphisms
\begin{align*}
	f^*\Omega_X &\simeq \Omega_U, \\
	\Omega_f &\simeq \rsheaf{U}, \\
	f^\#(\Dsheaf_{X,\lambda}^\opalg) &\simeq (f^\#\Dsheaf_{X,\lambda})^\opalg.
\end{align*}
In particular, the canonical isomorphism \eqref{eqn:IsomOp} commutes with the pull-back by the \'etale morphism.

Take a bounded trivialization $T=(U,\varphi, \Phi) \in \calB$ such that $\Omega_X|_U\simeq \Omega_U$ is isomorphic to $\rsheaf{U}$.
Take a trivialization $\alpha\colon \Omega_U \xrightarrow{\simeq} \rsheaf{U}$.
Then $T$ and $\alpha$ induce an isomorphism $\Phi^{\opalg, \alpha}$ given by
\begin{align*}
	\ntDsheaf_{X,\Lambda}^{\opalg}|_U \xrightarrow{\Phi^\opalg} \ntDsheaf_{U}^\opalg \xrightarrow{\simeq} \Omega_U\otimes_{\rsheaf{U}} \ntDsheaf_{U}\otimes_{\rsheaf{U}} \Omega_U^{\invshf} \rightarrow \ntDsheaf_{U}.
\end{align*}
We obtain a trivialization $T^{\opalg, \alpha}=(U,\varphi, \Phi^{\opalg,\alpha})$.

\begin{lemma}
	Take $S=(V,\psi,\Psi) \in \calB$ and an isomorphism $\beta\colon \Omega_V \rightarrow \rsheaf{V}$.
	Then we have $S^{\opalg, \beta}\simeq T^{\opalg, \alpha}$.
	In particular, $T^{\opalg, \alpha}$ is bounded.
\end{lemma}

\begin{proof}
	As we have seen in the proof of Lemma \ref{lem:BoundedTwistedBornology}, there is $f \in \rring{U\times_X V}^\times$ such that
	\begin{align*}
		\calZ(T^{\opalg,\alpha}, S^{\opalg, \beta}) = f^{-1}df - \calZ(T,S).
	\end{align*}
	The sign before $\calZ(T,S)$ comes from Lemma \ref{lem:OpAutomorphism}.
	This shows the lemma.
\end{proof}

\begin{definition}
	We denote by $\calB^\opalg$ the equivalence class of $T^{\opalg,\alpha}$.
\end{definition}

Corresponding to canonical isomorphisms of algebras of twisted differential operators, there are identities of bornologies.
Let $\iota$ be the diagonal embedding $X\hookrightarrow X\times X$.
For two algebras $\Dsheaf_1$ and $\Dsheaf_2$ of twisted differential operators on $X$,
we set
\begin{align*}
	\Dsheaf_1 \tensorDshf \Dsheaf_2 := \iota^\#(\Dsheaf_1 \boxtimes \Dsheaf_2).
\end{align*}
We use the same notation for families of algebras.

Let $\Dsheaf_{X,\Lambda}, \DsheafB_{X,\Lambda}$ and $\DsheafC_{X,\Lambda}$
be families of algebras of twisted differential operators on $X$ with the same index set $\Lambda$.
Let $f\colon Y\rightarrow X$ be a morphism of smooth varieties.
For the constant family $\ntDsheaf_{X,\Lambda}:=(\ntDsheaf_{X})_{\lambda \in \Lambda}$, we consider a bounded trivialization $(X,\id_X, \id)$.
We denote by $\calB_\id$ the equivalence class of the trivialization.
We use the same notation for the constant family $\ntDsheaf_{Y,\Lambda}$ on $Y$.
Fix an invertible sheaf $\calL$ on $X$.

By \cite[\S 1]{KaTa96}, we have canonical isomorphisms
\begin{align*}
	(\Dsheaf_{X,\Lambda}^\calL)^\opalg &\simeq (\Dsheaf_{X,\Lambda}^\opalg)^{\calL^\invshf} \\
	\Dsheaf_{X,\Lambda} \tensorDshf \DsheafB_{X,\Lambda} &\simeq \DsheafB_{X,\Lambda} \tensorDshf \Dsheaf_{X,\Lambda}\\
	(\Dsheaf_{X,\Lambda} \tensorDshf \DsheafB_{X,\Lambda})\tensorDshf \DsheafC_{X,\Lambda} &\simeq \Dsheaf_{X,\Lambda} \tensorDshf (\DsheafB_{X,\Lambda}\tensorDshf \DsheafC_{X,\Lambda}) \\
	\ntDsheaf_{X,\Lambda} \tensorDshf \Dsheaf_{X,\Lambda} &\simeq \Dsheaf_{X,\Lambda} \\
	\Dsheaf_{X,\Lambda}^\opalg \tensorDshf \Dsheaf_{X,\Lambda} &\simeq \ntDsheaf_{X,\Lambda}^\opalg \\
	f^\# \ntDsheaf_{X,\Lambda} &\simeq \ntDsheaf_{Y,\Lambda} \\
	f^\# (\Dsheaf_{X,\Lambda}^\calL) &\simeq (f^\#\Dsheaf_{X,\Lambda})^{f^*\calL} \\
	f^\#(\Dsheaf_{X,\Lambda}^\opalg)^{\Omega_f} &\simeq (f^\#(\Dsheaf_{X,\Lambda}))^\opalg.
\end{align*}
Here we set $\Omega_f = f^{-1}\Omega_{X}^{\invshf}\otimes_{f^{-1}\rsheaf{X}}\Omega_{Y}$ (see Subsection \ref{subsect:DirectImage}).
Since the isomorphisms are canonical, they are natural in $\Dsheaf_{X,\Lambda}, \DsheafB_{X,\Lambda}$ and $\DsheafC_{X,\Lambda}$.
It is easy to see that the isomorphisms and the operations commute with the pull-back by the following cartesian square:
\begin{align*}
	\xymatrix{
		Y \ar[r]^-f & X \\
		U\times_X Y \ar[r]^-{\widetilde{f}} \ar[u]^-{\widetilde{\varphi}}& U \ar[u]^-{\varphi},
	}
\end{align*}
where $\varphi\colon U\rightarrow X$ is an \'etale morphism.
For example, the following diagram commutes:
\begin{align*}
	\xymatrix{
		f^\#(\Dsheaf_{X,\Lambda}^\opalg)^{\Omega_f}|_{U\times_X Y} \ar[r]^{\simeq} & (f^\#(\Dsheaf_{X,\Lambda}))^\opalg|_{U\times_X Y} \\
		\widetilde{f}^\#((\Dsheaf_{X,\Lambda}|_U)^\opalg)^{\Omega_{\widetilde{f}}} \ar[r]^{\simeq} \ar[u]^{\simeq} & (\widetilde{f}^\#(\Dsheaf_{X,\Lambda}|_U))^\opalg \ar[u]^\simeq.
	}
\end{align*}

\begin{proposition}\label{prop:IdentityBornology}
	Let $\calB_1, \calB_2$ and $\calB_3$ be bornologies of $\Dsheaf_{X,\Lambda}$,
	$\DsheafB_{X,\Lambda}$ and $\DsheafC_{X,\Lambda}$.
	Under the above identifications, we have
	\begin{enumerate}[(i)]
		\item $(\calB_1^\calL)^\opalg = (\calB_1^\opalg)^{\calL^\invshf}$,
		\item $\calB_1 \tensorBorn \calB_2 = \calB_2 \tensorBorn \calB_1$,
		\item $(\calB_1 \tensorBorn \calB_2) \tensorBorn \calB_3 = \calB_1 \tensorBorn (\calB_2 \tensorBorn \calB_3)$,
		\item $\calB_{\id} \tensorBorn \calB_1 = \calB_1$,
		\item $\calB^\opalg_1 \tensorBorn \calB_1 = \calB^\opalg_{\id}$,
		\item $f^\#\calB_\id = \calB_\id$,
		\item $f^\#(\calB_1^\calL) = (f^\# \calB_1)^{\calL^\invshf}$,
		\item $f^\#(\calB_1^\opalg)^{\Omega_f} = (f^\#(\calB))^\opalg$.
	\end{enumerate}
\end{proposition}

\begin{proof}
	The proposition is clear by the constructions of bornologies and the naturality of the canonical isomorphisms as mentioned above.
\end{proof}

\subsection{Integral transform}

We consider integral transforms of $\ntDsheaf$-modules.
Let $\Dsheaf_{X,\Lambda}$ and $\Dsheaf_{Y,\Lambda}$ be families of algebras of twisted differential operators on smooth varieties $X$ and $Y$ with the same index set $\Lambda$, respectively.
Fix bornologies $\calB_X$ and $\calB_Y$ of $\Dsheaf_{X,\Lambda}$ and $\Dsheaf_{Y,\Lambda}$, respectively.

Recall that there is a canonical isomorphism
\begin{align*}
	\Dsheaf_{X,\lambda}^\opalg \tensorDshf \Dsheaf_{X,\lambda} \simeq \ntDsheaf_X^\opalg.
\end{align*}
We write $\iota\colon X\rightarrow X\times X$ for the diagonal embedding.
The isomorphism is induced from the action of $\ntDsheaf_X^\opalg$ on $\iota^*(\Dsheaf_{X,\lambda}^\opalg \boxtimes \Dsheaf_{X,\lambda}) \simeq \Dsheaf_{X,\lambda}\otimes_{\rsheaf{X}}\Dsheaf_{X,\lambda}$ given by
\begin{align}
	Z\cdot (A\otimes B) = A\widetilde{Z}\otimes B - A\otimes \widetilde{Z}B \label{eqn:ActionTX}
\end{align}
for $Z \in \calT_X (\subset \ntDsheaf_{X}^\opalg)$ and $A,B\in \Dsheaf_{X,\lambda}$, where $\widetilde{Z}$
is a section of $\calP(\Dsheaf_{X,\lambda})$ such that $\sigma(\widetilde{Z}) = Z$.
See Definition \ref{def:Picard} for the notation of Picard algebroids.

\begin{lemma}\label{lem:AssociativeTensor}
	Fix $\lambda \in \Lambda$.
	For any $\calA \in \Mod_{qc}(\ntDsheaf_{X}^\opalg)$, $\calB \in \Mod_{qc}(\Dsheaf_{X,\lambda}^\opalg)$ and $\calC \in \Mod_{qc}(\Dsheaf_{X,\lambda})$, we have
	\begin{align*}
		\calA \otimes_{\ntDsheaf_{X}} (\Omega_X^\invshf \otimes_{\rsheaf{X}}(\calB \otimes_{\rsheaf{X}} \calC)) \simeq (\calA \otimes_{\rsheaf{X}} \Omega_X^\invshf \otimes_{\rsheaf{X}}\calB) \otimes_{\Dsheaf_{X,\lambda}} \calC
	\end{align*}
	as sheaves on $X$.
\end{lemma}

\begin{proof}
	Both sides of the expression can be regarded as the sheaves of $\calT_X$-coinvariants in
	$\calA \otimes_{\rsheaf{X}} \Omega_X^\invshf \otimes_{\rsheaf{X}}\calB \otimes_{\rsheaf{X}} \calC$.
	It is easy to see that the two actions of $\calT_X$ coincide by using the action of Picard algebroids (see \eqref{eqn:ActionTX}).
\end{proof}

\begin{theorem}\label{thm:UniformlyBoundedIntegralTransform}
	Let $\calM \in D^b_{ub}(\Dsheaf_{X,\Lambda}, \calB_X)$
	and $\calN \in D^b_{ub}(\Dsheaf_{Y,\Lambda}\boxtimes \Dsheaf_{X,\Lambda}^\opalg, \calB_Y\boxtimes \calB^\opalg_X)$.
	Then we have
	\begin{align*}
		(Rq_*(\calN_\lambda \otimes^L_{p^{-1}\Dsheaf_{X,\lambda}} p^{-1}\calM_\lambda))_{\lambda \in \Lambda} \in D^b_{ub}(\Dsheaf_{Y,\Lambda}, \calB_Y),
	\end{align*}
	where $p$ (resp.\ $q$) is the projection from $Y\times X$ onto $X$ (resp.\ $Y$).
\end{theorem}

\begin{proof}
	Fix $\lambda \in \Lambda$.
	Then we have
	\begin{align*}
		&Dq_+(p^*\Omega_X^\invshf \otimes^L_{\rsheaf{Y\times X}}(\calN_{\lambda}\otimes^L_{\rsheaf{Y\times X}} Lp^*\calM_{\lambda}))\\
		\simeq &Rq_*((\Dsheaf_{Y,\lambda}\boxtimes \Omega_X) \otimes_{\Dsheaf_{Y,\lambda}\boxtimes \ntDsheaf_{X}}^L (p^*\Omega_X^\invshf \otimes^L_{\rsheaf{Y\times X}}(\calN_{\lambda}\otimes^L_{\rsheaf{Y\times X}} Lp^*\calM_{\lambda}))) \\
		\simeq &Rq_*(p^{-1}\Omega_X \otimes_{p^{-1}\ntDsheaf_{X}}^L (p^{-1}\Omega_X^\invshf \otimes^L_{p^{-1}\rsheaf{X}}(\calN_{\lambda}\otimes^L_{p^{-1}\rsheaf{X}} p^{-1}\calM_{\lambda}))) \\
		\simeq &Rq_*((p^{-1}(\Omega_X \otimes_{\rsheaf{X}} \Omega_X^\invshf) \otimes^L_{p^{-1}\rsheaf{X}}\calN_{\lambda})\otimes^L_{p^{-1}\Dsheaf_{X,\lambda}} p^{-1}\calM_{\lambda}) \\
		\simeq &Rq_*(\calN_\lambda \otimes^L_{p^{-1}\Dsheaf_{X,\lambda}} p^{-1}\calM_\lambda).
	\end{align*}
	The third isomorphism follows from Lemma \ref{lem:AssociativeTensor} by
	taking a flat resolution of $\Omega_X$, $\calN_\lambda$ and $\calM_\lambda$.
	By Proposition \ref{prop:IdentityBornology}, we have
	\begin{align*}
		((\calB_Y\boxtimes \calB^\opalg_X)\tensorBorn p^\#\calB_X)^{p^*\Omega_X^\invshf}
		= (\calB_Y\boxtimes \calB_\id^\opalg)^{\rsheaf{Y}\boxtimes \Omega_X^\invshf} = \calB_Y\boxtimes \calB_\id = q^\# \calB_Y.
	\end{align*}
	Therefore the theorem follows from Theorem \ref{thm:FunctorOnUniformlyBounded}.
\end{proof}

\subsection{Family of easy morphisms}

Retain the notation $X$, $Y$, $\Dsheaf_{X,\Lambda}$, $\Dsheaf_{Y,\Lambda}$, $\calB_X$ and $\calB_Y$ in the previous subsection.
We consider operations of $\ntDsheaf$-modules by the following family of morphisms:
\begin{align*}
	&f_y\colon X\rightarrow X\times Y \quad (y \in Y), \\
	&f_y(x) = (x, y).
\end{align*}

\begin{proposition}\label{prop:FamilyMorphismPull}
	For any $\calM \in D^b_{ub}(\Dsheaf_{X,\Lambda}\boxtimes \Dsheaf_{Y,\Lambda}, \calB_X\boxtimes \calB_Y)$, the family $(Lf_y^*(\calM_\lambda))_{\lambda\in \Lambda, y \in Y}$ is uniformly bounded with respect to the bornology $\calB_X$.
\end{proposition}

\begin{proof}
	It is enough to show the assertion for $\calM \in \Mod_{ub}(\Dsheaf_{X,\Lambda}\boxtimes \Dsheaf_{Y,\Lambda}, \calB_X\boxtimes \calB_Y)$.
	Fix $\lambda \in \Lambda$ and $y \in Y$.

	Take $(U, \varphi, \Phi) \in \calB_X$ and $(V, \psi, \Psi) \in \calB_Y$
	with affine $U$ and $V$.
	Fix closed embeddings $\iota_U\colon U\rightarrow \CC^n$ and $\iota_V\colon V\rightarrow \CC^m$, and $y'\in \psi^{-1}(y)$.
	Then we have a commutative diagram
	\begin{align*}
		\xymatrix{
			\CC^n \ar[r]^-{f''_y} & \CC^n \times \CC^m \\
			U \ar[r]^-{f'_y}\ar[u]^-{\iota_U} \ar[d]_-\varphi & U\times V \ar[u]_-{\iota_U\times \iota_V} \ar[d]^-{\varphi \times\psi}\\
			X \ar[r]^-{f_y}& X\times Y,
		}
	\end{align*}
	where $f'_y(x) = (x, y')$ and $f''_y(x) = (x, \iota_V(y'))$.
	Remark that the upper square is cartesian.

	By Facts \ref{fact:FundamentalDmodule} (ii) and \ref{fact:BaseChange}, we have
	\begin{align*}
		D(\iota_U)_+ \circ L\varphi \circ Lf_y^*(\calM_\lambda) \simeq L(f''_y)^*\circ D(\iota_U\times \iota_V)_+ \circ L(\varphi \times \psi)^*(\calM_\lambda).
	\end{align*}
	Since the degree of $f''_y$ is $1$, we have
	\begin{align*}
		m_{\iota_U}(L\varphi \circ Lf_y^*(\calM_\lambda)) \leq m_{\iota_U\times \iota_V}(L(\varphi \times \psi)^*(\calM_\lambda))
	\end{align*}
	by Fact \ref{fact:WeylAlgebraDirectInverseImage}.
	This shows the proposition.
\end{proof}

\begin{proposition}\label{prop:FamilyMorphismPush}
	Let $\calN \in D^b_{ub}(\Dsheaf_{X,\Lambda}, \calB_X)$.
	The family $(D(f_y)_+(\calN_\lambda))_{\lambda\in \Lambda, y \in Y}$ is uniformly bounded with respect to the bornology $\calB_X\boxtimes \calB_Y$.
\end{proposition}

\begin{proof}
	We retain the notation in the proof of Proposition \ref{prop:FamilyMorphismPull}.
	Then we have
	\begin{align*}
		L(\varphi \times \psi)^*\circ D(f_y)_+(\calN_\lambda) \simeq L\varphi^*(\calN_\lambda)\boxtimes D\iota_+(\rsheaf{\psi^{-1}(y)}),
	\end{align*}
	where $\iota\colon \psi^{-1}(y) \rightarrow V$ is the inclusion map.
	By Fact \ref{fact:MultiplicityTensor}, we have
	\begin{align*}
		m_{\iota_U\times \iota_V}(L\varphi^*(\calN_\lambda)\boxtimes D\iota_+(\rsheaf{\psi^{-1}(y)})) &= m_{\iota_U}(L\varphi^*(\calN_\lambda))
		m_{\iota_V}(D\iota_+(\rsheaf{\psi^{-1}(y)})).
	\end{align*}
	Since the multiplicity of the unique irreducible holonomic $\ntDsheaf_{\CC^m}$-module supported on a point is $1$, we have 
	\begin{align*}
		m_{\iota_V}(D\iota_+(\rsheaf{\psi^{-1}(y)})) = |\psi^{-1}(y)|.
	\end{align*}
	Since $\psi$ is \'etale, $|\psi^{-1}(y)|$ is bounded on $Y$.
	Therefore we have shown the proposition.
\end{proof}

%% file: examples.tex
\section{Examples of uniformly bounded family}\label{sect:ExampleUBF}

In general, it is not easy to construct bornologies and uniformly bounded families of twisted $\ntDsheaf$-modules.
An easy way to construct them is to use group actions.
In this section, we construct uniformly bounded families using 
principal bundles and group actions with finite orbits.

\subsection{Bornology of a principal bundle}\label{sect:BornologyPrincipalBundle}

Let $G$ be an affine algebraic group and $p\colon \widetilde{X}\rightarrow X$ a principal $G$-bundle over a smooth variety $X$.
Let $\Dsheaf_{\widetilde{X}}$ be a $G$-equivariant algebra of twisted differential operators.
For each $\lambda \in (\lie{g}^*)^G$, we have defined a $G$-equivariant algebra $\Dsheaf_{X,\lambda}$
of twisted differential operators on $X$ in (\ref{eqn:DefinitionTwisted}).
Then we obtain a family $(\Dsheaf_{X,\lambda})_{\lambda \in (\lie{g}^*)^G}$.
Put $\Lambda := (\lie{g}^*)^G$.

In this subsection, we shall show that the family admits a standard bornology
determined by the bundle $\widetilde{X}\rightarrow X$.

We can take a surjective \'etale morphism $\varphi\colon U\rightarrow X$ such that
the pull-back $p\colon U \times_{X} \widetilde{X} \rightarrow U$ of the $G$-bundle
is trivial and $\Dsheaf_{\widetilde{X}}|_{U\times_X \widetilde{X}}$
is isomorphic to the algebra $\ntDsheaf_{U\times_X \widetilde{X}}$.
Fix a section $s\colon U\rightarrow U\times_X \widetilde{X}$
and an isomorphism $\alpha\colon s^\#(\Dsheaf_{\widetilde{X}}|_{U\times_X \widetilde{X}}) \rightarrow \ntDsheaf_{U}$.

The section $s$ determines a trivialization $U\times G \simeq U\times_X \widetilde{X}$
and $\alpha$ induces an isomorphism $\Dsheaf_{\widetilde{X}}|_{U\times_X \widetilde{X}}\simeq \ntDsheaf_U\boxtimes \ntDsheaf_G$ by Proposition \ref{prop:DsheafLocal}.
Then we have an isomorphism
\begin{align*}
	\Phi^{U, s,\alpha}_\lambda\colon \Dsheaf_{X,\lambda}|_{U} = p_*(\Dsheaf_{\widetilde{X}}/R(I_{-\lambda+\delta})\Dsheaf_{\widetilde{X}})^G|_U \xrightarrow{\simeq} \ntDsheaf_{U}
\end{align*}
for any $\lambda \in \Lambda$.
See \eqref{eqn:DefinitionTwisted} for the notation.
Hence we obtain a trivialization $(U, \varphi, \Phi^{U, s,\alpha})$ of $\Dsheaf_{X,\Lambda}$.

\begin{proposition}\label{prop:BornologyPrincipalBundle}
	$(U, \varphi, \Phi^{U, s,\alpha})$ is bounded and its equivalence class does not depend on the choice of $\varphi\colon U\rightarrow X$, $s$ and $\alpha$.
\end{proposition}

\begin{proof}
	Let $(\psi\colon V\rightarrow X, t, \beta)$ be another choice of $(\varphi, s, \alpha)$.
	By considering the pull-back of $s, \alpha, t, \beta, \Phi^{s,\alpha}$ and $\Phi^{t,\beta}$ to $U\times_X V$, our computation can be done only on $U\times_X V$.
	Hence we can assume $U=V=X$, $\widetilde{X}=X\times G$ and $\Dsheaf_{\widetilde{X}} = \ntDsheaf_{\widetilde{X}} = \ntDsheaf_{X}\boxtimes \ntDsheaf_G$.
	
	We identify $s^\#(\ntDsheaf_{\widetilde{X}})$ and $t^\#(\ntDsheaf_{\widetilde{X}})$ with $\ntDsheaf_X$ by the canonical isomorphisms.
	Then $\alpha$ and $\beta$ are automorphisms of $\ntDsheaf_X$.
	Since $\alpha$ and $\beta$ are independent of $\lambda \in \Lambda$, the choice of $\alpha$ and $\beta$ does not affect the equivalence.
	Hence we can assume $\alpha = \beta = \id$.

	Fix $\lambda \in \Lambda$.
	By the decomposition $X\times G = s(X)G$, we have a monomorphism
	$\iota_s\colon \ntDsheaf_X \simeq \ntDsheaf_{s(X)} \rightarrow p_*(\ntDsheaf_{X\times G})^G$
	and the isomorphism $(\Phi^{U,s,\alpha})^{-1}\colon \ntDsheaf_X \rightarrow \ntDsheaf_{X,\lambda}$ factors through the monomorphism.
	We define $\iota_t$ similarly.
	Then $\Phi^{V,t,\beta}_\lambda\circ (\Phi^{U,s,\alpha}_\lambda)^{-1}$ is given by the following dot arrow:
	\begin{align*}
		\xymatrix{
			\ntDsheaf_X \simeq \ntDsheaf_{s(X)} \ar@{.>}[d]\ar[r]^-{\iota_s} & p_*(\ntDsheaf_{X\times G})^G \ar@{->>}[r]& \ntDsheaf_{X,\lambda} \ar@{=}[d] \\
			\ntDsheaf_X \simeq \ntDsheaf_{t(X)} \ar[r]^-{\iota_t} & p_*(\ntDsheaf_{X\times G})^G \ar@{->>}[r]& \ntDsheaf_{X,\lambda}.
		}
	\end{align*}

	We write $s(x) = (x, s'(x))$ ($x \in X$) and
	define an automorphism $a$ of $X\times G$ by $a(x, g) = (x, s'(x)g)$.
	For a local section $T \in \calT_X$, we denote by $T_s$
	the corresponding section of $\calT_{s(X)}$.

	There exist closed $1$-forms $\omega^s_1, \omega^s_2, \cdots, \omega^s_n$ $(n=\dim_\CC(\lie{g}))$ on $X$ such that for any local sections $T \in \calT_{X}$ and $f \in \rsheaf{X}\otimes \rring{G}$,
	\begin{align*}
		T_s f = ((a^*)^{-1}\circ T \circ a^*) f = Tf - \sum_i \omega_i^s(T) L(X_i)f,
	\end{align*}
	where $\set{X_i}_{i=1,2,\ldots, n}$ is a basis of $\lie{g}$
	and $L$ is the differential of the left translation on $G$.
	Similarly, we define $\set{\omega_j^t}$ for $t$.
	Therefore $\set{\Phi^{V,t,\beta}_\lambda\circ (\Phi^{U,s,\alpha}_\lambda)^{-1}}_{\lambda \in \Lambda}$ is contained in a finite-dimensional subspace spanned by $\set{\omega_i^s}$ and $\set{\omega_j^t}$ in $\calZ(X)$.
	We have proved the proposition.
\end{proof}

\begin{definition}
	We denote by $\calB(X, \widetilde{X})$ the equivalence class of the bounded trivialization $(U, \varphi, \Phi^{U,s,\alpha})$.
\end{definition}

Let $f\colon Y\rightarrow X$ be a morphism of smooth varieties.
Then we have a cartesian square
\begin{align*}
	\xymatrix{
		Y\times_X \widetilde{X} \ar[r]^-{\widetilde{f}} \ar[d]^-{q}&\widetilde{X} \ar[d]^-p\\
		Y \ar[r]^-f & X.
	}
\end{align*}
Put
\begin{align*}
	\widetilde{Y}&:= Y\times_X \widetilde{X}, \\
	\Dsheaf_{\widetilde{Y}} &:= \widetilde{f}^\#\Dsheaf_{\widetilde{X}}.
\end{align*}
It is easy to see that $q\colon \widetilde{Y}\rightarrow Y$ is a principal $G$-bundle
and $\Dsheaf_{\widetilde{Y}}$ is $G$-equivariant.
For each $\lambda \in \Lambda$, we can define an algebra $\Dsheaf_{Y,\lambda}$
of twisted differential operators on $Y$ as in \eqref{eqn:DefinitionTwisted}
and we have a canonical isomorphism
\begin{align*}
	\Dsheaf_{Y,\lambda} \simeq f^\#\Dsheaf_{X,\lambda}.
\end{align*}
We identify the two algebras by the isomorphism.
Then we obtain two bornologies $\calB(Y,\widetilde{Y})$
and $f^\#\calB(X,\widetilde{X})$ of $\Dsheaf_{Y, \Lambda} := (\Dsheaf_{Y,\lambda})_{\lambda \in \Lambda}$.

\begin{lemma}\label{lem:pull-backOfBornologyPrincipalBundle}
	$\calB(Y,\widetilde{Y})$ and $f^\#\calB(X,\widetilde{X})$
	are equal.
\end{lemma}

\begin{proof}
	It is clear from the definition of $\calB(X,\widetilde{X})$ and its pull-back (Definition \ref{def:pull-backBornology}).
\end{proof}

The following Theorem is a consequence of Lemma \ref{lem:pull-backOfBornologyPrincipalBundle} and
Theorem \ref{thm:FunctorOnUniformlyBounded}.

\begin{theorem}\label{thm:FunctorUniformlyBoundedPrincipalBundle}
	We have functors
	\begin{align*}
		Df_+\colon D^b_{ub}(\Dsheaf_{Y,\Lambda}, \calB(Y,\widetilde{Y})) \rightarrow D^b_{ub}(\Dsheaf_{X,\Lambda}, \calB(X,\widetilde{X})), \\
		Lf^*\colon D^b_{ub}(\Dsheaf_{X,\Lambda}, \calB(X,\widetilde{X}))\rightarrow D^b_{ub}(\Dsheaf_{Y,\Lambda}, \calB(Y,\widetilde{Y})),
	\end{align*}
	which are the restrictions of the direct image functor and the inverse image functor, respectively.
\end{theorem}

\begin{corollary}\label{cor:FamilyOfUniformlyBounded}
	Let $\calM \in D_h^b(\Dsheaf_{\widetilde{X}})$.
	Then the family $(Dp_{+,\lambda}(\calM))_{\lambda \in \Lambda}$ is
	uniformly bounded with respect to $\calB(X,\widetilde{X})$.
\end{corollary}

\begin{proof}
	We shall apply Theorem \ref{thm:FunctorUniformlyBoundedPrincipalBundle} to $Y=\widetilde{X}$ and $f = p$.
	The fiber product $\widetilde{X}\times_X \widetilde{X}$ is canonically isomorphic to the trivial bundle $\widetilde{X}\times G$.
	The isomorphism is given by $\widetilde{X}\times G \ni (x, g) \mapsto (x, xg) \in \widetilde{X}\times_X \widetilde{X}$.
	Hence the following diagram is a cartesian square:
	\begin{align*}
		\xymatrix{
			\widetilde{X} \times G \ar[r]^-m \ar[d]^{\pr}&\widetilde{X} \ar[d]^-p\\
			\widetilde{X} \ar[r]^-p & X,
		}
	\end{align*}
	where $m$ is the multiplication map and $\pr$ is the projection onto the first factor.
	Then the constant family $(\calM)_{\lambda \in \Lambda}$ is uniformly bounded with respect to $\calB(\widetilde{X}, \widetilde{X}\times G)$.
	Therefore the assertion follows from Theorem \ref{thm:FunctorUniformlyBoundedPrincipalBundle}.
\end{proof}

By Corollary \ref{cor:FamilyOfUniformlyBounded}, we can construct many uniformly bounded families of $\ntDsheaf$-modules parametrized by $(\lie{g}^*)^G$ using principal $G$-bundles.

\subsection{\texorpdfstring{$G$}{G}-equivariant bornology}

Let $X$ be a smooth $G$-variety of an affine algebraic group $G$, and $\Dsheaf_{X,\Lambda}$ be a family of $G$-equivariant algebras of twisted differential operators.
We write $\pi\colon G\times X\rightarrow X$ and $m\colon G\times X\rightarrow X$ for the projection and the multiplication map, respectively.
Since all $\Dsheaf_{X, \lambda}$ are $G$-equivariant, we have a canonical isomorphism
\begin{align*}
	\pi^\# \Dsheaf_{X,\Lambda} \simeq m^\# \Dsheaf_{X,\Lambda}.
\end{align*}
See \eqref{eqn:Gequivariant}.

\begin{definition}\label{def:EquivariantBornology}
	We say that a bornology $\calB$ of $\Dsheaf_{X,\Lambda}$ is $G$-equivariant if $\pi^\# \calB = m^\# \calB$ holds under the isomorphism $\pi^\# \Dsheaf_{X,\Lambda} \simeq m^\# \Dsheaf_{X,\Lambda}$.
\end{definition}

The following proposition is clear by the definition and Proposition \ref{prop:pull-backBornology}.

\begin{proposition}
	Let $f\colon Y\rightarrow X$ be a morphism of smooth $G$-varieties and $\calB$ a $G$-equivariant bornology of $\Dsheaf_{X,\Lambda}$.
	Then $f^\# \calB$ is $G$-equivariant.
\end{proposition}

We set $m_g:=m(g,\cdot)$ for $g \in G$.
Then $m_g$ is an automorphism of $X$.

\begin{proposition}\label{prop:TranslationByG}
	Let $\calM \in D^b_{ub}(\Dsheaf_{X,\Lambda}, \calB)$.
	Then $(Lm_g^*(\calM_\lambda))_{\lambda \in \Lambda, g\in G}$ is uniformly bounded with respect to $\calB$.
\end{proposition}

\begin{proof}
	For $g \in G$, let $f_g$ denotes the morphism $f_g\colon X\rightarrow G\times X$ defined by $f_g(x) = (g, x)$.
	Then we have $m_g = m \circ f_g$.
	Since $\calB$ is $G$-equivariant, $Lm^*(\calM)$ is uniformly bounded with respect to $\pi^\#\calB=m^\#\calB$.
	By Proposition \ref{prop:FamilyMorphismPull}, $(Lf_g^*\circ Lm^*(\calM_\lambda))_{\lambda \in \Lambda, g\in G}$ is uniformly bounded with respect to $\calB = f_g^\#\pi^\#\calB$.
	This shows the assertion.
\end{proof}

In Subsection \ref{sect:BornologyPrincipalBundle}, we have given a way to construct a bornology using a principal bundle.
We shall show that the bornology is $G$-equivariant if the bundle has $G$-equivariant structure.
Let $G$ and $T$ be affine algebraic groups and $p\colon \widetilde{X}\rightarrow X$ a principal $T$-bundle over a smooth variety $X$.
Suppose that $\widetilde{X}$ and $X$ are $G\times T$-varieties and $p$ is $G\times T$-equivariant.
Let $\Dsheaf_{\widetilde{X}}$ be a $G\times T$-equivariant algebra of twisted differential operators on $\widetilde{X}$.

Put $\Lambda :=(\lie{t}^*)^T$.
Then we have a family $\Dsheaf_{X,\Lambda} = (\Dsheaf_{X,\lambda})_{\lambda \in \Lambda}$ of $G\times T$-equivariant algebras, and its bornology $\calB(X,\widetilde{X})$ as in Subsection \ref{sect:BornologyPrincipalBundle}.

\begin{proposition}\label{prop:G-equivariantBornologyPrincipal}
	$\calB(X,\widetilde{X})$ is $G$-equivariant.
\end{proposition}

\begin{proof}
	Consider the following commutative diagram:
	\begin{align*}
		\xymatrix{
			\widetilde{X} \ar[d]^-{p} & G\times \widetilde{X} \ar[d]^-{\id\times p}\ar[l]_-{\pi}\ar[r]^-{m}& \widetilde{X} \ar[d]^-{p} \\
			X & G\times X \ar[l]_-{\pi}\ar[r]^-{m}& X,
		}
	\end{align*}
	where $\pi$ and $m$ are the projection and the multiplication map, respectively.
	Since $\Dsheaf_{\widetilde{X}}$ is $G$-equivariant, $\pi^\# \Dsheaf_{\widetilde{X}}$ and $m^\# \Dsheaf_{\widetilde{X}}$ are canonically isomorphic.
	We obtain a family $\Dsheaf_{G\times X, \Lambda}$ constructed from the principal $T$-bundle $G\times \widetilde{X}\rightarrow G\times X$.
	Then $\pi^\# \Dsheaf_{X,\Lambda}$ and $m^\# \Dsheaf_{X,\Lambda}$ are canonically isomorphic to $\Dsheaf_{G\times X, \Lambda} = \ntDsheaf_G\boxtimes \Dsheaf_{X,\Lambda}$.
	Under this identification, by Lemma \ref{lem:pull-backOfBornologyPrincipalBundle}, we have
	\begin{align*}
		\pi^\#\calB(X,\widetilde{X}) = \calB(G\times X, G\times \widetilde{X}) = m^\#\calB(X,\widetilde{X}).
	\end{align*}
	This implies that $\calB(X,\widetilde{X})$ is $G$-equivariant.
\end{proof}

We shall show the uniqueness of $G$-equivariant bornologies on a homogeneous variety.
Let $G$ and $H$ be affine algebraic group and its closed subgroup,
and $\ntDsheaf_G$ the algebra of non-twisted differential operators.
We write $p\colon G\rightarrow G/H$ for the natural projection.
Then we obtain a $G$-equivariant algebra $\ntDsheaf_{G/H,\lambda}$ of twisted differential operators on $G/H$ for any $\lambda \in (\lie{h}^*)^H$.
See \eqref{eqn:DefinitionTwisted}.

It is well-known that any $G$-equivariant algebra of twisted differential operators is canonically isomorphic to some $\ntDsheaf_{G/H,\lambda}$ (see \cite[Theorem 4.9.2]{Ka89}).
This is because it is generated by $\univ{g}$ and $\rsheaf{G/H}$.
Hence we consider a bornology of a family $\ntDsheaf_{G/H,\Lambda}:=(\ntDsheaf_{G/H,\Lambda(r)})_{r \in R}$ for $\Lambda\colon R\rightarrow (\lie{h}^*)^H$.

\begin{proposition}\label{prop:UniqueBornologyHomogeneous}
	There exists a unique $G$-equivariant bornology of $\ntDsheaf_{G/H,\Lambda}$.
\end{proposition}

\begin{proof}
	The existence is clear because $\calB(G/H, G)$ is a $G$-equivariant bornology of $\ntDsheaf_{G/H,\Lambda}$ by Proposition \ref{prop:G-equivariantBornologyPrincipal}.
	
	We shall show the uniqueness.
	Let $\calB$ be a $G$-equivariant bornology of $\ntDsheaf_{G/H,\Lambda}$.
	By Proposition \ref{prop:FundamentalBornology} (iv), it is enough to show $p^\# \calB = p^\# \calB(G/H,G)$.
	Let $\pi, m\colon G\times G/H\rightarrow G/H$ be the projection and the multiplication map, respectively, and $\iota\colon G\rightarrow G\times G/H$ a morphism given by $\iota(g) = (g, eH)$.
	Using the $G$-equivariant structure, we identify the following three families:
	\begin{align*}
		m^\#\ntDsheaf_{G/H, \Lambda},\quad \pi^\# \ntDsheaf_{G/H,\Lambda},\quad (\ntDsheaf_{G}\boxtimes \ntDsheaf_{G/H,\lambda(r)})_{r \in R}.
	\end{align*}

	Since $m\circ \iota = p$, by Proposition \ref{prop:pull-backBornology}, we have
	\begin{align*}
		p^\# \calB = \iota^\# m^\#\calB = \iota^\# \pi^\# \calB
		= \iota^\# (\calB_{\id}\boxtimes \calB) = \calB_{\id},
	\end{align*}
	where $\calB_{\id}$ is the equivalence class of the trivialization $(G, \id_G, \id)$ of the constant family $(\ntDsheaf_{G})_{r \in R}$.
	Therefore we have $p^\# \calB = \calB_{\id} = p^\# \calB(G/H,G)$.
\end{proof}

\subsection{Uniformly bounded family of irreducible modules}

Let $K$ be an affine algebraic group and $X$ a $K$-variety.
Let $\Dsheaf_{X,\Lambda}:=(\Dsheaf_{X,\lambda})_{\lambda\in \Lambda}$ be a family of $K$-equivariant algebras of twisted differential operators on $X$.
Fix a $K$-equivariant bornology $\calB$ of $\Dsheaf_{X,\Lambda}$.
A classification of $K$-equivariant $\Dsheaf_{X,\lambda}$-modules
is given by Beilinson--Bernstein \cite{BeBe81} (see also \cite[Theorem 2.4]{HMSW87}).

We review the classification.
Fix $\lambda \in \Lambda$ and $x \in X$.
We write $i\colon Kx \hookrightarrow X$ and $p\colon K\rightarrow Kx$ for the inclusion and the natural surjection, respectively.
Let $K_x$ denote the stabilizer of $x$ in $K$.
Since $i^\# \Dsheaf_{X,\lambda}$ is $K$-equivariant and $Kx$ is homogeneous,
there is a unique element $\mu(\lambda)$ of $(\lie{k}_x^*)^{K_x}$ such that
$i^\# \Dsheaf_{X,\lambda}$ is canonically isomorphic to
\begin{align*}
	\ntDsheaf_{Kx, \mu(\lambda)} := (p_*(\ntDsheaf_{K})\otimes_{\calU(\lie{k}_x)} \CC_{\mu(\lambda)})^{K_x},
\end{align*}
See \cite[Theorem 4.9.2]{Ka89}.
We identify $i^\# \Dsheaf_{X,\lambda}$ with $\ntDsheaf_{Kx, \mu(\lambda)}$ by the isomorphism.

\begin{fact}\label{fact:ClassificationDmodule}
	Let $\calM$ be an irreducible coherent $(\Dsheaf_{X,\lambda}, K)$-module whose support is $\overline{Kx}$.
	Then there exists a unique irreducible $K_x$-module $F$ such that
	\begin{enumerate}[(i)]
		\item $\lie{k}_x$ acts on $F$ by the character $\mu(\lambda)$,
		\item $\calM$ is isomorphic to the unique irreducible submodule
		of $D^0i_+(\Ind_{K_x}^K(F))$,
	\end{enumerate}
	where $\Ind_{K_x}^K(F)$ is the $(\ntDsheaf_{Kx,\mu(\lambda)}, K)$-module of local sections of the associated vector bundle $K\times_{K_x} F$ over $Kx \simeq K/K_x$.
	In particular, $\calM$ is holonomic.
\end{fact}

We shall show that a family of $(\Dsheaf_{X,\lambda}, K)$-modules with bounded lengths is uniformly bounded if $K$ has finitely many orbits in $X$.

\begin{lemma}\label{lem:KequivModules}
	Let $F$ be an irreducible $K_x$-module.
	Put $n=\dim_{\CC}(K_x)$.
	Assume that $\lie{k}_x$ acts on $F$ by the character $\mu(\lambda)$.
	Then $\Ind_{K_x}^K(F)$ is isomorphic to a direct summand of $D^{-n}p_{+,\mu(\lambda)}(\rsheaf{K})$.
\end{lemma}

\begin{proof}
	By Theorem \ref{thm:DirectImageTor} and the Poincar\'e duality (Fact \ref{fact:PoincareDuality}), we have
	\begin{align*}
		D^{-n}p_{+,\mu(\lambda)}(\rsheaf{K}) \simeq \Tor_n^{\calU(\lie{k}_x)}(\CC_{\mu(\lambda) - \delta}, p_*(\rsheaf{K}))
		\simeq (p_*(\rsheaf{K}) \otimes \CC_{\mu(\lambda)})^{(K_x)_0},
	\end{align*}
	where $\delta$ is the character $\lie{k}_x \ni X\mapsto \tr(\ad_{\lie{k}_x}(X))$.
	The assertion follows from the isomorphisms and the Frobenius reciprocity.
\end{proof}

\begin{lemma}\label{lem:KequivModules2}
	Let $\calM$ be an irreducible coherent $(\Dsheaf_{X,\lambda}, K)$-module whose support is $\overline{Kx}$.
	Then $\calM$ is isomorphic to a direct summand of $H^{-n} \circ Di_+\circ Dp_{+,\mu(\lambda)}(\rsheaf{K})$.
\end{lemma}

\begin{proof}
	Since $Kx$ is locally closed in $X$, the cohomology $D^{k}i_+(\calN)$ vanishes for any $k < 0$ and $\calN \in \Mod_{qc}(\Dsheaf_{X,\lambda})$.
	Using truncation functors (see Subsection \ref{sect:truncation}),
	we have
	\begin{align*}
		H^{-n} \circ Di_+\circ Dp_{+,\mu(\lambda)}(\rsheaf{K}) \simeq D^{0}i_+(D^{-n}p_{+,\mu(\lambda)}(\rsheaf{K})).
	\end{align*}
	Hence the assertion follows from Fact \ref{fact:ClassificationDmodule} and Lemma \ref{lem:KequivModules}.
\end{proof}

Let $\pi$ and $m$ be the projection and the multiplication map
from $K\times X$ to $X$, respectively.
We write $f_x(g) = (g, x)$ for $g \in K$ and $x \in X$.
Then we have $i\circ p = m\circ f_x$.
We denote by $D(f_x)_{+,\lambda}$ the direct image functor
$D^b_{qc}(\ntDsheaf_{G}) \rightarrow D^b_{qc}(\pi^\#\Dsheaf_{X,\lambda})$.
By Proposition \ref{prop:FamilyMorphismPush}, $(D(f_x)_{+, \lambda}(\rsheaf{K}))_{x \in X, \lambda \in \Lambda}$ is uniformly bounded with respect to $\pi^\# \calB$.

Since $\calB$ and any algebra in $\Dsheaf_{X,\Lambda}$ are $K$-equivariant,
we have $m^\# \Dsheaf_{X,\Lambda} \simeq \pi^\# \Dsheaf_{X,\Lambda}$
and $m^\#\calB = \pi^\# \calB$.
Therefore $(Dm_+ \circ D(f_x)_{+,\lambda}(\rsheaf{K}))_{x \in X, \lambda \in \Lambda}$ is uniformly bounded with respect to $\calB$ by Theorem \ref{thm:FunctorOnUniformlyBounded}.
By Lemma \ref{lem:KequivModules2} and $i\circ p = m\circ f_x$, we obtain

\begin{proposition}\label{prop:UniformlyBoundedIrreducibles}
	Let $\calM \in \prod_{\lambda \in \Lambda}\Mod_{h}(\Dsheaf_{X,\lambda}, K)$.
	Assume that each $\calM_\lambda$ is irreducible and its support is the closure of some $K$-orbit dependent on $\lambda$.
	Then $\calM$ is a uniformly bounded family with respect to $\calB$.
\end{proposition}

\begin{theorem}\label{thm:UniformlyBoundedIrreducibles}
	Let $\calM \in \prod_{\lambda \in \Lambda}\Mod_{h}(\Dsheaf_{X,\lambda}, K)$.
	Assume that $K$ has finitely many orbits in $X$ and the length of each $\calM_\lambda$ is bounded by a constant independent of $\lambda \in \Lambda$.
	Then $\calM$ is a uniformly bounded family with respect to $\calB$.
\end{theorem}

For the representation theory of real reductive Lie groups, we generalize the theorem to the universal covering group of $K$ in a sense.
Retain the notation $X, K, \Dsheaf_{X,\Lambda}, \calB$ as above
and assume that $K$ is connected.

Fix $\lambda \in \Lambda$ for a while.
Let $\nu$ be a character of $\lie{k}$ and $\calM$ a quasi-coherent $\Dsheaf_{X,\lambda}$-module.
We say that $\calM$ is a twisted $(\Dsheaf_{X,\lambda}, K)$-module with twist $\nu$ if the action of $\lie{k}$ on $\calM\otimes \CC_{\nu}$ lifts to an action of $K$.
Let $\Dsheaf_{X,(\lambda, \nu)}$ be the $K$-equivariant algebra $\Dsheaf_{X,\lambda}\otimes \End_{\CC}(\CC_{\nu})$, which is isomorphic to $\Dsheaf_{X,\lambda}$ without the $K$-equivariant structures.
Then $\calM$ is a twisted $(\Dsheaf_{X,\lambda}, K)$-module with twist $\nu$
if and only if $\calM$ admits a $K$-equivariant structure as an $\Dsheaf_{X,(\lambda, \nu)}$-module.

Take $(U, \varphi, \Phi) \in \calB$.
Then $(U, \varphi, (\Phi_\lambda)_{\lambda \in \Lambda, \nu \in (\lie{k}^*)^K})$
is a bounded trivialization of $(\Dsheaf_{X,(\lambda, \nu)})_{\lambda\in \Lambda, \nu \in (\lie{k}^*)^K}$.
Since the $K$-action on $\Dsheaf_{X,(\lambda, \nu)}$ is the same as that on $\Dsheaf_{X,\lambda}$, the bornology defined by $(U, \varphi, (\Phi_\lambda)_{\lambda \in \Lambda, \nu \in (\lie{k}^*)^K})$ is $K$-equivariant.

\begin{corollary}\label{cor:UniformlyBoundedIrreducibles}
	Proposition \ref{prop:UniformlyBoundedIrreducibles} and Theorem \ref{thm:UniformlyBoundedIrreducibles} hold even if all $\calM_\lambda$
	are twisted $(\Dsheaf_{X, \lambda}, K)$-modules.
\end{corollary}

\subsection{Finite orbits and uniformly bounded family}

Retain the notation $X, K, \Dsheaf_{X, \Lambda}, \calB$ in the previous subsection.
Assume that $K$ has finitely many orbits in $X$ and $K$ is connected.
In this subsection, we consider the $\Dsheaf_{X,\lambda}$-module $\Tor^{\univ{k}}_i(\Dsheaf_{X,\lambda}, F)$
for a finite-dimensional $\lie{k}$-module $F$.

To estimate the length of $\Tor^{\univ{k}}_i(\Dsheaf_{X,\lambda}, F)$, we need the following lemma about a complex of filtered modules.

\begin{lemma}\label{lem:FilteredCohomology}
	Let $\calA$ be a filtered ring and $(C^\cxdot, d^\cxdot)$ a complex of filtered $\calA$-modules.
	Then $\gr(H^i(C^\cxdot))$ is isomorphic to a subquotient of $H^i(\gr(C^\cxdot))$ for any $i \in \ZZ$.
\end{lemma}

\begin{proof}
	Fix $i \in \ZZ$.
	It is easy to see that the following canonical homomorphisms are injective:
	\begin{align}
		\Im(\gr(d^{i-1})) \rightarrow \gr(\Im(d^{i-1})) \rightarrow \gr(\Ker(d^i)) \rightarrow \Ker(\gr(d^i)), \label{eqn:ProofFilteredComplex}
	\end{align}
	where $\gr(d^k)\colon \gr(C^k)\rightarrow  \gr(C^{k+1})$ is the homomorphism induced from $d^k\colon C^k \rightarrow C^{k+1}$.
	The filtrations on $\Im(d^{i-1})$, $\Ker(d^{i})$ and $H^i(C^\cxdot)$
	are induced from that on $C^i$.
	Hence we have $\gr(H^i(C^\cdot)) \simeq \gr(\Ker(d^i))/\gr(\Im(d^{i-1}))$.
	This isomorphism and \eqref{eqn:ProofFilteredComplex} show the lemma.
\end{proof}

Let $\pi\colon T^*X\rightarrow X$ be the cotangent bundle.
We have a homomorphism $\sigma\colon S(\lie{k}) \rightarrow \rsheaf{T^*X}$
defined by taking the principal symbol of $\Dsheaf_{X,\lambda}$.
The homomorphism $\sigma$ does not depend on the choice of the $K$-equivariant algebra $\Dsheaf_{X,\lambda}$.
In fact, the composition $\lie{k}\rightarrow \calP(\Dsheaf_{X,\lambda})\rightarrow \calT_X$ coincides with the differential of the $K$-action on $X$,
and $\sigma$ is determined by $\sigma|_{\lie{k}}$.
Here $\calP(\Dsheaf_{X,\lambda})$ is the Picard algebroid associated to $\Dsheaf_{X,\lambda}$ (see Subsection \ref{sect:PicardAlgebroid}).

\begin{lemma}\label{lem:HolonomicBoundLength}
	Fix $\lambda \in \Lambda$.
	Let $\calM$ be an $\Dsheaf_{X,\lambda}$-module with a filtration,
	and $\calN$ a coherent $\pi_*\rsheaf{T^*X}$-module annihilated by $\sigma(\lie{k})$.
	If $\gr(\calM)$ is isomorphic to a subquotient of $\calN^{\oplus n}$,
	then $\calM$ is holonomic and there exists a constant $C(\calN)$ depending only on $\calN$ such that
	\begin{align*}
		\Len_{\Dsheaf_{X,\lambda}}(\calM) \leq C(\calN) \cdot n.
	\end{align*}
\end{lemma}

\begin{proof}
	Put $\widetilde{\calN}:=\rsheaf{T^*X}\otimes_{\pi^{-1}\pi_* \rsheaf{T^*X}}\calN$.
	Since $\sigma(\lie{k})$ annihilates $\widetilde{\calN}$ and $K$ has finitely many orbits in $X$, the support of $\widetilde{\calN}$ is contained in the union of the conormal bundles of all $K$-orbits in $X$.
	Since $\gr(\calM)$ is isomorphic to a subquotient of $\calN^{\oplus n}$, the filtration of $\calM$ is good, and hence $\calM$ is coherent by \cite[Theorem 2.1.3]{HTT08}.
	Moreover, the characteristic variety of $\calM$ is a union of the conormal bundles of some $K$-orbits in $X$.
	This shows that $\calM$ is holonomic.

	Let $C(\calN)$ be the sum of multiplicities of $\widetilde{\calN}$ along the conormal bundles of all $K$-orbits.
	Let $m(\calM)$ be the sum of multiplicities in the characteristic cycle of $\calM$.
	Then we have $m(\calM) \leq C(\calN)\cdot n$.
	Since the length of $\calM$ is bounded by $m(\calM)$ (see \cite[Proposition 5.1.9]{HTT08}), this shows the lemma.
\end{proof}

Let $U$ be the unipotent radical of $K$.
If necessary, replacing $K$ with its finite covering, we may assume that $[K/U, K/U]$ is simply-connected.

\begin{lemma}\label{lem:FiniteOrbitDmodule}
	There exists some constant $C > 0$ such that
	for any finite-dimensional $\lie{k}$-module $F$, $\lambda \in \Lambda$ and $i\in \ZZ$, we have
	\begin{align*}
		\Len_{\Dsheaf_{X,\lambda}}(\Tor^{\univ{k}}_i(\Dsheaf_{X,\lambda}, F))
		\leq C\cdot \dim_{\CC}(F),
	\end{align*}
	where $\Dsheaf_{X,\lambda}$ is considered  as a $\univ{k}$-module by the right action.
	Moreover, any composition factor of $\Tor^{\univ{k}}_i(\Dsheaf_{X,\lambda}, F)$ is a holonomic twisted $(\Dsheaf_{X,\lambda}, K)$-module.
\end{lemma}

\begin{remark}
	The lemma for $n=1$ is proved in \cite{Ta18}.
\end{remark}

\begin{proof}
	Fix $F$ and $\lambda$.
	By induction on the length of $F$, the assertion can be reduced to the case of irreducible $F$.
	Since $F$ is irreducible, $\lie{k}/\Ann_\lie{k}(F)$ is reductive, where $\Ann$ means the annihilator of a module.
	Hence we can take a character $\mu$ of $\lie{k}$ such that $F\otimes \CC_\mu$ lifts to a $K$-module.
	This implies that $\Tor^{\univ{k}}_i(\Dsheaf_{X,\lambda}, F)$ is a twisted $(\Dsheaf_{X,\lambda}, K)$-module with twist $\mu$.
	In fact, the homology can be computed by an $h$-complex of weak $(\Dsheaf_{X,(\lambda, \mu)}, K)$-modules in the sense of Bernstein--Lunts \cite[2.5]{BeLu95}.
	See \cite[Proposition 3.3]{Ki12} for the complex.
	Here $\Dsheaf_{X,(\lambda, \mu)}$ is a $K$-equivariant algebra defined before Corollary \ref{cor:UniformlyBoundedIrreducibles}.

	To compute $\Tor^{\univ{k}}_i(\Dsheaf_{X,\lambda}, F)$, we shall use the Chevalley--Eilenberg chain complex.
	See Fact \ref{fact:LiealgebraCohomology}.
	Let $(\Dsheaf_{X,\lambda}\otimes F \otimes \wedge^{-\cxdot} \lie{k}, d^\cxdot )$ be the complex.
	For any $i\geq 0$, the differential $d^{-i}$ is given by
	\begin{align}
		&d^{-i}(P\otimes f \otimes (X_1 \wedge X_2 \wedge \cdots \wedge X_i)) \nonumber \\
		= &\sum_{a} (-1)^{a+1} (PX_a\otimes f - P\otimes X_a f) \otimes X_1 \wedge X_2 \wedge \cdots \wedge \hat{X_a} \wedge \cdots \wedge X_i \nonumber \\
		+ &\sum_{a < b} (-1)^{a+b} P\otimes f \otimes [X_a, X_b]\wedge X_1 \wedge X_2 \wedge \cdots \wedge \hat{X_a} \wedge \cdots \wedge \hat{X_b} \wedge \cdots \wedge X_i. \nonumber \\
		\label{eqn:ProofChevalleyEilenberg}
	\end{align}

	We denote by $G$ the order filtration of $\Dsheaf_{X,\lambda}$.
	It induces a filtration $\widetilde{G}^i$ on $\Dsheaf_{X,\lambda}\otimes F \otimes \wedge^i \lie{k}$ as
	\begin{align*}
		\widetilde{G}^i_n(\Dsheaf_{X,\lambda}\otimes F \otimes \wedge^i \lie{k})
		= G_{n-i}(\Dsheaf_{X,\lambda})\otimes F\otimes \wedge^i \lie{k}
	\end{align*}
	for any $i \geq 0$.
	Then the complex $(\Dsheaf_{X,\lambda}\otimes F \otimes \wedge^{-\cxdot} \lie{k}, d^\cxdot )$ is a complex of filtered $\Dsheaf_{X,\lambda}$-modules.
	By \eqref{eqn:ProofChevalleyEilenberg}, we have
	\begin{align*}
		H^{-i}(\gr(\Dsheaf_{X,\lambda}\otimes F \otimes \wedge^{-\cxdot} \lie{k}))
		\simeq \Tor^{S(\lie{k})}_i(\pi_*\rsheaf{T^*X}, \CC) \otimes F
	\end{align*}
	as $\pi_* \rsheaf{T^*X}$-modules.
	Note that $\Tor^{S(\lie{k})}_i(\pi_*\rsheaf{T^*X}, \CC)$ is a coherent $\pi_*\rsheaf{T^*X}$-module because each term of the complex is coherent.

	By Lemma \ref{lem:FilteredCohomology}, $\gr(\Tor^{\univ{k}}_i(\Dsheaf_{X,\lambda}, F))$ is isomorphic to a subquotient of $\Tor^{S(\lie{k})}_i(\pi_*\rsheaf{T^*X}, \CC) \otimes F$.
	We can apply Lemma \ref{lem:HolonomicBoundLength} to $\calM=\Tor^{\univ{k}}_i(\Dsheaf_{X,\lambda}, F)$ and $\calN=\Tor^{S(\lie{k})}_i(\pi_*\rsheaf{T^*X}, \CC)$.
	Hence there is a constant $C_i$ depending only on $\Tor^{S(\lie{k})}_i(\pi_*\rsheaf{T^*X}, \CC)$ such that
	\begin{align*}
		\Len_{\Dsheaf_{X,\lambda}}(\Tor^{\univ{k}}_i(\Dsheaf_{X,\lambda}, F))
		\leq C_i \cdot \dim_{\CC}(F).
	\end{align*}
	$C:=\max_i\set{C_i}$ exists because $\Tor^{\univ{k}}_i(\cdot, \cdot)$ vanishes for any $i > \dim_\CC(\lie{k})$.
	The assertion in the lemma holds for this $C$.
\end{proof}

The following corollary is a direct consequence of Lemma \ref{lem:FiniteOrbitDmodule} and Corollary \ref{cor:UniformlyBoundedIrreducibles}.

\begin{corollary}\label{cor:UniformlyBoundedTor}
	Let $\calF$ be a set of $\lie{k}$-modules with bounded dimensions.
	Then the family $(\Tor^{\univ{k}}_{i}(\Dsheaf_{X,\lambda}, F))_{i \in \ZZ, F \in \calF, \lambda \in \Lambda}$ is uniformly bounded with respect to $\calB$.
\end{corollary}

Let $\Dsheaf_{Y, \Lambda}$ be a family of twisted differential operators on a smooth variety $Y$.
Fix a bornology $\calB'$ of $\Dsheaf_{Y, \Lambda}$.
We write $q\colon X\times Y\rightarrow Y$ for the projection onto the second factor.

\begin{theorem}\label{thm:UniformlyBoundedFiniteOrbits}
	Let $\calM \in \Mod_{ub}(\Dsheaf_{X,\Lambda}\boxtimes \Dsheaf_{Y,\Lambda}, \calB\boxtimes \calB')$.
	If all $\calM_{\lambda}$ are $q_*$-acyclic, then there exists a constant $C > 0$ such that
	\begin{align*}
		\Len_{\Dsheaf_{Y,\lambda}}(\Tor^{\univ{k}}_i(F, q_*(\calM_\lambda))) \leq C\cdot \dim_{\CC}(F)
	\end{align*}
	for any finite-dimensional $\lie{k}$-module $F$, $i \in \ZZ$ and $\lambda \in \Lambda$.
	Moreover, the family $(\Tor^{\univ{k}}_i(F, q_*(\calM_\lambda)))_{\lambda \in \Lambda, i \in \ZZ, F \in \calF}$ is uniformly bounded with respect to $\calB'$.
	Here $\calF$ is a set of finite-dimensional $\lie{k}$-modules whose dimensions are bounded.
\end{theorem}

\begin{proof}
	For $\calN \in D^b_h(\Dsheaf_{X,\lambda}^\opalg)$, put
	\begin{align*}
		T^i_\lambda(\calN) := R^iq_*(p^{-1}\calN\otimes^L_{p^{-1}\Dsheaf_{X,\lambda}} \calM_\lambda),
	\end{align*}
	where $p\colon X\times Y \rightarrow X$ is the projection onto the first factor.

	For $\lambda \in \Lambda$, let $I_\lambda$ be the set of all (isomorphism classes of) irreducible twisted $(\Dsheaf^\opalg_{X,\lambda}, K)$-modules.
	By Corollary \ref{cor:UniformlyBoundedIrreducibles}, the family $(\calN)_{\lambda \in \Lambda, \calN \in I_\lambda}$ is a uniformly bounded family with respect to $\calB$.
	Hence by Theorem \ref{thm:UniformlyBoundedIntegralTransform}, we can define a constant $C_1$ as
	\begin{align*}
		C_1 := \max\set{\Len_{\Dsheaf_{Y,\lambda}}(T^i_\lambda(\calN)): \lambda \in \Lambda, \calN \in I_\lambda, i\in \ZZ}.
	\end{align*}

	Fix a finite-dimensional $\lie{k}$-module $F$.
	Take a free resolution $J^\cxdot$ of the $\univ{k}$-module $F$.
	Then we have
	\begin{align*}
		T^{-i}_\lambda(J^\cxdot \otimes_{\univ{k}} \Dsheaf_{X,\lambda}) &= R^{-i}q_*(p^{-1}(J^\cxdot \otimes_{\univ{k}} \Dsheaf_{X,\lambda})\otimes^L_{p^{-1}\Dsheaf_{X,\lambda}} \calM_\lambda) \\
		&\simeq R^{-i} q_*(J^\cxdot \otimes_{\univ{k}}\calM_\lambda) \\
		&\simeq H^{-i}(J^\cxdot \otimes_{\univ{k}}q_*(\calM_\lambda)) \\
		&\simeq \Tor^{\univ{k}}_i(F, q_*(\calM_\lambda)).
	\end{align*}
	Here the second isomorphism holds because $J^k \otimes_{\univ{k}}\calM_\lambda$ is isomorphic to a direct sum of some copies of $\calM_\lambda$
	as a sheaf, and $\calM_\lambda$ is $q_*$-acyclic.
	This shows the second assertion by Theorem \ref{thm:UniformlyBoundedIntegralTransform} and Corollary \ref{cor:UniformlyBoundedTor}.
	
	To show the first assertion, we remark that $H^{-i}(J^\cxdot \otimes_{\univ{k}} \Dsheaf_{X,\lambda})$ is isomorphic to $\Tor^{\univ{k}}_i(F, \Dsheaf_{X,\lambda})$ as an $\Dsheaf_{X,\lambda}^\opalg$-module.
	By Lemma \ref{lem:AdditiveFunctionDerived} (ii), we have
	\begin{align*}
		\Len_{\Dsheaf_{Y,r}}(\Tor^{\univ{k}}_i(F, q_*(\calM_\lambda))) &= \Len_{\Dsheaf_{Y,r}}(T^{-i}_\lambda(J^\cxdot \otimes_{\univ{k}} \Dsheaf_{X,\lambda})) \\
		&\leq \sum_{j=0}^{\dim_{\CC}(\lie{k})} \Len_{\Dsheaf_{Y,\lambda}} (T^{-i+j}_\lambda(\Tor^{\univ{k}}_j(F, \Dsheaf_{X,\lambda}))).
	\end{align*}
	By Lemma \ref{lem:FiniteOrbitDmodule}, there is a constant $C_2$ independent of $F$ such that
	\begin{align*}
		\Len_{\Dsheaf_{X,\lambda}^\opalg}(\Tor^{\univ{k}}_j(F, \Dsheaf_{X,\lambda})) \leq C_2 \cdot \dim_{\CC}(F)
	\end{align*}
	for any $j \in \ZZ$ and $\lambda \in \Lambda$,
	and any composition factor of the module is in $I_\lambda$.
	By Lemma \ref{lem:AdditiveFunctionDerived} (i), we obtain
	\begin{align*}
		\Len_{\Dsheaf_{Y,\lambda}}(\Tor^{\univ{k}}_i(F, q_*(\calM_\lambda))) \leq C_1\cdot C_2 \cdot \dim_{\CC}(F) (\dim_{\CC}(\lie{k}) + 1).
	\end{align*}
	We have taken $C_1$ and $C_2$ independently of $F$, $i$ and $\lambda$.
	Therefore we have proved the theorem.
\end{proof}

%% file: zuckerman.tex
\section{Zuckerman derived functor and its localization}

In this section, we review the Zuckerman derived functors and their localization.
We use the functors to study the relative Lie algebra cohomology/homology.
The localization can be realized by a composition of direct image functors and inverse image functors.
Hence we can apply results about uniformly bounded families to study the functors and the cohomologies.

\subsection{Zuckerman functor}\label{sect:bernsteinFunctor}

In this subsection we review the Zuckerman derived functor.
We refer the reader to \cite[I.8]{BoWa00_continuous_cohomology} and \cite[6.3]{Wa88_real_reductive_I} for our construction.

Let $(\calA, G)$ be a generalized pair and $H$ a reductive subgroup of $G$.
Then $(\calA, H)$ forms a generalized pair and $(\lie{g}, H)$ forms a pair (see Definitions \ref{def:GeneralizedPair} and \ref{def:pair}).
Since $\calA$ is a $G$-module, for any $X \in \calA$, we can take $f_1, \ldots, f_n \in \rring{G}$
and $X_1, \ldots, X_n \in \calA$ such that
\begin{align*}
\Ad(g^{-1})(X) = \sum_i f_i(g)X_i
\end{align*}
for any $g \in G$.

Let $V$ be an $(\calA, H)$-module.
We define three actions on $\rring{G}\otimes V$ via
\begin{align*}
\mu(X)(f\otimes v) &= \sum_i f_i f \otimes X_i v, & (X \in \calA)\\
r(Y)(f\otimes v) &= R(Y)f \otimes v + f \otimes Yv, & (Y \in \lie{g})\\
r(g)(f\otimes v) &= R(g)f \otimes gv, & (g \in H)\\
l(g)(f\otimes v) &= L(g)f \otimes v & (g\in G)
\end{align*}
for $f \in \rring{G}$ and $v \in V$.
Here $L$ (resp.\ $R$) denotes the left (resp.\ right) regular action of $G$ on $\rring{G}$
and $f_i, X_i$ are the elements taken above for $X$.
It is easy to see that $\mu(X)$ does not depend on the choice of $\set{f_i}$ and $\set{X_i}$.
Note that the actions $\mu$ and $l$ commute with $r$
and we have
\begin{align*}
(\mu(X)-l(X))(\rring{G}\otimes V)^{r(\lie{g})} = 0
\end{align*}
for any $X \in \lie{g}$.
This implies that $\zuck{G}{H}(V) := (\rring{G}\otimes V)^{r(\lie{g}), r(H)}$ is an $(\calA, G)$-module via $\mu$ and $l$.

For an $(\calA, H)$-module $V$ and $i \in \NN$, we set
\begin{align*}
\Dzuck{G}{H}{i}(V):= H^i(\lie{g}, H; \rring{G}\otimes V),
\end{align*}
where $\rring{G}\otimes V$ is considered as a $(\lie{g}, H)$-module via the action $r$ to take the relative Lie algebra cohomology.
The two actions $l$ and $\mu$ satisfy the definition of $(\calA, G)$-modules (Definition \ref{def:AG-mod}), and hence $\Dzuck{G}{H}{i}(V)$ is an $(\calA, G)$-module.
See e.g.\ \cite[Proposition I.8.2]{BoWa00_continuous_cohomology} and \cite[Theorem 1.6]{MiPa98}.
Remark that we can prove that $\Dzuck{G}{H}{i}(V)$ is an $(\calA, G)$-module
under the weaker assumption that $G/H$ is affine without reductivity of $H$.
See Remark \ref{rmk:LocalZuckerman}.

\begin{fact}\label{fact:Zuckerman}
	$\Dzuck{G}{H}{i}(V)$ admits an $(\calA, G)$-module structure defined by $\mu$ and $l$.
	If, in addition, $\calA$ is flat as a right $\univ{g}$-module, then $\Dzuck{G}{H}{i}$ is isomorphic to the $i$-th right derived functor of $\zuck{G}{H}$.
\end{fact}

The functors $\Dzuck{G}{H}{i}$ are called the \define{Zuckerman derived functors}.

The following property is well-known and easy to see from the above isomorphism and the algebraic Peter--Weyl theorem (\cite[Theorem 4.2.7]{GoWa09}).
See e.g.\ \cite[Theorem I.8.8]{BoWa00_continuous_cohomology}.

\begin{fact}\label{fact:BernAndF}
	Let $V$ be an $(\calA, H)$-module.
	Assume that $G$ is reductive.
	For any $i \in \NN$, the irreducible decomposition of $\Dzuck{G}{H}{i}(V)$ as a $G$-module is given by
	\begin{align*}
		\Dzuck{G}{H}{i}(V) \simeq \bigoplus_F H^i(\lie{g}, H; F\otimes V)\otimes F^*,
	\end{align*}
	where the direct sum is over all isomorphism classes of irreducible $G$-modules.
	The isomorphism is natural in $V$.
\end{fact}

For a generalized pair $(\calA, G)$, we consider $(\calA\otimes \univ{g}, G)$ as a generalized pair
equipped with the diagonal homomorphism $\univ{g} \rightarrow \calA\otimes \univ{g}$
and the diagonal action of $G$ on $\calA\otimes \univ{g}$.

\begin{lemma}\label{lem:TorAndBern}
Let $V$ be an $(\calA, H)$-module.
Assume that $G$ is reductive.
Then for any $(\lie{g}, H)$-module $W$ and $i \in \NN$, there exists a natural isomorphism of $\calA^G$-modules
\begin{align*}
	\Dzuck{G}{H}{i}(V\otimes W)^G \simeq H^i(\lie{g}, H; V\otimes W),
\end{align*}
where $V\otimes W$ is considered as an $(\calA\otimes \univ{g}, H)$-module to apply the functor $\Dzuck{G}{H}{i}$.
\end{lemma}

\begin{proof}
	The isomorphism $\Dzuck{G}{H}{i}(V\otimes W)^G \simeq H^i(\lie{g}, H; V\otimes W)$ of vector spaces in Fact \ref{fact:BernAndF} is natural in $V$ and $W$.
	Hence the isomorphism is also an $\calA^G$-homomorphism.
\end{proof}

\begin{corollary}\label{cor:TorAndBernLength}
Retain the notation in Lemma \ref{lem:TorAndBern}.
Then for any $(\lie{g}, H)$-module $W$ and $i \in \NN$, we have
\begin{align*}
\Len_{\calA^G}(H^i(\lie{g}, H; V\otimes W)) \leq \Len_{\calA\otimes \univ{g},G}(\Dzuck{G}{H}{i}(V\otimes W)).
\end{align*}
\end{corollary}

\begin{proof}
By Proposition \ref{prop:AGandLength}, we have
\begin{align*}
\Len_{(\calA\otimes \univ{g})^G}(\Dzuck{G}{H}{i}(V\otimes W)^G) \leq \Len_{\calA\otimes \univ{g},G}(\Dzuck{G}{H}{i}(V\otimes W)).
\end{align*}
The action of $(\calA\otimes \univ{g})^G$ on $\Dzuck{G}{H}{i}(V\otimes W)^G$ factors through
$(\calA\otimes \univ{g})^G / (\calA\otimes \univ{g}\Delta(\lie{g}))^G$.
The inclusion map $\calA^G\otimes 1 \hookrightarrow (\calA\otimes \univ{g})^G$
induces an isomorphism
\begin{align*}
\calA^G\otimes 1 \xrightarrow{\simeq} (\calA\otimes \univ{g})^G / (\calA\otimes \univ{g}\Delta(\lie{g}))^G.
\end{align*}
The assertion therefore follows from Lemma \ref{lem:TorAndBern}.
\end{proof}

\subsection{Localization of the Zuckerman functor}\label{sect:LocalizationBern}

We review the localization of the Zuckerman derived functor.
We refer to \cite[II.4]{Bi90} and to \cite[4.6]{Ki12} for a conceptual treatment using the equivariant derived category
(see also \cite[3.7]{BeLu94}).

Let $G$ be an affine algebraic group over $\CC$ and $\Dsheaf_X$ a $G$-equivariant algebra of twisted differential operators on a smooth $G$-variety $X$.
Let $H$ be a reductive subgroup of $G$.
We construct an $(\Dsheaf_X, G)$-module from an $(\Dsheaf_X, H)$-module.

We consider the following diagram:
\begin{align*}
X \xleftarrow{\pi} G\times X \xrightarrow{a} G\times X \xrightarrow{q} G/H\times X \xrightarrow{\pi'} X,
\end{align*}
where $\pi$ and $\pi'$ are the projections, $a$ is the isomorphism given by $a(g, x) = (g, gx)$ and $q$ is the natural projection.
We consider the left two $X$ as $G\times H$-varieties letting $G$ act trivially and the others as $G\times H$-varieties letting $H$ act trivially.
We consider $G$ as a $G\times H$-variety via the left and right translations.
Then $\pi, a, q$ and $\pi'$ are $G\times H$-equivariant.
Since $\Dsheaf_X$ is $G$-equivariant, we have canonical isomorphisms
\begin{align*}
\pi^{\#}\Dsheaf_X \simeq \ntDsheaf_{G}\boxtimes \Dsheaf_{X} \simeq (\pi'\circ q\circ a)^\# \Dsheaf_X
\end{align*}
of $G\times G$-equivariant algebras (see Proposition \ref{prop:GGequivariant}).
In particular, they are $G\times H$-equivariant.

We set $n = \dim_{\CC}(\lie{h}), m = \dim_{\CC}(\lie{g}/\lie{h})$ and 
\begin{align*}
\DlocZuck{G}{H}{i}(\calM):=L_{m-i} \pi'_+(L_n q_+(a_{+}\pi^* (\calM))^{H/H_0}) \in \Mod_{qc}(\Dsheaf_X, G)
\end{align*}
for $\calM \in \Mod_{qc}(\Dsheaf_X, H)$ and $i \in \NN$.
Here $(\cdot)^{H/H_0}$ means taking the $H/H_0$-invariant part in a $H/H_0$-equivariant sheaf.
Since the functors to define $\DlocZuck{G}{H}{i}$ preserve holonomicity, we can replace $\Mod_{qc}$ by $\Mod_{h}$.

The functors $\DlocZuck{G}{H}{i}$ can be considered as a localization of the Zuckerman functors $\Dzuck{G}{H}{i}$ (see \cite[Theorem 4.4]{Bi90} and \cite[Proposition 4.17]{Ki12}).

\begin{proposition}\label{prop:commutativeBernSect}
Let $\calM$ be an object in $\Mod_{qc}(\Dsheaf_{X}, H)$.
If the global section functor $\sect\colon \Mod_{qc}(\Dsheaf_X) \rightarrow \Mod(\Dalg{X})$ is exact, then
there exists a natural isomorphism
\begin{align*}
\sect(\DlocZuck{G}{H}{i}(\calM)) \simeq \Dzuck{G}{H}{i}(\sect(\calM))
\end{align*}
of $(\Dalg{X}, G)$-modules for any $i \in \NN$.
\end{proposition}

\begin{proof}
Note that $G/H$ is affine by Matsushima's criterion \cite[Theorem 3.8]{Ti11}.
Fix $i \in \NN$.

Since $q\colon G\times X \rightarrow G/H\times X$ is a principal $H$-bundle,
$L_n q_+(\cdot)^{H/H_0}$ is isomorphic to $q_*(\cdot)^H$ by Theorem \ref{thm:DirectImageTor}.
Hence we have
\begin{align*}
	L_nq_+ a_+ \pi^*(\calM)^{H/H_0} \simeq q_*(\rsheaf{G}\boxtimes \calM)^H.
\end{align*}
The action of $\ntDalg{X}\subset \Dalg{G/H}\otimes \Dalg{X}$ on $\sect(q_*(\rsheaf{G}\boxtimes \calM)^H) \simeq (\rring{G}\otimes\sect(\calM))^H$ is given by
\begin{align}
	A\cdot f\otimes m = \sum_{i} f_i f\otimes A_i m \label{eqn:ActionOfDG}
\end{align}
where $\set{A_i}\subset \Dalg{X}$ and $\set{f_i} \subset \rring{G}$ are finite subsets satisfying $\sum_i f_i(g)A_i = \Ad(g^{-1})A$.
The $G$-action on $(\rring{G}\otimes\sect(\calM))^H$ is given by the left translation on $\rring{G}$.

Let $p\colon G/H\times X\rightarrow G/H$ be the projection and $q'\colon G\rightarrow G/H$
the natural projection.
Since $\pi'$ is projection, we can compute $L_{m-i} \pi'_+$ by the relative de Rham complex (see \cite[Lemma 1.5.27]{HTT08}).
Since $G/H$ is a homogeneous variety, the tangent sheaf $\calT_{G/H}$
is isomorphic to $(q'_*\rsheaf{G}\otimes \lie{g}/\lie{h})^H$.
Hence we can write the relative de Rham complex using Lie algebras as
\begin{align*}
	&\sect (L_{m-i} \pi'_+(q_*(\rsheaf{G}\boxtimes \calM)^H)) \\
	\simeq &H^{i-m}\circ \sect\circ R\pi'_*(p^{-1}(q'_*\rsheaf{G}\otimes \wedge^{m+\cxdot}(\lie{g}/\lie{h})^*)^H\otimes_{p^{-1}\rsheaf{G/H}} q_*(\rsheaf{G}\boxtimes \calM)^H) \\
	\simeq &H^{i}((\rring{G}\otimes \wedge^{\cxdot}(\lie{g}/\lie{h})^*)^H\otimes_{\rring{G/H}} (\rring{G}\otimes \sect(\calM))^H) \\
	\simeq &H^{i}((\rring{G}\otimes \wedge^{\cxdot}(\lie{g}/\lie{h})^* \otimes \sect(\calM))^H).
\end{align*}
Here the second isomorphism holds because $G/H$ is affine and $\sect$ is exact on $\Mod_{qc}(\Dsheaf_X)$, and the third isomorphism comes from the tensor product of the two locally free sheaves on the affine variety $G/H$.
The differentials in the above complexes are those induced from the relative de Rham complex.
By a straightforward computation, the complex $(\rring{G}\otimes \wedge^{\cxdot}(\lie{g}/\lie{h})^* \otimes \sect(\calM))^H$ is isomorphic to the complex $\Hom_H(\CE{\lie{g}}{H}{\cxdot}, \rring{G}\otimes \sect(\calM))$.
See Subsection \ref{sect:CEcomplex} for the Chevalley--Eilenberg chain complex $\CE{\lie{g}}{H}{\cxdot}$.
Therefore we have
\begin{align*}
	\sect (L_{m-i} \pi'_+(q_*(\rsheaf{G}\boxtimes \calM)^H))\simeq H^{i}(\lie{g}, H; \rring{G}\otimes \sect(\calM)) \simeq \Dzuck{G}{H}{i}(\sect(\calM)).
\end{align*}

As we have seen around \eqref{eqn:ActionOfDG}, under the isomorphism $\sect(\DlocZuck{G}{H}{i}(\calM))\simeq \sect(\DlocZuck{G}{H}{i}(\calM))$ of vector spaces, the $\Dalg{X}$-action and the $G$-action on $\sect(\DlocZuck{G}{H}{i}(\calM))$ coincide with those on $\Dzuck{G}{H}{i}(\sect(\calM))$ given in Fact \ref{fact:Zuckerman}.
We have therefore proved the proposition.
\end{proof}

\begin{remark}\label{rmk:LocalZuckerman}
	In the proof, we did not use the reductivity of $H$.
	In fact, one can define the Zuckerman functor $\Dzuck{G}{H}{i}(V)$ by the same way in the previous subsection for any $H$ if $G/H$ is affine.
\end{remark}

Let $G$ and $H$ be an affine algebraic group and its reductive subgroup, and $X$ a smooth $G$-variety.
Let $\Dsheaf_{X, \Lambda}:=(\Dsheaf_{X,\lambda})_{\lambda \in \Lambda}$ be a family of $G$-equivariant algebras of twisted differential operators on $X$.
Take a $G$-equivariant bornology $\calB$ of $\Dsheaf_{X,\Lambda}$.
See Definition \ref{def:EquivariantBornology}.

The functor $\DlocZuck{G}{H}{i}$ is defined by a composition of inverse image functors 
and direct image functors.
Hence $\DlocZuck{G}{H}{i}$ preserves the uniform boundedness.

\begin{theorem}\label{thm:UniformlyBoundedFamilyBernstein}
	Let $(\calM_\lambda)_{\lambda \in \Lambda}$ be a family of $(\Dsheaf_{X,\lambda}, H)$-modules.
	Suppose that $\calM$ is uniformly bounded with respect to $\calB$.
	Then $(\DlocZuck{G}{H}{i}(\calM_\lambda))_{i \in \ZZ, \lambda \in \Lambda}$ is uniformly bounded with respect to $\calB$.
\end{theorem}

\begin{proof}
	Recall the definition of the morphisms:
	\begin{align*}
		X \xleftarrow{\pi} G\times X \xrightarrow{a} G\times X \xrightarrow{q} G/H\times X \xrightarrow{\pi'} X.
	\end{align*}
	Since $\calB$ is $G$-equivariant, we have $\pi^\# \calB = (\pi'\circ q \circ a)^\# \calB$ by Definition \ref{def:EquivariantBornology}.
	The uniform boundedness is preserved by direct images, inverse images and taking subquotients by Proposition \ref{prop:FundamentalBoundeFamily} and Theorem \ref{thm:FunctorOnUniformlyBounded}.
	Hence we have proved the theorem.
\end{proof}

%% file: applications.tex
\section{Application to representation theory}\label{sect:applications}

In this section, we define a notion of uniformly bounded family of $\lie{g}$-modules.
A typical example is a family of Harish-Chandra modules with bounded lengths.
As an application of results about uniformly bounded families of $\ntDsheaf$-modules, we will show that the uniform boundedness of a family of $\lie{g}$-modules is preserved by several operations such as (cohomologically) parabolic induction and taking coinvariants.
We will also prove the boundedness of the lengths of $\univ{g}^{G'}$-modules, which is related to the branching problem and harmonic analysis.

\subsection{Uniformly bounded family of \texorpdfstring{$\lie{g}$}{g}-modules}\label{sect:BBcorrespondence}

In this subsection, we introduce the notion of uniformly bounded family of $\lie{g}$-modules.

Let $G$ be a connected reductive algebraic group and $B$ a Borel subgroup of $G$.
Fix a Levi decomposition $B=TU$, where $T$ is a maximal torus and $U$ is the unipotent radical of $B$.
Then the natural projection $p\colon G/U\rightarrow G/B$ is a principal $T$-bundle and $G$-equivariant.

We will reduce theorems about $\lie{g}$-modules to those about $\calD$-modules on $G/B$.
To do so, we review the Beilinson--Bernstein correspondence.
Let $\ntDsheaf_{G/U}$ be the algebra of non-twisted differential operators on $G/U$ equipped with the natural $G\times T$-equivariant structure.
For a character $\lambda$ of $\lie{t}$, we set
\begin{align*}
\ntDsheaf_{G/B,\lambda} := (\CC_{\lambda}\otimes_{\univ{t}}p_*\ntDsheaf_{G/U})^T
\end{align*}
and consider $\ntDsheaf_{G/B,\lambda}$ as a $G\times T$-equivariant algebra of twisted differential operators
as in Subsection \ref{sect:PrincipalBundle}.
Then $p^{\#}\ntDsheaf_{G/B,\lambda}$ is naturally isomorphic to $\ntDsheaf_{G/U}$.
Note that we can explicitly construct a bounded trivialization belonging to $\calB(G/B, G)$ using an open covering by the open Bruhat cell and its translations.

$\proots = \proots(\lie{g}, \lie{t})$ denotes the set of positive roots determined by $B$ and $T$.
We write $\rho$ for half the sum of the positive roots.
The following fact is called the Beilinson--Bernstein correspondence \cite{BeBe81}.

\begin{fact}\label{fact:BeilinsonBernstein}
Let $\lambda$ be a character of $\lie{t}$.
\begin{enumerate}[(i)]
	\item The homomorphism $\univ{g}\rightarrow \ntDalg{G/B,\lambda}(=\sect(\ntDsheaf_{G/B,\lambda}))$ is surjective and its kernel is equal to the minimal primitive ideal with infinitesimal character $\lambda-\rho$.
	\item If $\lambda - \rho$ is anti-dominant, then any quasi-coherent $\ntDsheaf_{G/B,\lambda}$-module $\calM$ is acyclic, i.e.\ $H^i(G/B, \calM)=0$ for any $i > 0$.
	Moreover, there exists a full subcategory $\Mod_{qc}^e(\ntDsheaf_{G/B, \lambda})$ of $\Mod_{qc}(\ntDsheaf_{G/B,\lambda})$ such that
	the global section functor
	\begin{align*}
		\sect\colon \Mod_{qc}^e(\ntDsheaf_{G/B,\lambda}) \rightarrow \Mod(\ntDalg{G/B,\lambda})
	\end{align*}
	gives an equivalence of categories.
	\item If $\lambda - \rho$ is regular and anti-dominant, then the global section functor $\sect\colon \Mod_{qc}(\ntDsheaf_{G/B,\lambda}) \rightarrow \Mod(\ntDalg{G/B,\lambda})$
	gives an equivalence of categories.
\end{enumerate}
\end{fact}

Motivated by the Beilinson--Bernstein correspondence and the definition of uniformly bounded family of $\ntDsheaf$-modules, we introduce the following definition.

\begin{definition}
	Let $(V_i)_{i \in I}$ be a family of $\lie{g}$-modules.
	We say that $(V_i)_{i \in I}$ is \define{uniformly bounded} if
	the following two conditions hold.
	\begin{enumerate}[(i)]
		\item The length of $V_i$ is bounded by a constant independent of $i \in I$.
		\item There exist a family $(\lambda(r))_{r \in R}$ of anti-dominant weights of $\lie{t}$ and a family $\calN \in \Mod_{ub}(\ntDsheaf_{G/B, \lambda(r)+\rho}, \calB(G/B,G))$ (see \ref{sect:BornologyPrincipalBundle}) such that any composition factor of any $V_i$ is isomorphic to some $\sect(\calN_r)$.
	\end{enumerate}
	We say that a family of $(\lie{g}, K)$-modules of a pair $(\lie{g}, K)$ is \define{uniformly bounded} if it is a uniformly bounded family of $\lie{g}$-modules.
\end{definition}

The uniform boundedness is preserved by several operations of $\lie{g}$-modules.
The following proposition is a direct consequence of Proposition \ref{prop:FundamentalBoundeFamily} and Theorem \ref{thm:TensorUniformlyBounded}.

\begin{proposition}\label{prop:FundamentailUniformlyBoundedGH}
	Let $\lie{g}$ and $\lie{h}$ be complex reductive Lie algebras.
	\begin{enumerate}[(i)]
		\item For a short exact sequence $0\rightarrow L \rightarrow M \rightarrow N \rightarrow 0$ in $\prod_{i \in I} \Mod(\lie{g})$,
		both $L$ and $N$ are uniformly bounded if and only if so is $M$.
		\item For any family $(\lambda_i)_{i \in I}$ of characters of $\lie{g}$
		and uniformly bounded family $(V_j)_{j \in J}$ of $\lie{g}$-modules,
		$(V_j\otimes \CC_{\lambda_i})_{i \in I, j \in J}$ is also uniformly bounded.
		\item For any set $\Phi$ of inner automorphisms of $\lie{g}$ and uniformly bounded family $(V_j)_{j \in J}$ of $\lie{g}$-modules, $(V_j^{\varphi})_{\varphi \in \Phi, j\in J}$ is also uniformly bounded.
		Here $V_j^{\varphi}$ is the $\lie{g}$-module defined by the composition $\lie{g}\xrightarrow{\varphi} \lie{g}\rightarrow \End_{\CC}(V_j)$.
		\item \label{enum:prop:UniformlyBoundedGH} Let $(V_i)_{i \in I}$ (resp.\ $(W_i)_{i \in I}$) be a family of $\lie{g}$-modules (resp.\ $\lie{h}$-modules).
		Then $(V_i\boxtimes W_i)_{i \in I}$ is a uniformly bounded family of $(\lie{g\oplus h})$-modules if and only if both $(V_i)_{i \in I}$ and $(W_i)_{i \in I}$ are uniformly bounded.
	\end{enumerate}	
\end{proposition}

\begin{proposition}\label{prop:FundamentalUniformlyBoundedGmod}
	Let $(\lie{g}, K)$ be a pair and $M$ a reductive subgroup of $K$.
	Let $(V_i)_{i \in I}$ be a uniformly bounded family of $(\lie{g}, M)$-modules.
	\begin{enumerate}[(i)]
		\item\label{enum:FundamentalZuckerman} $(\Dzuck{K}{M}{j}(V_i))_{i \in I, j \in \ZZ}$ is uniformly bounded.
		\item If $K$ is reductive, there exists a constant $C$ such that for any $i \in I$ and $j \in \ZZ$, we have
		\begin{align*}
			\Len_{\univ{g}^K}(H^j(\lie{k}, M; V_i)) \leq C.
		\end{align*}
		\item (ii) is also true if $(V_i)_{i \in I}$ is a uniformly bounded family of $(\lie{g}, \lie{m})$-modules and $H^j(\lie{k}, M; \cdot)$ is replaced by $H^j(\lie{k}, \lie{m}; \cdot)$.
		\item (ii) is also true if we replace $M$ with its covering $\widetilde{M}$.
	\end{enumerate}
\end{proposition}

\begin{proof}
	By definition, we can reduce the assertions to similar results about $\ntDsheaf$-modules on the flag variety.
	(i) follows from Theorem \ref{thm:UniformlyBoundedFamilyBernstein}.
	Taking the $K$-invariant part of (i), (ii) follows from Corollary \ref{cor:TorAndBernLength}.

	By the definition of the relative Lie algebra cohomology, we have
	\begin{align*}
		H^j(\lie{k}, \lie{m}; V_i) = H^j(\lie{k}, M_0; (V_i)_{M_0}),
	\end{align*}
	where $(V_i)_{M_0}$ is the sum of all $\lie{m}$-submodules in $V_i$ that can lift to $M_0$-modules.
	Hence (iii) follows from (ii).
	(iv) can be reduced to (iii) by
	\begin{align*}
		H^j(\lie{k}, \widetilde{M}; V_i) = H^j(\lie{k}, \lie{m}; V_i)^{\widetilde{M}/\widetilde{M}_0}.
	\end{align*}
	We have proved the proposition.
\end{proof}

Let $(\lie{g}, K)$ be a pair.
Then $K$ acts on the flag variety of $\lie{g}$, which is isomorphic to $G/B$.
Assume that $K$ has finitely many orbits in $G/B$.

\begin{proposition}\label{prop:TorUniformlyBoundedGmod}
	Let $\lie{h}$ be a complex reductive Lie algebra
	and $(M_i)_{i \in I}$ a uniformly bounded family of $(\lie{g\oplus h})$-modules.
	For any set $\calF$ of finite-dimensional $\lie{k}$-modules whose dimensions are bounded, the family $(\Tor^{\univ{k}}_{j}(F, M_i))_{i \in I, j \in \ZZ, F \in \calF}$ is a uniformly bounded family of $\lie{h}$-modules.
	Moreover, there exists a constant $C$ such that for any finite-dimensional $\lie{k}$-module $F$, $i \in I$ and $j \in \ZZ$, 
	\begin{align*}
		\Len_{\lie{h}}(\Tor^{\univ{k}}_{j}(F, M_i)) \leq C\cdot \dim_{\CC}(F).
	\end{align*}
\end{proposition}

\begin{proof}
	The proposition follows from Theorem \ref{thm:UniformlyBoundedFiniteOrbits}.
\end{proof}

\begin{remark}\label{rmk:TorUniformlyBoundedGmod}
	If $\lie{h} = 0$, then $\Len_{\lie{h}}(\Tor^{\univ{k}}_{0}(F, M_i))$ is the dimension of $F\otimes_{\univ{k}} M_i$.
	Hence we can deduce a kind of finite multiplicity theorems from the proposition.
	Replacing $G/B$ by a partial flag variety $G/P$ and $\ntDsheaf_{G/B,\lambda}\otimes_{\univ{k}}F$ by some holonomic $\ntDsheaf$-module,
	one can obtain several finite multiplicity theorems.
	We postpone the results to the sequel.
\end{remark}

A typical example of uniformly bounded family is a family of Harish-Chandra modules.

\begin{proposition}\label{prop:UniformlyBoundedG-modSpherical}
	Any family of $(\lie{g}, \lie{k})$-modules with bounded lengths is uniformly bounded.
	In particular, so is any family of $(\lie{g}, K)$-modules with bounded lengths.
\end{proposition}

\begin{proof}
	The second assertion follows from the first one because the length of $(\lie{g}, K)$-module $V$ is bounded by $|K/K_0|\cdot \Len_{\lie{g}}(V)$ by Lemma \ref{lem:GeneralizedPairConnected}.

	Take a covering $K'$ of $K_0$ such that $[K'/U_{K'}, K'/U_{K'}]$ is simply-connected, where $U_{K'}$ is the unipotent radical of $K'$.
	Then we have a homomorphism $K'\rightarrow K\rightarrow \Aut(\lie{g})$ and $K'$ has finitely many orbits in $G/B$.

	Let $V$ be an irreducible $(\lie{g}, \lie{k})$-module.
	We want to realize $V$ as $\sect(\calV)$ for some irreducible $\ntDsheaf$-module $\calV$ on $G/B$.
	Since the $\lie{k}$-action is locally finite, we can take a finite-dimensional irreducible $\lie{k}$-submodule $F \subset V$.
	Take a character $\mu$ of $\lie{k}$ such that the $\lie{k}$-action on $F\otimes \CC_\mu$ lifts to a $K'$-action.
	Since $V$ is irreducible, the multiplication map $\univ{g}\otimes F \rightarrow V$ is surjective.
	Hence the $\lie{k}$-action on $V\otimes \CC_\mu$ lifts to a $K'$-action.

	Let $\lambda$ be a character of $\lie{t}$ such that $\lambda - \rho$ is anti-dominant and $\lambda - \rho$ is the infinitesimal character of $V$.
	We can take an irreducible subquotient $\calV$ of $\ntDsheaf_{G/B, \lambda}\otimes_{\univ{g}} V$ such that $\sect(\calV)\simeq V$ (see \cite[Corollary 11.2.6]{HTT08}).
	By construction, $\calV$ is an irreducible twisted $(\ntDsheaf_{G/B,\lambda}, K')$-module with twist $\mu$.

	Since $K'$ has finitely many orbits in $G/B$, the proposition follows from Corollary \ref{cor:UniformlyBoundedIrreducibles}.
\end{proof}

\subsection{Induction of uniformly bounded family}

In this subsection, we will show uniform boundedness of some family of $\lie{g}$-modules.

Let $G$ be a connected reductive algebraic group and $B$ a Borel subgroup of $G$ with unipotent radical $U$.
Put $T:=B/U$.
We denote by $\calI_\chi$ the minimal primitive ideal of $\univ{g}$ with infinitesimal character $\chi$.
Let $W_G$ be the Weyl group of $G$.

\begin{proposition}\label{prop:UniformlyBoundedMinimalPrimitive}
	The family $(\univ{g}/\calI_\chi)_{\chi \in \lie{t}^*/W_G}$ of $(\lie{g}\oplus\lie{g}, \Delta(G))$-modules is uniformly bounded.
	In particular, the number of two sided ideals with fixed infinitesimal character is bounded by a constant independent of its infinitesimal character.
\end{proposition}

\begin{proof}
	$\univ{g}/\calI_\chi$ is isomorphic to $\sect((\ntDsheaf_{G/B,\lambda+\rho}\boxtimes \ntDsheaf_{G/B, \lambda'+\rho})\otimes_{\calU(\Delta(\lie{g}))} \CC)$,
	where $\lambda$ (resp.\ $\lambda'$) is an anti-dominant weight in $\chi$ (resp.\ $-\chi$).
	Since $\Delta(G)$ has finitely many orbits in $G/B\times G/B$, the assertion follows from Corollary \ref{cor:UniformlyBoundedTor}.
\end{proof}

\begin{remark}
	The structure of the $(\lie{g}\oplus \lie{g}, \Delta(G))$-modules can be reduced to that of Verma modules (see \cite[Section 6]{BeGe80_projective_functor}).
	The proposition can be deduced from this and Soergel's theorem \cite[Theorem 11]{So90} (see also Remark \ref{rmk:Soergel}).
\end{remark}

\begin{proposition}\label{prop:LengthVerma}
	Let $P$ be a parabolic subgroup of $G$ containing $B$ with unipotent radical $U_P$, and $(M_i)_{i \in I}$ a uniformly bounded family of $\lie{p}/\lie{u}_P$-modules.
	Then $(\univ{g}\otimes_{\univ{p}} M_i)_{i \in I}$ is a uniformly bounded family of $\lie{g}$-modules.
	In particular, the length of any Verma module is bounded by a constant independent of its highest weight.
\end{proposition}

\begin{proof}
	Since $P$ is parabolic, each $\lie{g}$-module $\univ{g}\otimes_{\univ{p}} M_i$ has an infinitesimal character $\chi_i$.
	Then we have
	\begin{align*}
		\univ{g}\otimes_{\univ{p}} M_i \simeq (\univ{g}/\calI_{\chi_i}\otimes M_i)\otimes_{\univ{p}} \CC.
	\end{align*}
	$(\univ{g}/\calI_{\chi_i}\otimes M_i)_{i\in I}$ is a uniformly bounded family of $(\lie{g\oplus g\oplus p/u}_P)$-modules by Propositions \ref{prop:UniformlyBoundedMinimalPrimitive} and \ref{prop:FundamentailUniformlyBoundedGH} (\ref{enum:prop:UniformlyBoundedGH}).
	Since $P$ has finitely many orbits in $G/B\times P/B$, the assertion follows from Proposition \ref{prop:TorUniformlyBoundedGmod}.
\end{proof}

\begin{remark}\label{rmk:Soergel}
	The second assertion is an easy consequence of Soergel's theorem \cite[Theorem 11]{So90}.
	In fact, the categorical structure of each block of the BGG category $\calO$ depends only on a pair of a Coxeter system and a subgroup of $W_G$ determined by the block,
	and the number of such pairs is finite.
\end{remark}

We consider cohomologically induced modules.
Let $(\lie{g}, K)$ be a pair and $\lie{p}$ a parabolic subalgebra of $\lie{g}$.
Take a Levi subalgebra $\lie{l}$ of $\lie{p}$ and a reductive subgroup $K_L$ of $K$ whose Lie algebra is contained in $\lie{l}\cap \lie{k}$.
Assume that $K_L$ normalizes $\lie{p}$ and $\lie{l}$.
We consider $\lie{l}$-modules as $\lie{p}$-modules through the natural surjection $\lie{p}\rightarrow \lie{l}$.

\begin{theorem}\label{thm:uniformlyBoundedLengthCohInd}
	Let $(V_i)_{i \in I}$ be a uniformly bounded family of $(\lie{l}, K_L)$-modules, e.g.\ a family of irreducible Harish-Chandra modules.
	(See Proposition \ref{prop:UniformlyBoundedG-modSpherical}.)
	Then $(\Dzuck{K}{K_L}{j}(\univ{g}\otimes_{\univ{p}}V_i))_{j \in \ZZ, i\in I}$ is a uniformly bounded family of $(\lie{g}, K)$-modules.
	In particular, there exists some constant $C$ such that for any $i \in I$ and $j \in \ZZ$, we have
	\begin{align*}
	\Len_{\lie{g}, K}(\Dzuck{K}{K_L}{j}(\univ{g}\otimes_{\univ{p}}V_i)) \leq C.
	\end{align*}
\end{theorem}

\begin{proof}
	The assertion follows from Propositions \ref{prop:LengthVerma} and \ref{prop:FundamentalUniformlyBoundedGmod} (i).
\end{proof}

It is well-known that a $(\lie{g}, K)$-module cohomologically induced from an irreducible module of
a parabolic subpair 
is of finite length (see e.g.\ \cite[Theorem 0.46]{KnVo95_cohomological_induction}).
In addition to this fact, we have shown that the lengths of such modules are bounded.

The following corollary is a special case of Theorem \ref{thm:uniformlyBoundedLengthCohInd} because the underlying Harish-Chandra module of any principal series representation can be realized as a cohomologically induced module.
See \cite[Propositions 11.57 and 11.65]{KnVo95_cohomological_induction}.

\begin{corollary}
	Let $G_\RR$ be a real reductive Lie group.
	Then there exists a constant $C$ such that the length of any principal series representation of $G_\RR$ is bounded by $C$.
\end{corollary}

\begin{remark}
	The corollary has been proved in \cite[Proposition 4.1]{KoOs13} by using the theory of minimal $K$-types and the translation principle.
\end{remark}

\subsection{\texorpdfstring{$\univ{g}^{G'}$}{U(g)G'}-modules}

For applications to the branching problem and harmonic analysis, we shall summarize several consequences of the results so far about uniformly bounded families.

Let $G$ be a reductive algebraic group and $G'$ a reductive subgroup of $G$.

\begin{theorem}\label{thm:uniformlyBoundedLengthGeneral}
	Let $(V_{i})_{i \in I}$ and $(V'_{i})_{i \in I}$ be uniformly bounded families of $\lie{g}$-modules and $\lie{g'}$-modules, respectively.
	Then there exists some constant $C$ such that for any $i \in I$ and $j \in \NN$, we have
	\begin{align*}
		\Len_{\univ{g}^{G'}}(\Tor^{\univ{g'}}_j(V_i, V'_i)) \leq C.
	\end{align*}
\end{theorem}

\begin{proof}
	By Corollary \ref{cor:TorAndBernLength} and Proposition \ref{prop:FundamentalUniformlyBoundedGmod} (ii) for $K=\Delta(G')$ and $M=\set{e}$, there is a constant $C$ such that for any $i \in I$ and $j \in \NN$, 
	\begin{align*}
		\Len_{\univ{g}^{G'}}(H^j(\lie{g'}; V_i \otimes V'_i)) \leq C.
	\end{align*}
	Put $n = \dim_{\CC}(\lie{g'})$.
	By the Poincar\'e duality (Fact \ref{fact:PoincareDuality}), we have
	\begin{align*}
		H^j(\lie{g'}; V_i \otimes V'_i) \simeq H_{n-j}(\lie{g'}; V_i\otimes V'_i)
		\simeq \Tor^{\univ{g'}}_{n-j}(V_i, V'_i).
	\end{align*}
	Since these isomorphisms are natural in $V_i$ and $V'_i$, the isomorphisms are $\univ{g}^{G'}$-homomorphisms.
	We have shown the theorem.
\end{proof}

\begin{corollary}\label{cor:uniformlyBoundedLengthN}
	Let $\lie{b'}$ be a Borel subalgebra of $\lie{g'}$ and $(V_i)_{i \in I}$ a uniformly bounded family of $\lie{g}$-modules.
	Then there exists some constant $C$ such that for any character $\lambda$ of $\lie{b'}$, $j \in \ZZ$ and $i \in I$, we have
	\begin{align*}
	\Len_{\univ{g}^{G'}}(\Tor^{\univ{b'}}_{j}(V_i, \CC_{\lambda})) \leq C.
	\end{align*}
	Moreover, the constant $C$ can be chosen independently of $\lie{b'}$.
\end{corollary}
	
\begin{proof}
	Since $\univ{g'}$ is a free right $\univ{b'}$-module, there is a natural isomorphism
	\begin{align*}
		\Tor^{\univ{b'}}_{j}(V_i, \CC_{\lambda}) \simeq \Tor^{\univ{g'}}_j(V_i, \univ{g'}\otimes_{\univ{b'}}\CC_\lambda)
	\end{align*}
	of $\univ{g}^{G'}$-modules.
	The family $(\univ{g'}\otimes_{\univ{b'}}\CC_{\lambda})_{\lambda, \lie{b'}}$ is uniformly bounded by Proposition \ref{prop:LengthVerma} and Proposition \ref{prop:FundamentailUniformlyBoundedGH} (iii).
	Hence the corollary follows from Theorem \ref{thm:uniformlyBoundedLengthGeneral}.
\end{proof}

\begin{corollary}\label{cor:uniformlyBoundedLengthInfChar}
	Let $(V_i)_{i \in I}$ be a uniformly bounded family of $\lie{g}$-modules.
	There exists some constant $C$ such that for any maximal ideal $\calI$ of $\univcent{g'}$, $i \in I$ and $j \in \ZZ$, we have
	\begin{align*}
	\Len_{\univ{g}^{G'}\otimes \univ{g'}}(\Tor_j^{\univcent{g'}}(\univcent{g'}/\calI, V_i)) \leq C.
	\end{align*}
\end{corollary}

\begin{proof}
	Since $\univ{g'}$ is a free $\univcent{g'}$-module,
	we have a natural isomorphism
	\begin{align*}
		\Tor_j^{\univcent{g'}}(\univcent{g'}/\calI, V_i)\simeq 
		\Tor_j^{\univ{g'}}(\univ{g'}/\calI\univ{g'}, V_i).
	\end{align*}
	Hence the corollary follows from Proposition \ref{prop:UniformlyBoundedMinimalPrimitive} and Theorem \ref{thm:uniformlyBoundedLengthGeneral}.
\end{proof}

Retain the notation $G$ and $G'$ as above.
Let $(\lie{g}, K)$ and $(\lie{g'}, K')$ be pairs (see Definition \ref{def:pair}).

\begin{corollary}\label{cor:uniformlyBoundedLengthGKGK}
	Assume that $K$ and $K'$ have finitely many orbits in the flag varieties of $\lie{g}$ and $\lie{g'}$, respectively.
	Then there exists some constant $C$ such that
	for any $i \in \NN$, irreducible $(\lie{g'}, K')$-module $V'$ and irreducible $(\lie{g}, K)$-module $V$, we have
	\begin{align*}
		\Len_{\univ{g}^{G'}}(\Tor^{\univ{g'}}_i(V, V')) \leq C.
	\end{align*}
\end{corollary}

\begin{proof}
	By Proposition \ref{prop:UniformlyBoundedG-modSpherical}, any family of irreducible $(\lie{g}, K)$-modules or irreducible $(\lie{g'}, K')$-modules is uniformly bounded.
	Hence the assertion follows from Theorem \ref{thm:uniformlyBoundedLengthGeneral}.
\end{proof}

\subsection{Euler--Poincar\'e characteristic}

We shall define the Euler--Poincar\'e characteristic in the setting of the branching problem and harmonic analysis.
Retain the notation $G$, $G'$, $K$ and $K'$ in the previous subsection.
Assume that $K'$ is reductive and contained in $K$, and $\Ad_{\lie{g}}(K')$ is contained in $\Ad_{\lie{g}}(G')$.

\begin{theorem}\label{thm:uniformlyBoundedLengthGeneralGK}
	Let $(V_{i})_{i \in I}$ (resp. $(V'_{i})_{i \in I}$) be a uniformly bounded family of $(\lie{g}, K)$-modules (resp. $(\lie{g'}, K')$-modules).
	Then there exists some constant $C$ such that for any $i \in I$ and $j \in \NN$, we have
	\begin{align*}
		\Len_{\univ{g}^{G'}}(H_j(\lie{g'}, K'; V_i \otimes V'_i)) \leq C.
	\end{align*}
	In particular, the Euler--Poincar\'e characteristic
	\begin{align*}
		\EP(V_i, V'_i):=\sum_{i}(-1)^i H_i(\lie{g'}, K'; V_i\otimes V'_i)
	\end{align*}
	is well-defined as an element of the Grothendieck group of the category
	of $\univ{g}^{G'}$-modules of finite length.
\end{theorem}

\begin{proof}
	Almost all of the proof is the same as that of Theorem \ref{thm:uniformlyBoundedLengthGeneral}.
	We note the difference.
	In this setting, the Poincar\'e duality (Fact \ref{fact:PoincareDuality}) is written as
	\begin{align*}
		H^{n-j}(\lie{g'}, K'; V_i \otimes V'_i \otimes \wedge^n (\lie{g'}/\lie{k'})) \simeq H_j(\lie{g'}, K'; V_i \otimes V'_i),
	\end{align*}
	where $n = \dim_{\CC}(\lie{g'}/\lie{k'})$ and the $\lie{g'}$-action on $\wedge^n (\lie{g'}/\lie{k'})$ is trivial.
	Hence the twisting by $\wedge^n (\lie{g'}/\lie{k'})$ does not affect the action of $\univ{g}^{G'}$.
	Therefore we have proved the theorem by Proposition \ref{prop:FundamentalUniformlyBoundedGmod} (iv).
\end{proof}

\begin{remark}
	It is clear that for the well-definedness of the Euler--Poincar\'e characteristic, we does not need the notion of uniformly bounded families.
	In fact, we need only holonomicity of modules.
\end{remark}

\begin{remark}
$H_i(\lie{g'}, K'; V\otimes V')^*$ is isomorphic to $\Ext^i_{\lie{g'}, K'}(V, (V')^*_{K'})$ as a $\univ{g}^{G'}$-module (see \cite[Corollary 3.2]{KnVo95_cohomological_induction}).
If $\Ext^i_{\lie{g'}, K'}(V, (V')^*_{K'})$ is not finite dimensional, the $\univ{g}^{G'}$-module does not have finite length because it is uncountably infinite dimensional.
\end{remark}

If all $H_i(\lie{g'}, K'; V\otimes V')$ are finite dimensional, we can define the ($\ZZ$-valued) Euler--Poincar\'e characteristic
\begin{align}
\dim_{\CC}\EP(V, V'):=\sum_{i}(-1)^i \dim_{\CC}(H_i(\lie{g'}, K'; V\otimes V')) \label{eqn:EP}
\end{align}
The characteristic for $p$-adic groups is studied in \cite{Pr13}, \cite{AiSa17} and \cite{ChSa18}.
Remark that $\EP(V, V')$ in the papers corresponds to $\dim_{\CC}\EP(V, (V')^*_{K'})$ in our notation.
We give sufficient conditions for the well-definedness of the $\ZZ$-valued characteristic in the sequels \cite[Corollary 7.17]{Ki20}.

\subsection{Theta lifting}

We apply Theorem \ref{thm:uniformlyBoundedLengthGeneralGK} to the theory of the Howe duality (see \cite{Ho89,Ho89_transcending}).

Let $G_{\RR}$ be a double cover of $\Sp(n, \RR)$.
Let $(H_{\RR}, H'_{\RR})$ be a reductive dual pair of $G_{\RR}$, i.e.\ $H_{\RR} = C_{G_{\RR}}(H'_{\RR})$ and $H'_{\RR} = C_{G_{\RR}}(H_{\RR})$ holds.
Here $C_{G_\RR}(\cdot)$ denotes the centralizer in $G_\RR$.
We write $\lieR{g}, \lieR{h}$ and $\lieR{h'}$ for the Lie algebras of $G_\RR$, $H_\RR$ and $H'_\RR$, respectively.

Fix a Cartan involution $\theta$ of $G_\RR$ which stabilizes $H_\RR$ and $H'_\RR$,
and put $K_\RR:=G_\RR^\theta, K_{H,\RR}:=H_\RR^\theta$ and $K_{H',\RR}:=(H'_\RR)^\theta$.
Then we have pairs $(\lie{g}, K), (\lie{h}, K_H)$ and $(\lie{h'}, K_{H'})$,
which are the complexifications of $(\lieR{g}, K_\RR)$, $(\lieR{h}, K_{H,\RR})$
and $(\lieR{h'}, K_{H',\RR})$, respectively.
We write $(\omega, V)$ for the underlying Harish-Chandra module of the Segal--Shale--Weil representation of $G_{\RR}$.
Then, by the classical invariant theory, we have $\omega(\univ{g})^{H} = \omega(\univ{h'})$.
Here $H$ is the centralizer in $\Sp(n,\CC)$ of the image of $H'_\RR$ by the covering map $G_\RR\rightarrow \Sp(n,\RR)$.

For an irreducible $(\lie{h}, K_H)$-module $V'$, we set
\begin{align*}
\Theta_i(V') := H_i(\lie{h}, K_H; V\otimes \dual{V'}),
\end{align*}
where $\dual{V'}$ is the space of all $K_H$-finite vectors in $(V')^*$.
Then $\Theta_i(V')$ is a $(\lie{h'}, K_{H'})$-module.
Let $\calR(\lie{h}, K_H, \omega)$ be the set of equivalence classes of irreducible $(\lie{h}, K_H)$-modules such that $\Theta_0(V')\neq 0$.

\begin{fact}[R. Howe {\cite[Theorem 2.1]{Ho89_transcending}}]
For any $V' \in \calR(\lie{h}, K_H, \omega)$, $\Theta_0(V')$ is of finite length and has a unique irreducible quotient $\theta(V')$.
The correspondence $\calR(\lie{h}, K_H, \omega) \ni V'\mapsto \theta(V') \in \calR(\lie{h'}, K_{H'}, \omega)$ is bijective.
\end{fact}

For any $i \in \NN$, $\Theta_i(V')$ is of finite length by $\omega(\univ{g})^{H} = \omega(\univ{h'})$ and Theorem \ref{thm:uniformlyBoundedLengthGeneralGK}.
More precisely, the following theorem holds.

\begin{theorem}\label{thm:ThetaLift}
Let $V'$ be an irreducible $(\lie{h}, K_H)$-module.
Then there exists some constant $C$ independent of $V'$ such that
\begin{align*}
\Len_{\lie{h'}, K_{H'}}(\Theta_i(V')) \leq C
\end{align*}
for any $i\in \NN$.
In particular, as an element of the Grothendieck group of the category of $(\lie{h'}, K_{H'})$-modules of finite length,
the Euler--Poincar\'e characteristic
\begin{align*}
\EP(V, \dual{V'}) = \sum_i (-1)^i\Theta_i(V')
\end{align*}
is well-defined.
\end{theorem}

The well-definedness of the Euler--Poincar\'e characteristic of the theta lifting for $p$-adic groups is proved and studied in \cite[Proposition 1.1]{APS17}.

\subsection{Uniformly bounded family in branching laws}

Let $G$ be a connected reductive algebraic group and $G'$ a connected reductive subgroup of $G$.
Using the restriction of modules, we can construct a uniformly bounded family of $\lie{g'}$-modules from one of $\lie{g}$-modules.
We consider the embedding $\iota\colon G'\rightarrow G'\times G'\times G$ defined by $\iota(g) = (e, g, g)$.

\begin{lemma}\label{lem:FindimQuot}
	Let $V$ be a $\lie{g}$-module and $V'$ an irreducible $\lie{g'}$-module, and set $\calI:= \Ann_{\univcent{g'}}(V')$.
	If $0 < \dim_{\CC} \Hom_{\lie{g'}}(V, V') < \infty$, then there exists an irreducible $(\lie{g'\oplus g}, \Delta(G'))$-module $W$ such that $V'\boxtimes W$ is isomorphic to a subquotient of $\Dzuck{\iota(G')}{\set{e}}{n}(\univ{g'}/\calI\otimes V)$,
	where $n:=\dim_{\CC}(G')$.
\end{lemma}

\begin{proof}
	Take a basis $\set{\varphi_i}$ of $\Hom_{\lie{g'}}(V/\calI V, V')(\simeq \Hom_{\lie{g'}}(V, V'))$ and its dual basis $\set{\lambda_i}$ of $\Hom_{\lie{g'}}(V/\calI V, V')^*$.
	Since $\Hom_{\lie{g'}}(V/\calI V, V')$ is finite dimensional, we obtain a $\univ{g'}\otimes \univ{g}^{G'}$-module homomorphism
	\begin{align*}
		V/\calI V \rightarrow V' \boxtimes \Hom_{\lie{g'}}(V/\calI V, V')^*
	\end{align*}
	given by $v\mapsto \sum_i \varphi_i(v)\otimes \lambda_i$.
	Hence the $\univ{g'}\otimes \univ{g}^{G'}$-module $V/\calI V$ has an irreducible quotient of the form $V'\boxtimes W_0$ for an irreducible $\univ{g}^{G'}$-module $W_0$.

	By Fact \ref{fact:PoincareDuality} and Lemma \ref{lem:TorAndBern}, we have
	\begin{align*}
		\Dzuck{\iota(G')}{\set{e}}{n}(\univ{g'}/\calI\otimes V)^{\iota(G')} &\simeq \univ{g'}/\calI\otimes_{\univ{g'}} V \\
		&\simeq V / \calI V
	\end{align*}
	as $\univ{g'}\otimes \univ{g}^{G'}$-modules.
	Hence $V'\boxtimes W_0$ is isomorphic to a quotient of $\Dzuck{\iota(G')}{\set{e}}{n}(\univ{g'}/\calI\otimes V)^{\iota(G')}$.
	This implies that we can take an irreducible subquotient $X$ of $\Dzuck{\iota(G')}{\set{e}}{n}(\univ{g'}/\calI\otimes V)$ such that $X^{\iota(G')} \simeq V'\boxtimes W_0$ (see e.g. \cite[Proposition 3.5.4]{Wa88_real_reductive_I}).
	Since $X = \univ{g}X^{\iota(G')}$, the $\lie{g'}$-module $X|_{\lie{g'}}$ is a direct sum of some copies of $V'$.
	Hence $X$ is naturally isomorphic to $V'\boxtimes \Hom_{\lie{g'}}(V', X)$ and the natural $(\lie{g'\oplus g})$-action on $\Hom_{\lie{g'}}(V', X)$ is irreducible.
	We have shown the lemma.
\end{proof}

\begin{remark}
	Suppose that $V$ is irreducible.
	By the proof, one can see that if the Beilinson--Bernstein localization of $V$ is regular holonomic, then that of $V'$ is also regular holonomic.
\end{remark}

\begin{theorem}\label{thm:BranchingUniformlyBounded}
	Let $(V_i)_{i \in I}$ be a uniformly bounded family of $\lie{g}$-modules and $(V'_i)_{i \in I}$ a family of irreducible $\lie{g'}$-modules.
	If $0 < \dim_{\CC}(\Hom_{\lie{g'}}(V_i, V'_i)) < \infty$ for any $i \in I$, then $(V'_i)_{i \in I}$ is uniformly bounded.
\end{theorem}

\begin{proof}
	Set $\calI_i := \Ann_{\univcent{g'}}(V'_i)$.
	Then $(\univ{g'}/\calI_i\otimes V_i)_{i \in I}$ is a uniformly bounded family of $(\lie{g'\oplus g'\oplus g})$-modules by Proposition \ref{prop:UniformlyBoundedMinimalPrimitive}.
	By Proposition \ref{prop:FundamentalUniformlyBoundedGmod} (\ref{enum:FundamentalZuckerman}), $(\Dzuck{\iota(G')}{\set{e}}{n}(\univ{g'}/\calI_i\otimes V_i))_{i \in I}$ is a uniformly bounded family of $(\lie{g'\oplus g'\oplus g}, \iota(G'))$-modules.
	Here we set $n := \dim_{\CC}(G')$.
	
	By Lemma \ref{lem:FindimQuot}, for each $i \in I$, we can take an irreducible $(\lie{g'\oplus g}, \Delta(G'))$-module $W_i$ such that $V'_i\boxtimes W_i$ is a subquotient of $\Dzuck{\iota(G')}{\set{e}}{n}(\univ{g'}/\calI_i\otimes V_{i})$.
	This implies that $(V'_i\boxtimes W_i)_{i \in I}$ is a uniformly bounded family of $(\lie{g'\oplus g'\oplus g}, \iota(G'))$-modules.
	By Proposition \ref{prop:FundamentailUniformlyBoundedGH} (\ref{enum:prop:UniformlyBoundedGH}), the family $(V'_i)_{i \in I}$ is uniformly bounded.
\end{proof}

\subsection{Tensoring with finite-dimensional modules}

Let $G$ be a connected reductive algebraic group.
We shall show that the uniformly boundedness is preserved by tensoring with finite-dimensional modules.
In particular, the uniformly boundedness is preserved by the translation functors.

\begin{lemma}\label{lem:TranslationLength}
	Let $(V_i)_{i \in I}$ be a uniformly bounded family of $\lie{g}$-modules and $F$ a finite-dimensional $\lie{g}$-module.
	Then there exists a constant $C > 0$ independent of $F$ such that
	\begin{align*}
		\Len_{\lie{g}}(V_i\otimes F) \leq C\cdot \dim_{\CC}(F)^2
	\end{align*}
	for any $i \in I$.
\end{lemma}

\begin{proof}
	Clearly, we can assume that $F$ is completely reducible.
	Since the lengths of all $V_i$ are bounded by a constant independent of $i \in I$, we can also assume that all $V_i$ are irreducible.
	Set $n:=\dim_{\CC}(F)$.

	Fix $i \in I$.
	By Kostant's theorem \cite[Theorem 7.133]{KnVo95_cohomological_induction}, $V_i\otimes F$ is a direct sum of finitely many submodules $W_1, W_2, \ldots, W_m$ with generalized infinitesimal characters $\chi_1, \chi_2, \ldots, \chi_m$, respectively.
	More precisely, we have $m \leq n$ and $\calI_{j}^{|W_{\lie{g}}|} W_j = 0$ for any $1\leq j \leq m$,
	where $\calI_{j}$ is the maximal ideal of $\univcent{g}$ corresponding to $\chi_j$ and $W_{\lie{g}}$ is the Weyl group of $\lie{g}$.
	There is a $\lie{g}$-module surjection
	\begin{align*}
		\calI_j^{k} \otimes (W_j/\calI_j W_j) \twoheadrightarrow \calI_j^k W_j / \calI_j^{k+1} W_j
	\end{align*}
	for any $k \in \NN$, and $\calI_j^{k}$ is generated by $r^k$ elements as a $\univcent{g}$-module.
	Here $r$ is the rank of $\lie{g}$.
	Hence we have
	\begin{align}
		\Len_{\lie{g}}(V_i\otimes F) &= \sum_j \Len_{\lie{g}}(W_j) \\
		&\leq C' \cdot \sum_j \Len_{\lie{g}}(W_j/\calI_j W_j) \\
		&=C' \cdot \sum_j \Len_{\lie{g}}((V_i\otimes F)/\calI_j (V_i\otimes F)), \label{eqn:TranslationLengthInf}
	\end{align}
	where $C'$ is a constant depending only on $|W_{\lie{g}}|$ and $r$.
	
	We shall estimate $\Len_{\lie{g}}((V_i\otimes F)/\calI_j (V_i\otimes F))$.
	By Corollary \ref{cor:uniformlyBoundedLengthInfChar}, there exists a constant $C''$ depending only on the family $(V_i)_{i \in I}$ such that
	\begin{align}
		\Len_{\univ{g}\otimes \univ{g\oplus g}^{\Delta(G)}}((V_i\otimes F)/\calI_j (V_i\otimes F)) \leq C'' \label{eqn:TranslationLength}
	\end{align}
	for any $j$.
	Here the action of $\univ{g\oplus g}^{\Delta(G)}$ on $V_i\otimes F$ factors through
	\begin{align*}
		(\univ{g}/\Ann_{\univ{g}}(V_i)\otimes \End_{\CC}(F))^{\Delta(G)}.
	\end{align*}
	Take a maximal torus $T$ of $G$.
	Since $\univ{g}/\Ann_{\univ{g}}(V_i)$ is isomorphic to a submodule of $\rring{G/T}$ as a $G$-module (see \cite[Theorem 12.4.2]{GoWa09}), we have
	\begin{align*}
		\dim_{\CC}((\univ{g}/\Ann_{\univ{g}}(V_i)\otimes \End_{\CC}(F))^{\Delta(G)}) \leq \dim_{\CC}(\End_{T}(F)) \leq \dim_{\CC}(F)^2.
	\end{align*}
	In particular, the dimension of any irreducible module of $(\univ{g}/\Ann_{\univ{g}}(V_i)\otimes \End_{\CC}(F))^{\Delta(G)}$ is less than or equal to $\dim_{\CC}(F)$.
	By \eqref{eqn:TranslationLength}, we have
	\begin{align*}
		\Len_{\lie{g}}((V_i\otimes F)/\calI_j (V_i\otimes F)) \leq C'' \cdot \dim_{\CC}(F).
	\end{align*}
	Combining \eqref{eqn:TranslationLengthInf}, we obtain
	\begin{equation*}
		\Len_{\lie{g}}(V_i\otimes F) \leq C'C''\cdot \dim_{\CC}(F)^2. \qedhere
	\end{equation*}
\end{proof}

One can prove and refine Lemma \ref{lem:TranslationLength} using twisting of $\ntDsheaf$-modules on the flag variety of $\lie{g}$ or the theory of projective functors \cite{BeGe80_projective_functor}.
For $(\lie{g}, K)$-modules, a more precise estimate is known \cite[Proposition 5.4.1 and its proof]{Ko08}.

\begin{theorem}\label{thm:UniformlyBoundedTranslation}
	Let $(V_i)_{i \in I}$ be a uniformly bounded family of $\lie{g}$-modules and $(F_j)_{j \in J}$ a family of finite-dimensional $\lie{g}$-modules with bounded dimensions.
	Then $(V_i\otimes F_j)_{i\in I, j\in J}$ is a uniformly bounded family of $\lie{g}$-modules.
\end{theorem}

\begin{proof}
	For $i\in I$ and $j\in J$, let $R_{ij}$ be the set of all composition factors of $V_i\otimes F_j$.
	By Lemma \ref{lem:TranslationLength}, the lengths of all $V_i\otimes F_j$ are bounded by a constant independent of $i \in I$ and $j \in J$.
	Hence it suffices to show that the family $(W)_{W \in R_{ij}, i\in I, j\in J}$ is uniformly bounded.

	As we have seen in the proof of Lemma \ref{lem:TranslationLength}, any element of $R_{ij}$ is a subquotient of $(V_i\otimes F_j)/\calI(V_i\otimes F_j)$ for a maximal ideal $\calI$ of $\univcent{g}$.
	By Theorem \ref{thm:BranchingUniformlyBounded}, the family $(W)_{W \in R_{ij}, i\in I, j\in J}$ is uniformly bounded.
	Note that although we have proved Theorem \ref{thm:BranchingUniformlyBounded} for a family of irreducible quotients, the proof also works for a family of irreducible subquotients under some finiteness assumption.
\end{proof}

\subsection{Category of \texorpdfstring{$(\lie{g}, \lie{k})$}{(g,k)}-modules}

Let $G$ be a connected reductive algebraic group and $K$ a finite covering of a connected reductive subgroup of $G$.
Suppose that $[K, K]$ is simply-connected.
We denote by $\calC(\lie{g}, \lie{k})$ the full subcategory of $\Mod(\lie{g})$ whose object is 
\begin{enumerate}[(i)]
	\item of finite length,
	\item locally finite and completely reducible as a $\lie{k}$-module, and
	\item $\lie{k}$-admissible, i.e.\ any $\lie{k}$-isotypic component is finite dimensional.
\end{enumerate}
Such a module is called a generalized Harish-Chandra module by I. Penkov and G. Zuckerman (see e.g.\ \cite{PeZu14}).
We write $\calC_\chi(\lie{g}, \lie{k})$ for the full subcategory of $\calC(\lie{g}, \lie{k})$ whose object has the infinitesimal character $\chi$.

In this subsection, we study the category $\calC_{\chi}(\lie{g}, \lie{k})$.
It is related to the branching problem and harmonic analysis because the algebra $\univ{g}^\lie{k}$ roughly controls multiplicities and its modules can be obtained from the $\Delta(K)$-invariant part of $(\lie{g\oplus k}, \Delta(K))$-modules.
See Theorem \ref{thm:uniformlyBoundedLengthGeneral}.

\begin{theorem}\label{thm:twosidedidal}
Let $\calI$ (resp.\ $\calJ$) be a maximal ideal of $\univcent{g}$ (resp.\ $\univcent{k}$).
The number of two sided ideals of $\univ{g}^{K}/(\calI+\calJ)\univ{g}^{K}$
is bounded by some constant independent of $\calI$ and $\calJ$.
\end{theorem}

\begin{proof}
We have a $(\univ{g}^{K}, \univ{g}^K)$-bimodule isomorphism
\begin{align*}
\univ{g}^{K}/(\calI+\calJ)\univ{g}^{K} &= (\univ{g}/(\calI\univ{g}+\calJ\univ{g}))^{K} \\
&\simeq (\univ{k}/\calJ \univ{k}\otimes_{\univ{k}}\univ{g}/\calI\univ{g})^{K}.
\end{align*}
By Proposition \ref{prop:UniformlyBoundedMinimalPrimitive}, $(\univ{g}/\calI\univ{g})_\calI$ and $(\univ{k}/\calJ \univ{k})_\calJ$
are uniformly bounded families of $(\lie{g\oplus g})$-modules and $(\lie{k\oplus k})$-modules, respectively.
By applying Theorem \ref{thm:uniformlyBoundedLengthGeneral} to $\univ{k}/\calJ \univ{k}\otimes \univ{g}/\calI\univ{g}$, there exists some constant $C$ independent of $\calI$ and $\calJ$ such that
\begin{align*}
\Len_{\univ{g}^K\otimes \univ{g}^K}((\univ{k}/\calJ \univ{k}\otimes_{\univ{k}}\univ{g}/\calI\univ{g})^{K}) \leq C.
\end{align*}
This shows the theorem.
\end{proof}

\begin{remark}
If $\lie{k}$ does not contain any non-trivial ideal of $\lie{g}$,
the center of $\univ{g}^{K}$ is equal to $\univcent{g}\univcent{k}\simeq \univcent{g}\otimes \univcent{k}$ by \cite[Theorem 10.1]{Kn94}.
\end{remark}

Let $\calI_\chi$ be the minimal primitive ideal of $\univ{g}$ with infinitesimal character $\chi$.

\begin{theorem}
	Any family of objects in $\calC(\lie{g}, \lie{k})$ with bounded lengths is a uniformly bounded family.
	In particular, for any irreducible object $V \in \calC_{\chi}(\lie{g}, \lie{k})$, there exist an anti-dominant $\lambda \in \chi$ and some $\calM \in \Mod_h(\ntDsheaf_{G/B,\lambda+\rho})$ such that $V\simeq \sect(\calM)$.
	(See Subsection \ref{sect:BBcorrespondence} for the notation.)
\end{theorem}

\begin{remark}
	The second assertion has been proved by A. V. Petukhov \cite{Pe12}.
	More precisely, one can see from our proof that $\calM$ is regular holonomic.
\end{remark}

\begin{proof}
	It is enough to show that the family of all irreducible objects in $\calC(\lie{g}, \lie{k})$ (modulo isomorphism) is uniformly bounded.

	Let $V$ be an irreducible object in $\calC(\lie{g}, \lie{k})$ with infinitesimal character $\chi$.
	Then we can take a character $\mu$ of $\lie{k}$ such that $V\otimes \CC_\mu$
	lifts to a $K$-module (see the proof of Proposition \ref{prop:UniformlyBoundedG-modSpherical}).
	Since $V\otimes \CC_\mu$ is an irreducible $(\lie{k}/[\lie{k},\lie{k}]\oplus \lie{g}, \Delta(K))$-module, replacing $V$ by $V\otimes \CC_\mu$ and $(\lie{g}, \lie{k})$ by $(\lie{k}/[\lie{k},\lie{k}]\oplus \lie{g}, \Delta(K))$, we can assume that $V$ is an irreducible $(\lie{g}, K)$-module.
	See also the proof of Corollary \ref{cor:UniformlyBoundedIrreducibles}.

	We put $W:=\Dzuck{K}{\set{e}}{n}(\univ{g}/\calI_\chi)$, where $n = \dim_{\CC}(\lie{k})$ and we take the functor $\Dzuck{K}{\set{e}}{n}$ with respect to the left $\lie{k}$-action.
	Then for any irreducible $K$-module $F$, we have
	\begin{align*}
		\Hom_{K}(F, W) \simeq F^*\otimes_{\univ{k}} \univ{g}/\calI_\chi
	\end{align*}
	by Fact \ref{fact:BernAndF}.
	This implies that $W$ is a $(\lie{g\oplus g}, K\times K)$-module.
	
	$\dual{V}$ denotes the subspace of all $K$-finite vectors in $V^*$, which is the dual in $\calC(\lie{g}, \lie{k})$.
	It is easy to see that $\dual{V}$ is irreducible by the $K$-admissibility.
	We shall show that $V\boxtimes \dual{V}$ is a subquotient of $W$.

	Fix an irreducible $K$-submodule $F$ of $V$.
	Then we have isomorphisms
	\begin{align*}
		\Hom_{K\times K}(F\boxtimes F^*, W) &\simeq F^*\otimes_{\univ{k}} \univ{g}/\calI_\chi \otimes_{\univ{k}}F \\
		&\simeq (\univ{g}/\calI_\chi \otimes_{\univ{k}} \End_{\CC}(F))^K\\
		&\simeq (\univ{g}/(\calI_\chi + \univ{g}\Ann_{\univ{k}}(F)))^K
	\end{align*}
	as $\univ{g\oplus g}^{K\times K}$-modules.
	Here $\Ann_{\univ{k}}(F)$ denotes the annihilator of $F$ in $\univ{k}$.
	Since $\Hom_{K}(F, V)$ is a finite-dimensional irreducible $\univ{g}^K$-module (see Fact \ref{fact:UgK-module}), we have a surjection
	\begin{align*}
		(\univ{g}/(\calI_\chi + \univ{g}\Ann_{\univ{k}}(F)))^K \rightarrow \End_{\CC}(\Hom_K(F, V))
	\end{align*}
	by the Jacobson density theorem.
	This implies that there is a surjection
	\begin{align*}
		\Hom_{K\times K}(F\boxtimes F^*, W) &\rightarrow \Hom_{K\times K}(F\boxtimes F^*, V\boxtimes \dual{V}) \\
		&\simeq \End_{\CC}(\Hom_K(F, V))
	\end{align*}
	of $\univ{g\oplus g}^{K\times K}$-modules.
	By \cite[Proposition 3.5.4]{Wa88_real_reductive_I}, $V\boxtimes \dual{V}$ is isomorphic to a subquotient of $W$.

	By Propositions \ref{prop:UniformlyBoundedMinimalPrimitive} and \ref{prop:FundamentalUniformlyBoundedGmod} (i), $(\Dzuck{K}{\set{e}}{n}(\univ{g}/\calI_\chi))_{\chi}$ is a uniformly bounded family of $(\lie{g\oplus g}, K\times K)$-modules.
	By Proposition \ref{prop:FundamentailUniformlyBoundedGH} (i) the family $(V\boxtimes \dual{V})_{V \in \calS}$ is uniformly bounded, where $\calS$ is the set of all equivalence classes of irreducible $K$-admissible $(\lie{g}, K)$-modules.
	From this and Proposition \ref{prop:FundamentailUniformlyBoundedGH} (\ref{enum:prop:UniformlyBoundedGH}), the family $(V)_{V\in \calS}$ is uniformly bounded.
	This shows the theorem.
\end{proof}

In the proof, we have proved the following corollary.
For a character $\mu$ of $\lie{k}$, we denote by $\calC_{\chi, \mu}(\lie{g}, \lie{k})$ the full subcategory of $\calC_\chi(\lie{g}, \lie{k})$ whose object $V$
satisfies that $V\otimes \CC_\mu$ lifts to a $K$-module.

\begin{corollary}\label{cor:NumberOfIrreducibles}
	Let $\mu$ be a character of $\lie{k}$ and $\chi$ an infinitesimal character of $\lie{g}$.
	The number of equivalence classes of irreducible objects in $\calC_{\chi, \mu}(\lie{g}, \lie{k})$ is bounded by some constant independent of $\mu$ and $\chi$.
\end{corollary}

\begin{remark}
	The number of equivalence classes of irreducible objects in $\calC_\chi(\lie{g}, \lie{k})$ may be infinite.
	$(\lie{g}, \lie{k}) = (\sl(2,\CC), \so(2,\CC))$ gives an example.
	See \cite[Theorem 1.3.1]{HoTa92}.
\end{remark}

For a $(\lie{g}, \lie{k})$-module $V$ and an irreducible finite-dimensional $\lie{k}$-module $F$, we denote by $V(F)$ the isotypic component with respect to $F$.
Then $V(F)$ is a $\univ{k}\otimes \univ{g}^K$-module.
It is well-known that the $\lie{g}$-module structure on $V$ is related to the $\univ{k}\otimes \univ{g}^K$-module structure on $V(F)$.
See \cite[Lemma 3.5.3]{Wa88_real_reductive_I} and \cite{LeMc73}.
Recall that $\univ{k}$ and $\univ{g}^K$ are noetherian.

\begin{fact}\label{fact:UgK-module}
	Let $V$ be a $(\lie{g}, \lie{k})$-module and $F$ an irreducible finite-dimensional $\lie{k}$-module.
	For any submodule $W$ of $V(F)$, we have $(\univ{g} W)(F) = W$.
	In particular, the length of $V(F)$ is less than or equal to that of $V$,
	and $V(F)$ is finitely generated if $V$ is finitely generated.
\end{fact}

\begin{lemma}\label{lem:FiniteUgK}
	Let $F$ be an irreducible $K$-module.
	Then $\univ{g}(F)$ with respect to the adjoint action is finitely generated as a left/right $\univ{g}^K$-module.
	In particular, any finitely generated submodule of the $(\univ{g}^K, \univ{g}^K)$-bimodule $\univ{g}$ is finitely generated as a left/right $\univ{g}^K$-module.
\end{lemma}

\begin{proof}
	$F^*\otimes \univ{g}$ is finitely generated left $\univ{g}$-module.
	By \cite[Lemma 2.2]{Ki14}, $(F^*\otimes \univ{g})^K$ is finitely generated left $\univ{g}^K$-module.
	Since $\univ{g}(F)$ is canonically isomorphic to $F\otimes \Hom_K(F, \univ{g})$ as a $\univ{g}^K$-module, this shows the first assertion.

	Any finitely generated submodule of the $(\univ{g}^K, \univ{g}^K)$-bimodule $\univ{g}$ is contained in a finite sum of some $K$-isotypic components.
	Since $\univ{g}^K$ is noetherian, the second assertion follows from the first one.
\end{proof}

\begin{lemma}\label{lem:FiniteGK}
	Let $V$ be a $(\lie{g}, \lie{k})$-module and $F$ an irreducible finite-dimensional $\lie{k}$-module.
	If $V(F)$ is finite dimensional and generates $V$, then $V$ is in $\calC(\lie{g}, \lie{k})$.
\end{lemma}

\begin{proof}
	Since the multiplication map $\univ{g}V(F)\rightarrow V$ is surjective,
	$V$ is completely reducible as a $\lie{k}$-module.
	We shall show that $V$ is $\lie{k}$-admissible and of finite length.

	Let $F'$ be an irreducible finite-dimensional $\lie{k}$-module.
	We shall show that $V(F')$ is finite dimensional.
	Since $V$ is finitely generated, $V(F')$ is finitely generated as a $\univ{g}^K$-module by Fact \ref{fact:UgK-module}.
	Since $V$ is generated by $V(F)$, we can take a finite-dimensional subspace $X \subset \univ{g}$ such that $V(F')\subset \univ{g}^K XV(F)$.
	By Lemma \ref{lem:FiniteUgK}, there exists a finite-dimensional subspace $X' \subset \univ{g}$ such that $\univ{g}^K X \univ{g}^K = X'\univ{g}^K$.
	Then we have
	\begin{align*}
		V(F') \subset \univ{g}^K XV(F) = X' \univ{g}^K V(F) = X' V(F),
	\end{align*}
	and hence $V(F')$ is finite dimensional.
	Therefore $V$ is $\lie{k}$-admissible.

	We shall show that $V$ is of finite length.
	Since $V$ is generated by $V(F)$, $\Ann_{\univcent{g}}(V)$ is of finite codimension in $\univcent{g}$.
	Hence $V$ is a finite direct sum of $\lie{g}$-submodules with generalized infinitesimal characters.
	We can assume that $V$ has a generalized infinitesimal character $\chi$.

	Since $V$ is generated by $V(F)$, there is a character $\mu$ of $\lie{k}$ such that $V\otimes \CC_\mu$ lifts to a $K$-module.
	Then any irreducible subquotient of $V$ is in $\calC_{\chi,\mu}(\lie{g}, \lie{k})$.
	By Corollary \ref{cor:NumberOfIrreducibles}, the number of equivalence classes of irreducible objects in $\calC_{\chi,\mu}(\lie{g}, \lie{k})$ is finite.
	Since $V$ is $\lie{k}$-admissible and noetherian, this shows that $V$ is of finite length.
\end{proof}

Suppose that $0=V_0\subset V_1 \subset V_2 \subset \cdots \subset V_r = V$ is the socle filtration of $V \in \calC(\lie{g}, \lie{k})$, that is, each $V_i / V_{i-1}$ is the sum of all irreducible submodules in $V/V_{i-1}$.
The length $r$ is called the Loewy length of $V$.

\begin{theorem}\label{thm:LoewyLength}
	The Loewy length of any object in $\calC_\chi(\lie{g}, \lie{k})$ is bounded by some constant independent of the object and the infinitesimal character $\chi$.
\end{theorem}

\begin{proof}
	We construct projective objects in $\calC(\lie{g}, \lie{k})$ using $\univ{g}\otimes_{\univ{k}}F$, which is a projective object in the category of all $(\lie{g}, \lie{k})$-modules whose $\lie{k}$-actions are completely reducible.

	Let $F$ be an irreducible finite-dimensional $\lie{k}$-module
	and $\chi$ an infinitesimal character of $\lie{g}$.
	Put
	\begin{align*}
		\widetilde{P}_{F,\chi} := \univ{g}/\calI_\chi \otimes_{\univ{k}}F
	\end{align*}
	Then there is a canonical isomorphism
	\begin{align*}
		\Hom_{\lie{k}}(F, \widetilde{P}_{F,\chi}) \simeq (\univ{g}/(\calI_\chi + \univ{g}\Ann_{\univ{k}}(F)))^K
	\end{align*}
	of $(\univ{g}^K, \univ{g}^K)$-modules.
	We regard $\calA:=\Hom_{\lie{k}}(F, \widetilde{P}_{F,\chi})$ as an algebra under this isomorphism.

	Let $\calJ$ be the union of all left ideals of finite codimension in $\calA$.
	Then $\calJ$ is also the union of all two-sided ideals of finite codimension.
	By Theorem \ref{thm:twosidedidal}, the number of two-sided ideals is finite.
	Hence $\calJ$ is of finite codimension in $\calA$.

	There is a canonical isomorphism
	\begin{align*}
		F\otimes \Hom_{\lie{k}}(F, \widetilde{P}_{F,\chi}) \xrightarrow{\simeq} \widetilde{P}_{F,\chi}(F)
	\end{align*}
	of $\univ{k}\otimes \univ{g}^K\otimes \univ{g}^K$-modules.
	We consider $F\otimes \calJ$ as a subspace of $\widetilde{P}_{F,\chi}$ by the isomorphism.
	Put $\widetilde{J}:=\univ{g}\cdot (F\otimes \calJ)$ and
	\begin{align*}
		P_{F,\chi} := \widetilde{P}_{F,\chi} / \widetilde{J}.
	\end{align*}
	Since $\widetilde{J}(F) = F\otimes J$ by Fact \ref{fact:UgK-module}, we have
	\begin{align*}
		P_{F,\chi}(F) \simeq F\otimes (\Hom_{\lie{k}}(F, \widetilde{P}_{F,\chi}) / \calJ).
	\end{align*}
	Hence $P_{F,\chi}$ is generated by the finite-dimensional subspace $P_{F,\chi}(F)$.
	By Lemma \ref{lem:FiniteGK}, $P_{F,\chi}$ is an object in $\calC_\chi(\lie{g}, \lie{k})$.

	By construction, $P_{F,\chi}$ is projective in $\calC_\chi(\lie{g}, \lie{k})$.
	In fact, the image of $F\otimes \calJ$ by any $\lie{g}$-homomorphism $\widetilde{P}_{F,\chi} \rightarrow V \in \calC_\chi(\lie{g}, \lie{k})$ should be zero.
	
	Hence any object in $\calC_\chi(\lie{g}, \lie{k})$ is isomorphic to a quotient of a finite direct sum of projective objects of the form $P_{F,\chi}$.
	It is enough to bound the Loewy length of $P_{F,\chi}$.
	By Proposition \ref{prop:UniformlyBoundedMinimalPrimitive} and Theorem \ref{thm:uniformlyBoundedLengthGeneral}, there is a constant $C$ independent of $\chi$ and $F$ such that
	\begin{align*}
		\Len_{\univ{g}\otimes \univ{g}^K}(P_{F,\chi}) \leq \Len_{\univ{g}\otimes \univ{g}^K}(\univ{g}/\calI_\chi \otimes_{\univ{k}}F) \leq C.
	\end{align*}
	By the following lemma, the Loewy length of $P_{F,\chi}$ as a $\lie{g}$-module is bounded by $C$.
\end{proof}

\begin{lemma}
	Let $\calA$ be a $\CC$-algebra and $V$ an irreducible $\univ{g}\otimes \calA$-module.
	If $V$ is in $\calC(\lie{g}, \lie{k})$ as a $\lie{g}$-module,
	then $V$ is completely reducible as a $\lie{g}$-module.
\end{lemma}

\begin{proof}
	Since $V$ has finite length as a $\lie{g}$-module, $V$ has an irreducible $\lie{g}$-submodule $W$.
	Since $V$ is irreducible as a $\univ{g}\otimes 
	\calA$-module, we have $V = \calA\cdot W$.
	This implies that $V$ is a sum of some copies of $W$, and hence $V$ is completely reducible as a $\lie{g}$-module.
\end{proof}

In the proof of Theorem \ref{thm:LoewyLength}, we have proved the following proposition.

\begin{proposition}\label{prop:EnoughProj}
	$\calC_\chi(\lie{g}, \lie{k})$ has enough projectives.
\end{proposition}

%% file: ref.bbl
\def\cprime{$'$} \def\cprime{$'$}